\documentclass{article}%
\usepackage{amssymb}
\usepackage{graphicx}
\usepackage{amsmath}%
\usepackage{amsfonts}
\providecommand{\U}[1]{\protect\rule{.1in}{.1in}}
\providecommand{\U}[1]{\protect\rule{.1in}{.1in}}
\newtheorem{theorem}{Theorem}

\newtheorem{corollary}[theorem]{Corollary}

\newtheorem{definition}[theorem]{Definition}

\newtheorem{lemma}[theorem]{Lemma}
\newtheorem{notation}[theorem]{Notation}

\newtheorem{proposition}[theorem]{Proposition}
\newtheorem{remark}[theorem]{Remark}

\newenvironment{proof}[1][Proof]{\textbf{#1.} }{\ \rule{0.5em}{0.5em}}
\begin{document}

\title{Axiomatic Differential Geometry II-4\\-Its Developments-\\Chapter 4: The Fr\"{o}licher-Nijenhuis Algebra}
\author{Hirokazu Nishimura\\Institute of Mathematics, University of Tsukuba\\Tsukuba, Ibaraki, 305-8571\\Japan}
\maketitle

\begin{abstract}
In our previous paper (Axiomatic Differential Geometry II-3) we have discussed
the general Jacobi identity, from which the Jacobi identity of vector fields
follows readily. In this paper we derive Jacobi-like identities of
tangent-vector-valued forms from the general Jacobi identity.

\end{abstract}

\section{Introduction}

In our previous paper \cite{nishi-k} we presented the general Jacobi identity,
from which the Jacobi identity of vector fields followed readily. The
principal objective in this paper is to show that the Jacobi-like identity of
tangent-vector-valued forms follows no less readily from the general Jacobi identity.

In orthodox differential geometry, establishing the Jacobi identity of vector
fields on a smooth manifold is a trifling exercise, because we are able to
identify vector fields with derivations. Similarly, since we are capable of
identifying tangent-vector-valued forms with derivations of a certain kind
over the algebra of differential form, the Jacobi-like identity of
tangent-vector-valued forms is essentially no less difficult, though somewhat
cumbersome, which Fr\"{o}licher and Nijenhuis did in the 1950's.

Within our general framework of axiomatic differential geometry, such luxury
is no longer permitted, so that the significance of the general Jacobi
identity could not be exaggerated. Generally speaking, tangent-vector-valued
forms are defined to be mappings subject to three conditions (the details will
be seen in Section \ref{s4}), namely, the first condition being what might be
called the Dirac condition after Dirac distributions, the second condition
being multi-linearity, and the third condition being anti-symmetry. Our
general approach enables us to discern three levels in which the Jacobi-like
identity holds, namely, tangent-vector-valued forms without multi-linearity or
anti-symmetry, those without multi-linearity, and those without anti-symmetry.
The case of tangent-vector-valued forms without multi-linearity or
anti-symmetry is most fundamental. By taking Jacobi-like identities of
tangent-vector-valued forms of the above three kinds at once, we get the very
Jacobi-like identity that Fr\"{o}licher and Nijenhuis provided more than half
a century ago.

This paper is organized as follows. After some preliminaries in Section
\ref{s2}, we review the general Jacobi identity in Section \ref{s3}. Section
\ref{s4} is devoted to two distinct viewpoints towards tangent-vector-valued
forms, just as we gave two viewpoints towards vector fields in . Section
\ref{s5} is concerned with Jacobi-like identities of tangent-vector-valued
forms. We are working within the axiomatics of differential geometry in
\cite{nishi-j}, namely, a DG-category
\[
\left(  \mathcal{K},\mathbb{R},\mathbf{T},\alpha\right)
\]

\section{\label{s2}Preliminaries}

\subsection{Weil Algebras and Infinitesimal Objects}

The notion of a \textit{Weil algebra} was introduced by Weil himself in
\cite{wei}. We denote by $\mathbf{W}$ the category of Weil algebras. Roughly
speaking, each Weil algebra corresponds to an infinitesimal object in the
shade. By way of example, the Weil algebra $\mathbb{R}[X]/(X^{2})$ (=the
quotient ring of the polynomial ring $\mathbb{R}[X]$\ of an indeterminate
$X$\ over $\mathbb{R}$ modulo the ideal $(X^{2})$\ generated by $X^{2}$)
corresponds to the infinitesimal object of first-order nilpotent
infinitesimals, while the Weil algebra $\mathbb{R}[X]/(X^{3})$ corresponds to
the infinitesimal object of second-order nilpotent infinitesimals. Although an
infinitesimal object is undoubtedly imaginary in the real world, as has
harassed both mathematicians and philosophers of the 17th and the 18th
centuries (because mathematicians at that time preferred to talk infinitesimal
objects as if they were real entities), each Weil algebra yields its
corresponding \textit{Weil functor} on the category of smooth manifolds of
some kind to itself, which is no doubt a real entity. By way of example, the
Weil algebra $\mathbb{R}[X]/(X^{2})$ yields the tangent bundle functor as its
corresponding Weil functor. Intuitively speaking, the Weil functor
corresponding to a Weil algebra stands for the exponentiation by the
infinitesimal object corresponding to the Weil algebra at issue. For Weil
functors on the category of finite-dimensional smooth manifolds, the reader is
referred to \S 35 of \cite{kolar}, while the reader can find a readable
treatment of Weil functors on the category of smooth manifolds modelled on
convenient vector spaces in \S 31 of \cite{kri}.

\textit{Synthetic differential geometry }(usually abbreviated to SDG), which
is a kind of differential geometry with a cornucopia of nilpotent
infinitesimals, was forced to invent its models, in which nilpotent
infinitesimals were visible. For a standard textbook on SDG, the reader is
referred to \cite{lav}, while he or she is referred to \cite{kock} for the
model theory of SDG constructed vigorously by Dubuc \cite{dub} and others.
Although we do not get involved in SDG herein, we will exploit locutions in
terms of infinitesimal objects so as to make the paper highly readable. Thus
we prefer to write $\mathcal{W}_{D}$\ and $\mathcal{W}_{D_{2}}$\ in place of
$\mathbb{R}[X]/(X^{2})$ and $\mathbb{R}[X]/(X^{3})$ respectively, where $D$
stands for the infinitesimal object of first-order nilpotent infinitesimals,
and $D_{2}$\ stands for the infinitesimal object of second-order nilpotent
infinitesimals. To Newton and Leibniz, $D$ stood for
\[
\{d\in\mathbb{R}\mid d^{2}=0\}
\]
while $D_{2}$\ stood for
\[
\{d\in\mathbb{R}\mid d^{3}=0\}
\]
We will write $\mathcal{W}_{d\in D_{2}\mapsto d^{2}\in D}$ for the homomorphim
of Weil algebras $\mathbb{R}[X]/(X^{2})\rightarrow\mathbb{R}[X]/(X^{3})$
induced by the homomorphism $X\rightarrow X^{2}$ of the polynomial ring
\ $\mathbb{R}[X]$ to itself. Such locutions are justifiable, because the
category $\mathbf{W}$ of Weil algebras in the real world and the category of
infinitesimal objects in the shade are dual to each other in a sense. Thus we
have a contravariant functor $\mathcal{W}$\ from the category of infinitesimal
objects in the shade to the category of Weil algebras in the real world. Its
inverse contravariant functor from the category of Weil algebras in the real
world to the category of Weil algebras in the real world is denoted by
$\mathcal{D}$. By way of example, $\mathcal{D}_{\mathbb{R}[X]/(X^{2})}$ and
$\mathcal{D}_{\mathbb{R}[X]/(X^{3})}$\ stand for $D$ and $D_{2}$%
\ respectively. To familiarize himself or herself with such locutions, the
reader is strongly encouraged to read the first two chapters of \cite{lav},
even if he or she is not interested in SDG at all.

We need to fix notation and terminology for simplicial objects, which form an
important subclass of infinitesimal objects. \textit{Simplicial objects} are
infinitesimal objects of the form
\begin{align*}
&  D^{n}\{\mathfrak{p}\}\\
&  =\{(d_{1},...,d_{n})\in D^{n}\mid d_{i_{1}}...d_{i_{k}}=0\text{ }%
(\forall(i_{1},...,i_{k})\in\mathfrak{p)\}}%
\end{align*}
where $\mathfrak{p}$\ is a finite set of finite sequences $(i_{1},...,i_{k}%
)$\ of natural numbers between $1$ and $n$, including the endpoints, with
$i_{1}<...<i_{k}$. If $\mathfrak{p}$\ is empty, $D^{n}\{\mathfrak{p}\}$ is
$D^{n}$ itself. If $\mathfrak{p}$ consists of all the binary sequences, then
$D^{n}\{\mathfrak{p}\}$ represents $D(n)$ in the standard terminology of SDG.
Given two simplicial objects $D^{m}\{\mathfrak{p}\}$\ and $D^{n}%
\{\mathfrak{q}\}$, we define a simplicial object $D^{m}\{\mathfrak{p}\}\oplus
D^{n}\{\mathfrak{q}\}$ to be
\[
D^{m+n}\{\mathfrak{p}\oplus\mathfrak{q}\}
\]
where
\begin{align*}
&  \mathfrak{p}\oplus\mathfrak{q}\\
&  =\mathfrak{p}\cup\{(j_{1}+m,...,j_{k}+m)\mid(j_{1},...,j_{k})\in
\mathfrak{q}\}\\
\cup\{(i,j+m) &  \mid1\leq i\leq m,\ 1\leq j\leq n\}
\end{align*}
Since the operation $\oplus$\ is associative, we can combine any finite number
of simplicial objects by $\oplus$ without bothering about how to insert
parentheses. Given morphisms of simplicial objects $\Phi_{i}:D^{m_{i}%
}\{\mathfrak{p}_{i}\}\rightarrow D^{m}\{\mathfrak{p}\}\ (1\leq i\leq n)$,
there exists a unique morphism of simplicial objects $\Phi:D^{m_{1}%
}\{\mathfrak{p}_{1}\}\oplus...\oplus D^{m_{n}}\{\mathfrak{p}_{n}\}\rightarrow
D^{m}\{\mathfrak{p}\}$ whose restriction to $D^{m_{i}}\{\mathfrak{p}_{i}\}$
coincides with $\Phi_{i}$ for each $i$. We denote this $\Phi$\ by $\Phi
_{1}\oplus...\oplus\Phi_{n}$. We write $D(n)$ for $\{(d,...,d)\in D^{n}\mid
d_{i}d_{j}=0$ for any $i\neq j\}$.

\subsection{Some Conventions}

\begin{notation}
We use the following notations:

\begin{enumerate}
\item We write
\[
\left[  A\rightarrow B\right]
\]
for the exponential $B^{A}$.

\item We denote a canonical injection $A\rightarrow B$ by $i_{A}^{B}$.

\item We denote a canonical projection $A\rightarrow B$ by $\pi_{B}^{A}$.

\item An object $M$\ is always assumed to be microlinear.

\item The evaluation morphism $\left[  A\rightarrow B\right]  \times
A\rightarrow B$ is denoted by
\[
\mathrm{ev}_{\left[  A\rightarrow B\right]  \times A}%
\]

\item We denote by $\mathbb{S}_{p}$\ the set of permutations of $\left\{
1,...,p\right\}  $. Given $\sigma\in\mathbb{S}_{p}$, its signature is denoted
by $\varepsilon_{\sigma}$.
\end{enumerate}
\end{notation}

\section{\label{s3}The General Jacobi Identity}

\begin{proposition}
\label{t3.1}The diagram
\[%
\begin{array}
[c]{ccccc}
&  & \alpha_{\mathcal{W}_{\varphi}}\left(  M\right)  &  & \\
& \mathbf{T}^{\mathcal{W}_{D^{3}\{(1,3),(2,3)\}}}M & \rightarrow &
\underset{2}{\mathbf{T}^{\mathcal{W}_{D^{2}}}M} & \\
\alpha_{\mathcal{W}_{\psi}}\left(  M\right)  & \downarrow &  & \downarrow &
\alpha_{\mathcal{W}_{i_{D(2)}^{D^{2}}}}\left(  M\right) \\
& \underset{1}{\mathbf{T}^{\mathcal{W}_{D^{2}}}M} & \rightarrow &
\mathbf{T}^{\mathcal{W}_{D(2)}}M & \\
&  & \alpha_{\mathcal{W}_{i_{D(2)}^{D^{2}}}}\left(  M\right)  &  &
\end{array}
\]
is a pullback diagram, where the assumptive mapping
\[
\varphi:D^{2}\rightarrow D^{3}\{(1,3),(2,3)\}
\]
is
\[
(d_{1},d_{2})\in D^{2}\mapsto(d_{1},d_{2},0)\in D^{3}\{(1,3),(2,3)\}
\]
while the assumptive mapping
\[
\psi:D^{2}\rightarrow D^{3}\{(1,3),(2,3)\}
\]
is
\[
(d_{1},d_{2})\in D^{2}\mapsto(d_{1},d_{2},d_{1}d_{2})\in D^{3}\{(1,3),(2,3)\}
\]

\end{proposition}

\begin{remark}
The numbers $1,2$\ under $\mathbf{T}^{\mathcal{W}_{D^{2}}}M$\ are given simply
so as for the reader to easily relate each occurrence of $\mathbf{T}%
^{\mathcal{W}_{D^{2}}}M$\ in the above Proposition to its corresponding
occurrence in the following Corollary.
\end{remark}

\begin{corollary}
\label{t3.1'}We have
\[
\mathbf{T}^{\mathcal{W}_{D^{3}\{(1,3),(2,3)\}}}M=\underset{1}{\mathbf{T}%
^{\mathcal{W}_{D^{2}}}M}\times_{\mathbf{T}^{\mathcal{W}_{D(2)}}M}\underset
{2}{\mathbf{T}^{\mathcal{W}_{D^{2}}}M}%
\]

\end{corollary}

\begin{notation}
We will write
\[
\zeta^{\overset{\cdot}{-}}\left(  M\right)  :\mathbf{T}^{\mathcal{W}_{D^{2}}%
}M\times_{\mathbf{T}^{\mathcal{W}_{D(2)}}M}\mathbf{T}^{\mathcal{W}_{D^{2}}%
}M\rightarrow\mathbf{T}^{\mathcal{W}_{D}}M
\]
for the morphism
\begin{align*}
& \mathbf{T}^{\mathcal{W}_{D^{2}}}M\times_{\mathbf{T}^{\mathcal{W}_{D(2)}}%
M}\mathbf{T}^{\mathcal{W}_{D^{2}}}M\\
& =\mathbf{T}^{\mathcal{W}_{D^{3}\{(1,3),(2,3)\}}}M\\
& \underrightarrow{\alpha_{\mathcal{W}_{d\in D\mapsto(0,0,d)\in D^{3}%
\{(1,3),(2,3)\}}}}\\
& \mathbf{T}^{\mathcal{W}_{D}}M
\end{align*}

\end{notation}

\begin{proposition}
\label{t3.2}The morphism
\begin{align*}
& \underset{1}{\mathbf{T}^{\mathcal{W}_{D^{2}}}M}\times_{\mathbf{T}%
^{\mathcal{W}_{D(2)}}M}\underset{2}{\mathbf{T}^{\mathcal{W}_{D^{2}}}M}\\
& \underrightarrow{\zeta^{\overset{\cdot}{-}}\left(  M\right)  }\\
& \mathbf{T}^{\mathcal{W}_{D}}M
\end{align*}
and the morphism
\begin{align*}
& \underset{1}{\mathbf{T}^{\mathcal{W}_{D^{2}}}M}\times_{\mathbf{T}%
^{\mathcal{W}_{D(2)}}M}\underset{2}{\mathbf{T}^{\mathcal{W}_{D^{2}}}M}\\
& =\underset{2}{\mathbf{T}^{\mathcal{W}_{D^{2}}}M}\times_{\mathbf{T}%
^{\mathcal{W}_{D(2)}}M}\underset{1}{\mathbf{T}^{\mathcal{W}_{D^{2}}}M}\\
& \underrightarrow{\zeta^{\overset{\cdot}{-}}\left(  M\right)  }\\
& \mathbf{T}^{\mathcal{W}_{D}}M
\end{align*}
sum up only to vanish, where the numbers $1,2$\ under $\mathbf{T}%
^{\mathcal{W}_{D^{2}}}M$\ are given simply so as for the reader to easily
relate each occurrence of $\mathbf{T}^{\mathcal{W}_{D^{2}}}M$\ to another.
\end{proposition}

\begin{proposition}
\label{t3.3}The diagram
\[%
\begin{array}
[c]{ccccc}
&  & \alpha_{\mathcal{W}_{\varphi_{1}^{3}}}\left(  M\right)  &  & \\
& \mathbf{T}^{\mathcal{W}_{D^{4}\{(2,4),(3,4)\}}}M & \rightarrow &
\underset{2}{\mathbf{T}^{\mathcal{W}_{D^{3}}}M} & \\
\alpha_{\mathcal{W}_{\psi_{1}^{3}}}\left(  M\right)  & \downarrow &  &
\downarrow & \alpha_{\mathcal{W}_{i_{D^{3}\{(2,3)\}}^{D^{3}}}}\left(  M\right)
\\
& \underset{1}{\mathbf{T}^{\mathcal{W}_{D^{3}}}M} & \rightarrow &
\mathbf{T}^{\mathcal{W}_{D^{3}\{(2,3)\}}}M & \\
&  & \alpha_{\mathcal{W}_{i_{D^{3}\{(2,3)\}}^{D^{3}}}}\left(  M\right)  &  &
\end{array}
\]
is a pullback diagram, where the assumptive mapping
\[
\varphi_{1}^{3}:D^{3}\rightarrow D^{4}\{(2,4),(3,4)\}
\]
is
\[
(d_{1},d_{2},d_{3})\in D^{3}\mapsto(d_{1},d_{2},d_{3},0)\in D^{4}%
\{(2,4),(3,4)\}\text{,}%
\]
while the assumptive mapping
\[
\psi_{1}^{3}:D^{3}\rightarrow D^{4}\{(2,4),(3,4)\}
\]
is
\[
(d_{1},d_{2},d_{3})\in D^{3}\mapsto(d_{1},d_{2},d_{3},d_{2}d_{3})\in
D^{4}\{(2,4),(3,4)\}
\]

\end{proposition}

\begin{remark}
The numbers $1,2$\ under $\mathbf{T}^{\mathcal{W}_{D^{3}}}M$\ are given simply
so as for the reader to easily relate each occurrence of $\mathbf{T}%
^{\mathcal{W}_{D^{3}}}M$\ in the above Proposition to its corresponding
occurrence in the following Corollary.
\end{remark}

\begin{corollary}
\label{t3.3'}We have
\begin{align*}
& \underset{1}{\mathbf{T}^{\mathcal{W}_{D^{3}}}M}\times_{\mathbf{T}%
^{\mathcal{W}_{D^{3}\{(2,3)\}}}M}\underset{2}{\mathbf{T}^{\mathcal{W}_{D^{3}}%
}M}\\
& =\mathbf{T}^{\mathcal{W}_{D^{4}\{(2,4),(3,4)\}}}M
\end{align*}

\end{corollary}

\begin{notation}
We will write
\[
\zeta^{\underset{1}{\overset{\cdot}{-}}}\left(  M\right)  :\mathbf{T}%
^{\mathcal{W}_{D^{3}}}M\times_{\mathbf{T}^{\mathcal{W}_{D^{3}\{(2,3)\}}}%
M}\mathbf{T}^{\mathcal{W}_{D^{3}}}M\rightarrow\mathbf{T}^{\mathcal{W}_{D^{2}}%
}M
\]
for the morphism
\begin{align*}
& \mathbf{T}^{\mathcal{W}_{D^{3}}}M\times_{\mathbf{T}^{\mathcal{W}%
_{D^{3}\{(2,3)\}}}M}\mathbf{T}^{\mathcal{W}_{D^{3}}}M\\
& =\mathbf{T}^{\mathcal{W}_{D^{4}\{(2,4),(3,4)\}}}M\\
& \underrightarrow{\alpha_{\mathcal{W}_{(d_{1},d_{2})\in D^{2}\mapsto
(d_{1},0,0,d_{2})\in D^{4}\{(2,4),(3,4)\}}}\left(  M\right)  }\\
& \mathbf{T}^{\mathcal{W}_{D^{2}}}M
\end{align*}

\end{notation}

\begin{proposition}
\label{t3.4}The diagram
\[%
\begin{array}
[c]{ccccc}
&  & \alpha_{\mathcal{W}_{\varphi_{2}^{3}}}\left(  M\right)  &  & \\
& \mathbf{T}^{\mathcal{W}_{D^{4}\{(1,4),(3,4)\}}}M & \rightarrow &
\underset{2}{\mathbf{T}^{\mathcal{W}_{D^{3}}}M} & \\
\alpha_{\mathcal{W}_{\psi_{2}^{3}}}\left(  M\right)  & \downarrow &  &
\downarrow & \alpha_{\mathcal{W}_{i_{D^{3}\{(1,3)\}}^{D^{3}}}}\left(  M\right)
\\
& \underset{1}{\mathbf{T}^{\mathcal{W}_{D^{3}}}M} & \rightarrow &
\mathbf{T}^{\mathcal{W}_{D^{3}\{(1,3)\}}}M & \\
&  & \alpha_{\mathcal{W}_{i_{D^{3}\{(1,3)\}}^{D^{3}}}}\left(  M\right)  &  &
\end{array}
\]
is a pullback diagram, where the assumptive mapping
\[
\varphi_{2}^{3}:D^{3}\rightarrow D^{4}\{(1,4),(3,4)\}
\]
is
\[
(d_{1},d_{2},d_{3})\in D^{3}\mapsto(d_{1},d_{2},d_{3},0)\in D^{4}%
\{(1,4),(3,4)\}
\]
while the assumptive mapping
\[
\psi_{2}^{3}:D^{3}\rightarrow D^{4}\{(1,4),(3,4)\}
\]
is
\[
(d_{1},d_{2},d_{3})\in D^{3}\mapsto(d_{1},d_{2},d_{3},d_{1}d_{3})\in
D^{4}\{(1,4),(3,4)\}
\]

\end{proposition}

\begin{remark}
The numbers $1,2$\ under $\mathbf{T}^{\mathcal{W}_{D^{3}}}M$\ are given simply
so as for the reader to easily relate each occurrence of $\mathbf{T}%
^{\mathcal{W}_{D^{3}}}M$\ in the above Proposition to its corresponding
occurrence in the following Corollary.
\end{remark}

\begin{corollary}
\label{t3.4'}We have
\begin{align*}
& \underset{1}{\mathbf{T}^{\mathcal{W}_{D^{3}}}M}\times_{\mathbf{T}%
^{\mathcal{W}_{D^{3}\{(1,3)\}}}M}\underset{2}{\mathbf{T}^{\mathcal{W}_{D^{3}}%
}M}\\
& =\mathbf{T}^{\mathcal{W}_{D^{4}\{(1,4),(3,4)\}}}M
\end{align*}

\end{corollary}

\begin{notation}
We will write
\[
\zeta^{\underset{2}{\overset{\cdot}{-}}}\left(  M\right)  :\mathbf{T}%
^{\mathcal{W}_{D^{3}}}M\times_{\mathbf{T}^{\mathcal{W}_{D^{3}\{(1,3)\}}}%
M}\mathbf{T}^{\mathcal{W}_{D^{3}}}M\rightarrow\mathbf{T}^{\mathcal{W}_{D^{2}}%
}M
\]
for the morphism
\begin{align*}
& \mathbf{T}^{\mathcal{W}_{D^{3}}}M\times_{\mathbf{T}^{\mathcal{W}%
_{D^{3}\{(1,3)\}}}M}\mathbf{T}^{\mathcal{W}_{D^{3}}}M\\
& =\mathbf{T}^{\mathcal{W}_{D^{4}\{(1,4),(3,4)\}}}M\\
& \underrightarrow{\alpha_{\mathcal{W}_{(d_{1},d_{2})\in D^{2}\mapsto
(0,d_{1},0,d_{2})\in D^{4}\{(1,4),(3,4)\}}}\left(  M\right)  }\\
& \mathbf{T}^{\mathcal{W}_{D^{2}}}M
\end{align*}

\end{notation}

\begin{proposition}
\label{t3.5}The diagram
\[%
\begin{array}
[c]{ccccc}
&  & \alpha_{\mathcal{W}_{\varphi_{3}^{3}}}\left(  M\right)  &  & \\
& \mathbf{T}^{\mathcal{W}_{D^{4}\{(1,4),(2,4)\}}}M & \rightarrow &
\underset{2}{\mathbf{T}^{\mathcal{W}_{D^{3}}}M} & \\
\alpha_{\mathcal{W}_{\psi_{3}^{3}}}\left(  M\right)  & \downarrow &  &
\downarrow & \alpha_{\mathcal{W}_{i_{D^{3}\{(1,2)\}}^{D^{3}}}}\left(  M\right)
\\
& \underset{1}{\mathbf{T}^{\mathcal{W}_{D^{3}}}M} & \rightarrow &
\mathbf{T}^{\mathcal{W}_{D^{3}\{(1,2)\}}}M & \\
&  & \alpha_{\mathcal{W}_{i_{D^{3}\{(1,2)\}}^{D^{3}}}}\left(  M\right)  &  &
\end{array}
\]
is a pullback diagram, where the assumptive mapping
\[
\varphi_{3}^{3}:D^{3}\rightarrow D^{4}\{(1,4),(2,4)\}
\]
is
\[
(d_{1},d_{2},d_{3})\in D^{3}\mapsto(d_{1},d_{2},d_{3},0)\in D^{4}%
\{(1,4),(2,4)\}
\]
while the assumptive mapping
\[
\psi_{3}^{3}:D^{3}\rightarrow D^{4}\{(1,4),(2,4)\}
\]
is
\[
(d_{1},d_{2},d_{3})\in D^{3}\mapsto(d_{1},d_{2},d_{3},d_{1}d_{2})\in
D^{4}\{(1,4),(2,4)\}
\]

\end{proposition}

\begin{remark}
The numbers $1,2$\ under $\mathbf{T}^{\mathcal{W}_{D^{3}}}M$\ are given simply
so as for the reader to easily relate each occurrence of $\mathbf{T}%
^{\mathcal{W}_{D^{3}}}M$\ in the above Proposition to its corresponding
occurrence in the following Corollary.
\end{remark}

\begin{corollary}
\label{t3.5'}We have
\begin{align*}
& \underset{1}{\mathbf{T}^{\mathcal{W}_{D^{3}}}M}\times_{\mathbf{T}%
^{\mathcal{W}_{D^{3}\{(1,2)\}}}M}\underset{2}{\mathbf{T}^{\mathcal{W}_{D^{3}}%
}M}\\
& =\mathbf{T}^{\mathcal{W}_{D^{4}\{(1,4),(2,4)\}}}M
\end{align*}

\end{corollary}

\begin{notation}
We will write
\[
\zeta^{\underset{3}{\overset{\cdot}{-}}}\left(  M\right)  :\mathbf{T}%
^{\mathcal{W}_{D^{3}}}M\times_{\mathbf{T}^{\mathcal{W}_{D^{3}\{(1,2)\}}}%
M}\mathbf{T}^{\mathcal{W}_{D^{3}}}M\rightarrow\mathbf{T}^{\mathcal{W}_{D^{2}}%
}M
\]
for the morphism
\begin{align*}
& \mathbf{T}^{\mathcal{W}_{D^{3}}}M\times_{\mathbf{T}^{\mathcal{W}%
_{D^{3}\{(1,2)\}}}M}\mathbf{T}^{\mathcal{W}_{D^{3}}}M\\
& =\mathbf{T}^{\mathcal{W}_{D^{4}\{(1,4),(2,4)\}}}M\\
& \alpha_{\mathcal{W}_{(d_{1},d_{2})\in D^{2}\mapsto(0,0,d_{1},d_{2})\in
D^{4}\{(1,4),(3,4)\}}}\left(  M\right) \\
& \mathbf{T}^{\mathcal{W}_{D^{2}}}M
\end{align*}

\end{notation}

\begin{notation}
We will introduce three notations.

\begin{enumerate}
\item We will write
\[
\zeta^{(\ast_{123}\overset{\cdot}{\underset{1}{-}}\ast_{132})\overset{\cdot
}{-}(\ast_{231}\overset{\cdot}{\underset{1}{-}}\ast_{321})}\left(  M\right)
:\triangle\left(  M\right)  \rightarrow\mathbf{T}^{\mathcal{W}_{D}}M
\]
for the composition of morphisms
\begin{align*}
& \pi_{\left(  \left(  \underset{321}{\mathbf{T}^{\mathcal{W}_{D^{3}}}M}%
\times_{\mathbf{T}^{\mathcal{W}_{D^{3}\{(2,3)\}}}M}\underset{231}%
{\mathbf{T}^{\mathcal{W}_{D^{3}}}M}\right)  \times_{\mathbf{T}^{\mathcal{W}%
_{D(2)}}M}\left(  \underset{132}{\mathbf{T}^{\mathcal{W}_{D^{3}}}M}%
\times_{\mathbf{T}^{\mathcal{W}_{D^{3}\{(2,3)\}}}M}\underset{123}%
{\mathbf{T}^{\mathcal{W}_{D^{3}}}M}\right)  \right)  }^{\triangle\left(
M\right)  }\\
& :\triangle\left(  M\right)  \rightarrow\left(
\begin{array}
[c]{c}%
\underset{321}{\mathbf{T}^{\mathcal{W}_{D^{3}}}M}\\
\times_{\mathbf{T}^{\mathcal{W}_{D^{3}\{(2,3)\}}}M}\\
\underset{231}{\mathbf{T}^{\mathcal{W}_{D^{3}}}M}%
\end{array}
\right)  \times_{\mathbf{T}^{\mathcal{W}_{D(2)}}M}\left(
\begin{array}
[c]{c}%
\underset{132}{\mathbf{T}^{\mathcal{W}_{D^{3}}}M}\\
\times_{\mathbf{T}^{\mathcal{W}_{D^{3}\{(2,3)\}}}M}\\
\underset{123}{\mathbf{T}^{\mathcal{W}_{D^{3}}}M}%
\end{array}
\right)
\end{align*}
\begin{align*}
& \zeta^{\overset{\cdot}{\underset{1}{-}}}\left(  M\right)  \times
_{\mathbf{T}^{\mathcal{W}_{D(2)}}M}\zeta^{\overset{\cdot}{\underset{1}{-}}%
}\left(  M\right)  :\\
& :\left(
\begin{array}
[c]{c}%
\underset{321}{\mathbf{T}^{\mathcal{W}_{D^{3}}}M}\\
\times_{\mathbf{T}^{\mathcal{W}_{D^{3}\{(2,3)\}}}M}\\
\underset{231}{\mathbf{T}^{\mathcal{W}_{D^{3}}}M}%
\end{array}
\right)  \times_{\mathbf{T}^{\mathcal{W}_{D(2)}}M}\left(
\begin{array}
[c]{c}%
\underset{132}{\mathbf{T}^{\mathcal{W}_{D^{3}}}M}\\
\times_{\mathbf{T}^{\mathcal{W}_{D^{3}\{(2,3)\}}}M}\\
\underset{123}{\mathbf{T}^{\mathcal{W}_{D^{3}}}M}%
\end{array}
\right) \\
& \rightarrow\mathbf{T}^{\mathcal{W}_{D^{2}}}M\times_{\mathbf{T}%
^{\mathcal{W}_{D(2)}}M}\mathbf{T}^{\mathcal{W}_{D^{2}}}M
\end{align*}
\[
\zeta^{\overset{\cdot}{-}}\left(  M\right)  :\mathbf{T}^{\mathcal{W}_{D^{2}}%
}M\times_{\mathbf{T}^{\mathcal{W}_{D(2)}}M}\mathbf{T}^{\mathcal{W}_{D^{2}}%
}M\rightarrow\mathbf{T}^{\mathcal{W}_{D}}M
\]
in succession, where $\triangle\left(  M\right)  $\ denotes
\[
\left[
\begin{array}
[c]{ccc}
&
\begin{array}
[c]{c}%
\left(
\begin{array}
[c]{c}%
\underset{321}{\mathbf{T}^{\mathcal{W}_{D^{3}}}M}\\
\times_{\mathbf{T}^{\mathcal{W}_{D^{3}\{(2,3)\}}}M}\\
\underset{231}{\mathbf{T}^{\mathcal{W}_{D^{3}}}M}%
\end{array}
\right) \\
\times_{\mathbf{T}^{\mathcal{W}_{D(2)}}M}\\
\left(
\begin{array}
[c]{c}%
\underset{132}{\mathbf{T}^{\mathcal{W}_{D^{3}}}M}\\
\times_{\mathbf{T}^{\mathcal{W}_{D^{3}\{(2,3)\}}}M}\\
\underset{123}{\mathbf{T}^{\mathcal{W}_{D^{3}}}M}%
\end{array}
\right)
\end{array}
& \\%
\begin{array}
[c]{c}%
\times_{\mathbf{T}^{\mathcal{W}_{D^{3}\oplus D^{3}}}M}\\
\,\\
\,
\end{array}
&  &
\begin{array}
[c]{c}%
\times_{\mathbf{T}^{\mathcal{W}_{D^{3}\oplus D^{3}}}M}\\
\,\\
\,
\end{array}
\\%
\begin{array}
[c]{c}%
\left(
\begin{array}
[c]{c}%
\underset{132}{\mathbf{T}^{\mathcal{W}_{D^{3}}}M}\\
\times_{\mathbf{T}^{\mathcal{W}_{D^{3}\{(1,3)\}}}M}\\
\underset{312}{\mathbf{T}^{\mathcal{W}_{D^{3}}}M}%
\end{array}
\right) \\
\times_{\mathbf{T}^{\mathcal{W}_{D(2)}}M}\\
\left(
\begin{array}
[c]{c}%
\underset{213}{\mathbf{T}^{\mathcal{W}_{D^{3}}}M}\\
\times_{\mathbf{T}^{\mathcal{W}_{D^{3}\{(1,3)\}}}M}\\
\underset{231}{\mathbf{T}^{\mathcal{W}_{D^{3}}}M}%
\end{array}
\right)
\end{array}
& \times_{\mathbf{T}^{\mathcal{W}_{D^{3}\oplus D^{3}}}M} &
\begin{array}
[c]{c}%
\left(
\begin{array}
[c]{c}%
\underset{213}{\mathbf{T}^{\mathcal{W}_{D^{3}}}M}\\
\times_{\mathbf{T}^{\mathcal{W}_{D^{3}\{(1,2)\}}}M}\\
\underset{123}{\mathbf{T}^{\mathcal{W}_{D^{3}}}M}%
\end{array}
\right) \\
\times_{\mathbf{T}^{\mathcal{W}_{D(2)}}M}\\
\left(
\begin{array}
[c]{c}%
\underset{321}{\mathbf{T}^{\mathcal{W}_{D^{3}}}M}\\
\times_{\mathbf{T}^{\mathcal{W}_{D^{3}\{(1,2)\}}}M}\\
\underset{312}{\mathbf{T}^{\mathcal{W}_{D^{3}}}M}%
\end{array}
\right)
\end{array}
\end{array}
\right]
\]

\item We will write the morphism
\[
\zeta^{(\ast_{231}\overset{\cdot}{\underset{2}{-}}\ast_{213})\overset{\cdot
}{-}(\ast_{312}\overset{\cdot}{\underset{2}{-}}\ast_{132})}\left(  M\right)
:\triangle\left(  M\right)  \rightarrow\mathbf{T}^{\mathcal{W}_{D}}M
\]
for the composition of morphisms
\begin{align*}
& \pi_{\left(  \left(  \underset{132}{\mathbf{T}^{\mathcal{W}_{D^{3}}}M}%
\times_{\mathbf{T}^{\mathcal{W}_{D^{3}\{(1,3)\}}}M}\underset{312}%
{\mathbf{T}^{\mathcal{W}_{D^{3}}}M}\right)  \times_{\mathbf{T}^{\mathcal{W}%
_{D(2)}}M}\left(  \underset{213}{\mathbf{T}^{\mathcal{W}_{D^{3}}}M}%
\times_{\mathbf{T}^{\mathcal{W}_{D^{3}\{(1,3)\}}}M}\underset{231}%
{\mathbf{T}^{\mathcal{W}_{D^{3}}}M}\right)  \right)  }^{\triangle\left(
M\right)  }\\
& :\triangle\left(  M\right)  \rightarrow\left(
\begin{array}
[c]{c}%
\underset{132}{\mathbf{T}^{\mathcal{W}_{D^{3}}}M}\\
\times_{\mathbf{T}^{\mathcal{W}_{D^{3}\{(1,3)\}}}M}\\
\underset{312}{\mathbf{T}^{\mathcal{W}_{D^{3}}}M}%
\end{array}
\right)  \times_{\mathbf{T}^{\mathcal{W}_{D(2)}}M}\left(
\begin{array}
[c]{c}%
\underset{213}{\mathbf{T}^{\mathcal{W}_{D^{3}}}M}\\
\times_{\mathbf{T}^{\mathcal{W}_{D^{3}\{(1,3)\}}}M}\\
\underset{231}{\mathbf{T}^{\mathcal{W}_{D^{3}}}M}%
\end{array}
\right)
\end{align*}
\begin{align*}
& \zeta^{\overset{\cdot}{\underset{2}{-}}}\left(  M\right)  \times
_{\mathbf{T}^{\mathcal{W}_{D(2)}}M}\zeta^{\overset{\cdot}{\underset{2}{-}}%
}\left(  M\right)  :\\
& :\left(
\begin{array}
[c]{c}%
\underset{132}{\mathbf{T}^{\mathcal{W}_{D^{3}}}M}\\
\times_{\mathbf{T}^{\mathcal{W}_{D^{3}\{(1,3)\}}}M}\\
\underset{312}{\mathbf{T}^{\mathcal{W}_{D^{3}}}M}%
\end{array}
\right)  \times_{\mathbf{T}^{\mathcal{W}_{D(2)}}M}\left(
\begin{array}
[c]{c}%
\underset{213}{\mathbf{T}^{\mathcal{W}_{D^{3}}}M}\\
\times_{\mathbf{T}^{\mathcal{W}_{D^{3}\{(1,3)\}}}M}\\
\underset{231}{\mathbf{T}^{\mathcal{W}_{D^{3}}}M}%
\end{array}
\right) \\
& \rightarrow\mathbf{T}^{\mathcal{W}_{D^{2}}}M\times_{\mathbf{T}%
^{\mathcal{W}_{D(2)}}M}\mathbf{T}^{\mathcal{W}_{D^{2}}}M
\end{align*}
\[
\zeta^{\overset{\cdot}{-}}\left(  M\right)  :\mathbf{T}^{\mathcal{W}_{D^{2}}%
}M\times_{\mathbf{T}^{\mathcal{W}_{D(2)}}M}\mathbf{T}^{\mathcal{W}_{D^{2}}%
}M\rightarrow\mathbf{T}^{\mathcal{W}_{D}}M
\]
in succession.

\item We will write the morphism
\[
\zeta^{(\ast_{312}\overset{\cdot}{\underset{3}{-}}\ast_{321})\overset{\cdot
}{-}(\ast_{123}\overset{\cdot}{\underset{3}{-}}\ast_{213})}\left(  M\right)
:\triangle\left(  M\right)  \rightarrow\mathbf{T}^{\mathcal{W}_{D}}M
\]
for the composition of morphisms
\begin{align*}
& \pi_{\left(  \left(  \underset{213}{\mathbf{T}^{\mathcal{W}_{D^{3}}}M}%
\times_{\mathbf{T}^{\mathcal{W}_{D^{3}\{(1,2)\}}}M}\underset{123}%
{\mathbf{T}^{\mathcal{W}_{D^{3}}}M}\right)  \times_{\mathbf{T}^{\mathcal{W}%
_{D(2)}}M}\left(  \underset{321}{\mathbf{T}^{\mathcal{W}_{D^{3}}}M}%
\times_{\mathbf{T}^{\mathcal{W}_{D^{3}\{(1,2)\}}}M}\underset{312}%
{\mathbf{T}^{\mathcal{W}_{D^{3}}}M}\right)  \right)  }^{\triangle\left(
M\right)  }\\
& :\triangle\left(  M\right)  \rightarrow\left(
\begin{array}
[c]{c}%
\underset{213}{\mathbf{T}^{\mathcal{W}_{D^{3}}}M}\\
\times_{\mathbf{T}^{\mathcal{W}_{D^{3}\{(1,2)\}}}M}\\
\underset{123}{\mathbf{T}^{\mathcal{W}_{D^{3}}}M}%
\end{array}
\right)  \times_{\mathbf{T}^{\mathcal{W}_{D(2)}}M}\left(
\begin{array}
[c]{c}%
\underset{321}{\mathbf{T}^{\mathcal{W}_{D^{3}}}M}\\
\times_{\mathbf{T}^{\mathcal{W}_{D^{3}\{(1,2)\}}}M}\\
\underset{312}{\mathbf{T}^{\mathcal{W}_{D^{3}}}M}%
\end{array}
\right)
\end{align*}
\begin{align*}
& \zeta^{\overset{\cdot}{\underset{3}{-}}}\left(  M\right)  \times
_{\mathbf{T}^{\mathcal{W}_{D(2)}}M}\zeta^{\overset{\cdot}{\underset{3}{-}}%
}\left(  M\right) \\
& :\left(
\begin{array}
[c]{c}%
\underset{213}{\mathbf{T}^{\mathcal{W}_{D^{3}}}M}\\
\times_{\mathbf{T}^{\mathcal{W}_{D^{3}\{(1,2)\}}}M}\\
\underset{123}{\mathbf{T}^{\mathcal{W}_{D^{3}}}M}%
\end{array}
\right)  \times_{\mathbf{T}^{\mathcal{W}_{D(2)}}M}\left(
\begin{array}
[c]{c}%
\underset{321}{\mathbf{T}^{\mathcal{W}_{D^{3}}}M}\\
\times_{\mathbf{T}^{\mathcal{W}_{D^{3}\{(1,2)\}}}M}\\
\underset{312}{\mathbf{T}^{\mathcal{W}_{D^{3}}}M}%
\end{array}
\right) \\
& \rightarrow\mathbf{T}^{\mathcal{W}_{D^{2}}}M\times_{\mathbf{T}%
^{\mathcal{W}_{D(2)}}M}\mathbf{T}^{\mathcal{W}_{D^{2}}}M
\end{align*}
\[
\zeta^{\overset{\cdot}{-}}\left(  M\right)  :\mathbf{T}^{\mathcal{W}_{D^{2}}%
}M\times_{\mathbf{T}^{\mathcal{W}_{D(2)}}M}\mathbf{T}^{\mathcal{W}_{D^{2}}%
}M\rightarrow\mathbf{T}^{\mathcal{W}_{D}}M
\]
in succession.
\end{enumerate}
\end{notation}

\begin{theorem}
\label{t3.6}(\underline{The general Jacobi Identity}) The three morphisms
\[
\zeta^{(\ast_{123}\overset{\cdot}{\underset{1}{-}}\ast_{132})\overset{\cdot
}{-}(\ast_{231}\overset{\cdot}{\underset{1}{-}}\ast_{321})}\left(  M\right)
:\triangle\left(  M\right)  \rightarrow\mathbf{T}^{\mathcal{W}_{D}}M
\]
\[
\zeta^{(\ast_{231}\overset{\cdot}{\underset{2}{-}}\ast_{213})\overset{\cdot
}{-}(\ast_{312}\overset{\cdot}{\underset{2}{-}}\ast_{132})}\left(  M\right)
:\triangle\left(  M\right)  \rightarrow\mathbf{T}^{\mathcal{W}_{D}}M
\]
\[
\zeta^{(\ast_{312}\overset{\cdot}{\underset{3}{-}}\ast_{321})\overset{\cdot
}{-}(\ast_{123}\overset{\cdot}{\underset{3}{-}}\ast_{213})}\left(  M\right)
:\triangle\left(  M\right)  \rightarrow\mathbf{T}^{\mathcal{W}_{D}}M
\]
sum up only to vanish.
\end{theorem}

\section{\label{s4}Tangent-Vector-Valued Forms}

\begin{notation}
We write
\[
\underline{\Omega}_{\left(  0\right)  }^{\left(  p,1\right)  }\left(
M\right)
\]
for
\[
\left[  \mathbf{T}^{\mathcal{W}_{D^{p}}}M\rightarrow\mathbf{T}^{\mathcal{W}%
_{D}}M\right]  \text{,}%
\]
while we write
\[
\overline{\Omega}_{\left(  0\right)  }^{\left(  p,1\right)  }\left(  M\right)
\]
for
\[
\mathbf{T}^{\mathcal{W}_{D}}\left[  \mathbf{T}^{\mathcal{W}_{D^{p}}%
}M\rightarrow M\right]
\]

\end{notation}

It is easy to see that

\begin{proposition}
\label{t4.1}We have
\[
\underline{\Omega}_{\left(  0\right)  }^{\left(  p,1\right)  }\left(
M\right)  =\overline{\Omega}_{\left(  0\right)  }^{\left(  p,1\right)
}\left(  M\right)
\]

\end{proposition}

\begin{notation}
We write
\[
\underline{\Omega}_{\left(  1\right)  }^{\left(  p,1\right)  }\left(
M\right)
\]
for the equalizer of the exponential transpose of
\begin{align*}
& \left[  \mathbf{T}^{\mathcal{W}_{D^{p}}}M\rightarrow\mathbf{T}%
^{\mathcal{W}_{D}}M\right]  \times\mathbf{T}^{\mathcal{W}_{D^{p}}}M\\
& \underrightarrow{\mathrm{ev}_{\left[  \mathbf{T}^{\mathcal{W}_{D^{p}}%
}M\rightarrow\mathbf{T}^{\mathcal{W}_{D}}M\right]  \times\mathbf{T}%
^{\mathcal{W}_{D^{p}}}M}}\\
& \mathbf{T}^{\mathcal{W}_{D}}M\\
& \underrightarrow{\alpha_{\mathcal{W}_{1\rightarrow D}}}\\
& M
\end{align*}
and that of
\begin{align*}
& \left[  \mathbf{T}^{\mathcal{W}_{D^{p}}}M\rightarrow\mathbf{T}%
^{\mathcal{W}_{D}}M\right]  \times\mathbf{T}^{\mathcal{W}_{D^{p}}}M\\
& \underrightarrow{\pi_{\mathbf{T}^{\mathcal{W}_{D^{p}}}M}^{\left[
\mathbf{T}^{\mathcal{W}_{D^{p}}}M\rightarrow\mathbf{T}^{\mathcal{W}_{D}%
}M\right]  \times\mathbf{T}^{\mathcal{W}_{D^{p}}}M}}\\
& \mathbf{T}^{\mathcal{W}_{D^{p}}}M\\
& \underrightarrow{\alpha_{\mathcal{W}_{1\rightarrow D}}}\\
& M
\end{align*}
while we write
\[
\overline{\Omega}_{\left(  1\right)  }^{\left(  p,1\right)  }\left(  M\right)
\]
for the equalizer of the exponential transpose of
\begin{align*}
& \mathbf{T}^{\mathcal{W}_{D}}\left[  \mathbf{T}^{\mathcal{W}_{D^{p}}%
}M\rightarrow M\right]  \times\mathbf{T}^{\mathcal{W}_{D^{p}}}M\\
& \underrightarrow{\alpha_{\mathcal{W}_{1\rightarrow D}}\left(  \left[
\mathbf{T}^{\mathcal{W}_{D^{p}}}M\rightarrow M\right]  \right)  }\\
& \left[  \mathbf{T}^{\mathcal{W}_{D^{p}}}M\rightarrow M\right]
\times\mathbf{T}^{\mathcal{W}_{D^{p}}}M\\
& \underrightarrow{\mathrm{ev}_{\left[  \mathbf{T}^{\mathcal{W}_{D^{p}}%
}M\rightarrow M\right]  \times\mathbf{T}^{\mathcal{W}_{D^{p}}}M}}\\
& M
\end{align*}
and that of
\begin{align*}
& \mathbf{T}^{\mathcal{W}_{D}}\left[  \mathbf{T}^{\mathcal{W}_{D^{p}}%
}M\rightarrow M\right]  \times\mathbf{T}^{\mathcal{W}_{D^{p}}}M\\
& \underrightarrow{\pi_{\mathbf{T}^{\mathcal{W}_{D^{p}}}M}^{\mathbf{T}%
^{\mathcal{W}_{D}}\left[  \mathbf{T}^{\mathcal{W}_{D^{p}}}M\rightarrow
M\right]  \times\mathbf{T}^{\mathcal{W}_{D^{p}}}M}}\\
& \mathbf{T}^{\mathcal{W}_{D^{p}}}M\\
& \underrightarrow{\alpha_{\mathcal{W}_{1\rightarrow D^{p}}}\left(  M\right)
}\\
& M
\end{align*}

\end{notation}

It is easy to see that

\begin{proposition}
\label{t4.2}We have
\[
\underline{\Omega}_{\left(  1\right)  }^{\left(  p,1\right)  }\left(
M\right)  =\overline{\Omega}_{\left(  1\right)  }^{\left(  p,1\right)
}\left(  M\right)
\]

\end{proposition}

\begin{notation}
We write
\[
\underline{\Omega}_{\left(  12\right)  }^{\left(  p,1\right)  }\left(
M\right)
\]
for the intersection of $\underline{\Omega}_{\left(  1\right)  }^{\left(
p,1\right)  }\left(  M\right)  $\ and the equalizer of the exponential
transpose of
\begin{align*}
& \left[  \mathbf{T}^{\mathcal{W}_{D^{p}}}M\rightarrow\mathbf{T}%
^{\mathcal{W}_{D}}M\right]  \times\mathbf{T}^{\mathcal{W}_{D^{p}}}%
M\times\mathbb{R}\\
& \underrightarrow{\mathrm{id}_{\left[  \mathbf{T}^{\mathcal{W}_{D^{p}}%
}M\rightarrow\mathbf{T}^{\mathcal{W}_{D}}M\right]  }\times\left(  \underset
{i}{\cdot}\right)  _{\mathbf{T}^{\mathcal{W}_{D^{p}}}M\times\mathbb{R}}}\\
& \left[  \mathbf{T}^{\mathcal{W}_{D^{p}}}M\rightarrow\mathbf{T}%
^{\mathcal{W}_{D}}M\right]  \times\mathbf{T}^{\mathcal{W}_{D^{p}}}M\\
& \underrightarrow{\mathrm{ev}_{\left[  \mathbf{T}^{\mathcal{W}_{D^{p}}%
}M\rightarrow\mathbf{T}^{\mathcal{W}_{D}}M\right]  \times\mathbf{T}%
^{\mathcal{W}_{D^{p}}}M}}\\
& \mathbf{T}^{\mathcal{W}_{D}}M
\end{align*}
and that of
\begin{align*}
& \left[  \mathbf{T}^{\mathcal{W}_{D^{p}}}M\rightarrow\mathbf{T}%
^{\mathcal{W}_{D}}M\right]  \times\mathbf{T}^{\mathcal{W}_{D^{p}}}%
M\times\mathbb{R}\\
& \underrightarrow{\mathrm{ev}_{\left[  \mathbf{T}^{\mathcal{W}_{D^{p}}%
}M\rightarrow\mathbf{T}^{\mathcal{W}_{D}}M\right]  \times\mathbf{T}%
^{\mathcal{W}_{D^{p}}}M}\times\mathrm{id}_{\mathbb{R}}}\\
& \mathbf{T}^{\mathcal{W}_{D}}M\times\mathbb{R}\\
& \underrightarrow{\left(  \cdot\right)  _{\mathbf{T}^{\mathcal{W}_{D}}%
M\times\mathbb{R}}}\\
& \mathbf{T}^{\mathcal{W}_{D}}M
\end{align*}
where

\begin{enumerate}
\item The morphism
\[
\left(  \underset{i}{\cdot}\right)  _{\mathbf{T}^{\mathcal{W}_{D^{p}}}%
M\times\mathbb{R}}:\mathbf{T}^{\mathcal{W}_{D^{p}}}M\times\mathbb{R}%
\rightarrow\mathbf{T}^{\mathcal{W}_{D^{p}}}M
\]
\ stands for the morphism whose exponential transpose is
\begin{align*}
& \mathbf{T}^{\mathcal{W}_{D^{p}}}M\\
& \underrightarrow{\alpha_{\mathcal{W}_{\left(  d_{1},...,d_{p},a\right)  \in
D^{p}\times\mathbb{R}\rightarrow\left(  d_{1},...,ad_{i},...,d_{p}\right)  \in
D^{p}}}\left(  M\right)  }\\
& \mathbf{T}^{\mathcal{W}_{D^{p}\times\mathbb{R}}}M\\
& =\left[  \mathbb{R}\rightarrow\mathbf{T}^{\mathcal{W}_{D^{p}}}M\right]
\end{align*}

\item The morphism
\[
\left(  \cdot\right)  _{\mathbf{T}^{\mathcal{W}_{D}}M\times\mathbb{R}%
}:\mathbf{T}^{\mathcal{W}_{D}}M\times\mathbb{R}\rightarrow\mathbf{T}%
^{\mathcal{W}_{D}}M
\]
\ stands for the morphism whose exponential transpose is
\begin{align*}
& \mathbf{T}^{\mathcal{W}_{D}}M\\
& \underrightarrow{\alpha_{\mathcal{W}_{\left(  d,a\right)  \in D\times
\mathbb{R}\rightarrow ad\in D}}\left(  M\right)  }\\
& \mathbf{T}^{\mathcal{W}_{D\times\mathbb{R}}}M\\
& =\left[  \mathbb{R}\rightarrow\mathbf{T}^{\mathcal{W}_{D}}M\right]
\end{align*}

\end{enumerate}
\end{notation}

\begin{notation}
We write
\[
\overline{\Omega}_{\left(  12\right)  }^{\left(  p,1\right)  }\left(
M\right)
\]
for the intersection of $\overline{\Omega}_{\left(  1\right)  }^{\left(
p,1\right)  }\left(  M\right)  $\ and the equalizer of the exponential
transpose of
\begin{align*}
& \mathbf{T}^{\mathcal{W}_{D}}\left[  \mathbf{T}^{\mathcal{W}_{D^{p}}%
}M\rightarrow M\right]  \times\mathbf{T}^{\mathcal{W}_{D^{p}}}M\times
\mathbb{R}\\
& \underrightarrow{\mathrm{id}_{\mathbf{T}^{\mathcal{W}_{D}}\left[
\mathbf{T}^{\mathcal{W}_{D^{p}}}M\rightarrow M\right]  }\times\left(
\underset{i}{\cdot}\right)  _{\mathbf{T}^{\mathcal{W}_{D^{p}}}M\times
\mathbb{R}}}\\
& \mathbf{T}^{\mathcal{W}_{D}}\left[  \mathbf{T}^{\mathcal{W}_{D^{p}}%
}M\rightarrow M\right]  \times\mathbf{T}^{\mathcal{W}_{D^{p}}}M\\
& \underrightarrow{\mathrm{id}_{\mathbf{T}^{\mathcal{W}_{D}}\left[
\mathbf{T}^{\mathcal{W}_{D^{p}}}M\rightarrow M\right]  }\times\alpha
_{\mathcal{W}_{1\rightarrow D}}\left(  \mathbf{T}^{\mathcal{W}_{D^{p}}%
}M\right)  }\\
& \mathbf{T}^{\mathcal{W}_{D}}\left[  \mathbf{T}^{\mathcal{W}_{D^{p}}%
}M\rightarrow M\right]  \times\mathbf{T}^{\mathcal{W}_{D}}\left(
\mathbf{T}^{\mathcal{W}_{D^{p}}}M\right) \\
& =\mathbf{T}^{\mathcal{W}_{D}}\left(  \left[  \mathbf{T}^{\mathcal{W}_{D^{p}%
}}M\rightarrow M\right]  \times\mathbf{T}^{\mathcal{W}_{D^{p}}}M\right) \\
& \underrightarrow{\mathbf{T}^{\mathcal{W}_{D}}\mathrm{ev}_{\left[
\mathbf{T}^{\mathcal{W}_{D^{p}}}M\rightarrow M\right]  \times\mathbf{T}%
^{\mathcal{W}_{D^{p}}}M}}\\
& \mathbf{T}^{\mathcal{W}_{D}}M
\end{align*}
and that of
\begin{align*}
& \mathbf{T}^{\mathcal{W}_{D}}\left[  \mathbf{T}^{\mathcal{W}_{D^{p}}%
}M\rightarrow M\right]  \times\mathbf{T}^{\mathcal{W}_{D^{p}}}M\times
\mathbb{R}\\
& \underrightarrow{\mathrm{id}_{\mathbf{T}^{\mathcal{W}_{D}}\left[
\mathbf{T}^{\mathcal{W}_{D^{p}}}M\rightarrow M\right]  }\times\alpha
_{\mathcal{W}_{1\rightarrow D}}\left(  \mathbf{T}^{\mathcal{W}_{D^{p}}%
}M\right)  \times\mathrm{id}_{\mathbb{R}}}\\
& \mathbf{T}^{\mathcal{W}_{D}}\left[  \mathbf{T}^{\mathcal{W}_{D^{p}}%
}M\rightarrow M\right]  \times\mathbf{T}^{\mathcal{W}_{D}}\left(
\mathbf{T}^{\mathcal{W}_{D^{p}}}M\right)  \times\mathbb{R}\\
& \underrightarrow{\mathbf{T}^{\mathcal{W}_{D}}\mathrm{ev}_{M}^{\left[
\mathbf{T}^{\mathcal{W}_{D^{p}}}M\rightarrow M\right]  \times\mathbf{T}%
^{\mathcal{W}_{D^{p}}}M}\times\mathrm{id}_{\mathbb{R}}}\\
& \mathbf{T}^{\mathcal{W}_{D}}M\times\mathrm{id}_{\mathbb{R}}\\
& \underrightarrow{\left(  \cdot\right)  _{\mathbf{T}^{\mathcal{W}_{D}}%
M\times\mathbb{R}}}\\
& \mathbf{T}^{\mathcal{W}_{D}}M
\end{align*}

\end{notation}

It is easy to see that

\begin{proposition}
\label{t4.3}We have
\[
\underline{\Omega}_{\left(  12\right)  }^{\left(  p,1\right)  }\left(
M\right)  =\overline{\Omega}_{\left(  12\right)  }^{\left(  p,1\right)
}\left(  M\right)
\]

\end{proposition}

\begin{notation}
We write
\[
\underline{\Omega}_{\left(  13\right)  }^{\left(  p,1\right)  }\left(
M\right)
\]
for the intersection of $\underline{\Omega}_{\left(  1\right)  }^{\left(
p,1\right)  }\left(  M\right)  $\ and the equalizers, with all $\sigma
\in\mathbb{S}_{p}$, of the exponential transpose of
\begin{align*}
& \left[  \mathbf{T}^{\mathcal{W}_{D^{p}}}M\rightarrow\mathbf{T}%
^{\mathcal{W}_{D}}M\right]  \times\mathbf{T}^{\mathcal{W}_{D^{p}}}M\\
& \underrightarrow{\mathrm{id}_{\left[  \mathbf{T}^{\mathcal{W}_{D^{p}}%
}M\rightarrow\mathbf{T}^{\mathcal{W}_{D}}M\right]  }\times\left(
\cdot^{\sigma}\right)  _{\mathbf{T}^{\mathcal{W}_{D^{p}}}M}}\\
& \left[  \mathbf{T}^{\mathcal{W}_{D^{p}}}M\rightarrow\mathbf{T}%
^{\mathcal{W}_{D}}M\right]  \times\mathbf{T}^{\mathcal{W}_{D^{p}}}M\\
& \underrightarrow{\mathrm{ev}_{\left[  \mathbf{T}^{\mathcal{W}_{D^{p}}%
}M\rightarrow\mathbf{T}^{\mathcal{W}_{D}}M\right]  \times\mathbf{T}%
^{\mathcal{W}_{D^{p}}}M}}\\
& \mathbf{T}^{\mathcal{W}_{D}}M
\end{align*}
and that of
\begin{align*}
& \left[  \mathbf{T}^{\mathcal{W}_{D^{p}}}M\rightarrow\mathbf{T}%
^{\mathcal{W}_{D}}M\right]  \times\mathbf{T}^{\mathcal{W}_{D^{p}}}M\\
& \underrightarrow{\mathrm{ev}_{\left[  \mathbf{T}^{\mathcal{W}_{D^{p}}%
}M\rightarrow\mathbf{T}^{\mathcal{W}_{D}}M\right]  \times\mathbf{T}%
^{\mathcal{W}_{D^{p}}}M}}\\
& \mathbf{T}^{\mathcal{W}_{D}}M\\
& \underrightarrow{\varepsilon_{\sigma}}\\
& \mathbf{T}^{\mathcal{W}_{D}}M
\end{align*}
where the morphism
\[
\left(  \cdot^{\sigma}\right)  _{\mathbf{T}^{\mathcal{W}_{D^{p}}}M}%
:\mathbf{T}^{\mathcal{W}_{D^{p}}}M\rightarrow\mathbf{T}^{\mathcal{W}_{D^{p}}}M
\]
stands for the morphism
\begin{align*}
& \mathbf{T}^{\mathcal{W}_{D^{p}}}M\\
& \underrightarrow{\alpha_{\mathcal{W}_{\left(  d_{1},...,d_{p}\right)  \in
D^{p}\mapsto\left(  d_{\sigma\left(  1\right)  },...,d_{\sigma\left(
p\right)  }\right)  \in D^{p}}}\left(  M\right)  }\\
& \mathbf{T}^{\mathcal{W}_{D^{p}}}M
\end{align*}
We write
\[
\overline{\Omega}_{\left(  13\right)  }^{\left(  p,1\right)  }\left(
M\right)
\]
for the intersection of $\underline{\Omega}_{\left(  1\right)  }^{\left(
p,1\right)  }\left(  M\right)  $\ and the equalizers, with all $\sigma
\in\mathbb{S}_{p}$, of the exponential transpose of
\begin{align*}
& \mathbf{T}^{\mathcal{W}_{D}}\left[  \mathbf{T}^{\mathcal{W}_{D^{p}}%
}M\rightarrow M\right]  \times\mathbf{T}^{\mathcal{W}_{D^{p}}}M\\
& \underrightarrow{\mathrm{id}_{\mathbf{T}^{\mathcal{W}_{D}}\left[
\mathbf{T}^{\mathcal{W}_{D^{p}}}M\rightarrow M\right]  }\times\left(
\cdot^{\sigma}\right)  _{\mathbf{T}^{\mathcal{W}_{D^{p}}}M}}\\
& \mathbf{T}^{\mathcal{W}_{D}}\left[  \mathbf{T}^{\mathcal{W}_{D^{p}}%
}M\rightarrow M\right]  \times\mathbf{T}^{\mathcal{W}_{D^{p}}}M\\
& \underrightarrow{\mathrm{id}_{\mathbf{T}^{\mathcal{W}_{D}}\left[
\mathbf{T}^{\mathcal{W}_{D^{p}}}M\rightarrow M\right]  }\times\alpha
_{\mathcal{W}_{1\rightarrow D}}\left(  \mathbf{T}^{\mathcal{W}_{D^{p}}%
}M\right)  }\\
& \mathbf{T}^{\mathcal{W}_{D}}\left[  \mathbf{T}^{\mathcal{W}_{D^{p}}%
}M\rightarrow M\right]  \times\mathbf{T}^{\mathcal{W}_{D}}\left(
\mathbf{T}^{\mathcal{W}_{D^{p}}}M\right) \\
& =\mathbf{T}^{\mathcal{W}_{D}}\left(  \left[  \mathbf{T}^{\mathcal{W}_{D^{p}%
}}M\rightarrow M\right]  \times\mathbf{T}^{\mathcal{W}_{D^{p}}}M\right) \\
& \underrightarrow{\mathbf{T}^{\mathcal{W}_{D}}\mathrm{ev}_{\left[
\mathbf{T}^{\mathcal{W}_{D^{p}}}M\rightarrow M\right]  \times\mathbf{T}%
^{\mathcal{W}_{D^{p}}}M}}\\
& \mathbf{T}^{\mathcal{W}_{D}}M
\end{align*}
and that of
\begin{align*}
& \mathbf{T}^{\mathcal{W}_{D}}\left[  \mathbf{T}^{\mathcal{W}_{D^{p}}%
}M\rightarrow M\right]  \times\mathbf{T}^{\mathcal{W}_{D^{p}}}M\\
& \underrightarrow{\mathrm{id}_{\mathbf{T}^{\mathcal{W}_{D}}\left[
\mathbf{T}^{\mathcal{W}_{D^{p}}}M\rightarrow M\right]  }\times\alpha
_{\mathcal{W}_{1\rightarrow D}}\left(  \mathbf{T}^{\mathcal{W}_{D^{p}}%
}M\right)  }\\
& \mathbf{T}^{\mathcal{W}_{D}}\left[  \mathbf{T}^{\mathcal{W}_{D^{p}}%
}M\rightarrow M\right]  \times\mathbf{T}^{\mathcal{W}_{D}}\left(
\mathbf{T}^{\mathcal{W}_{D^{p}}}M\right) \\
& \underrightarrow{\mathbf{T}^{\mathcal{W}_{D}}\mathrm{ev}_{\left[
\mathbf{T}^{\mathcal{W}_{D^{p}}}M\rightarrow M\right]  \times\mathbf{T}%
^{\mathcal{W}_{D^{p}}}M}}\\
& \mathbf{T}^{\mathcal{W}_{D}}M\\
& \underrightarrow{\varepsilon_{\sigma}}\\
& \mathbf{T}^{\mathcal{W}_{D}}M
\end{align*}

\end{notation}

It is easy to see that

\begin{proposition}
\label{t4.4}We have
\[
\underline{\Omega}_{\left(  13\right)  }^{\left(  p,1\right)  }\left(
M\right)  =\overline{\Omega}_{\left(  13\right)  }^{\left(  p,1\right)
}\left(  M\right)
\]

\end{proposition}

\begin{notation}
We write
\[
\underline{\Omega}_{\left(  123\right)  }^{\left(  p,1\right)  }\left(
M\right)
\]
for the intersection of $\underline{\Omega}_{\left(  12\right)  }^{\left(
p,1\right)  }\left(  M\right)  $\ and $\underline{\Omega}_{\left(  13\right)
}^{\left(  p,1\right)  }\left(  M\right)  $, while we write
\[
\overline{\Omega}_{\left(  123\right)  }^{\left(  p,1\right)  }\left(
M\right)
\]
for the intersection of $\overline{\Omega}_{\left(  12\right)  }^{\left(
p,1\right)  }\left(  M\right)  $\ and $\overline{\Omega}_{\left(  13\right)
}^{\left(  p,1\right)  }\left(  M\right)  $.
\end{notation}

It is easy to see that

\begin{proposition}
\label{t4.5}We have
\[
\underline{\Omega}_{\left(  123\right)  }^{\left(  p,1\right)  }\left(
M\right)  =\overline{\Omega}_{\left(  123\right)  }^{\left(  p,1\right)
}\left(  M\right)
\]

\end{proposition}

\section{\label{s5}The Jacobi Identity in the Fr\"{o}licher-Nijenhuis Algebra}

\subsection{\label{s5.1}Preparatory Considerations}

Let us begin this subsection by adding the following definition to our lexicon.

\begin{definition}
Given an object $M$\ in the category $\mathcal{K}$ and natural numbers $p,q$,
we define a morphism
\[
\underline{\mathrm{Conv}}_{p,q}^{M}:\left[  \mathbf{T}^{\mathcal{W}_{D^{p}}%
}M\rightarrow M\right]  \times\left[  \mathbf{T}^{\mathcal{W}_{D^{q}}%
}M\rightarrow M\right]  \rightarrow\left[  \mathbf{T}^{\mathcal{W}_{D^{p+q}}%
}M\rightarrow M\right]
\]
in the category $\mathcal{K}$\ to be the exponential transpose of
\begin{align*}
& \left[  \mathbf{T}^{\mathcal{W}_{D^{p}}}M\rightarrow M\right]  \times\left[
\mathbf{T}^{\mathcal{W}_{D^{q}}}M\rightarrow M\right]  \times\mathbf{T}%
^{\mathcal{W}_{D^{p+q}}}M\\
& =\left[  \mathbf{T}^{\mathcal{W}_{D^{p}}}M\rightarrow M\right]
\times\left[  \mathbf{T}^{\mathcal{W}_{D^{q}}}M\rightarrow M\right]
\times\mathbf{T}^{\mathcal{W}_{D^{p}}\otimes_{k}\mathcal{W}_{D^{q}}}M\\
& =\left[  \mathbf{T}^{\mathcal{W}_{D^{p}}}M\rightarrow M\right]
\times\left[  \mathbf{T}^{\mathcal{W}_{D^{q}}}M\rightarrow M\right]
\times\mathbf{T}^{\mathcal{W}_{D^{q}}\otimes_{k}\mathcal{W}_{D^{p}}}M\\
& =\left[  \mathbf{T}^{\mathcal{W}_{D^{p}}}M\rightarrow M\right]
\times\left[  \mathbf{T}^{\mathcal{W}_{D^{q}}}M\rightarrow M\right]
\times\mathbf{T}^{\mathcal{W}_{D^{p}}}\left(  \mathbf{T}^{\mathcal{W}_{D^{q}}%
}M\right) \\
& \underrightarrow{\mathrm{id}_{\left[  \mathbf{T}^{\mathcal{W}_{D^{p}}%
}M\rightarrow M\right]  }\times\alpha_{\mathcal{W}_{D^{p}\rightarrow1}}\left(
\left[  \mathbf{T}^{\mathcal{W}_{D^{q}}}M\rightarrow M\right]  \right)
\times\mathrm{id}_{\mathbf{T}^{\mathcal{W}_{D^{p}}}\left(  \mathbf{T}%
^{\mathcal{W}_{D^{q}}}M\right)  }}\\
& \left[  \mathbf{T}^{\mathcal{W}_{D^{p}}}M\rightarrow M\right]
\times\mathbf{T}^{\mathcal{W}_{D^{p}}}\left[  \mathbf{T}^{\mathcal{W}_{D^{q}}%
}M\rightarrow M\right]  \times\mathbf{T}^{\mathcal{W}_{D^{p}}}\left(
\mathbf{T}^{\mathcal{W}_{D^{q}}}M\right) \\
& =\left[  \mathbf{T}^{\mathcal{W}_{D^{p}}}M\rightarrow M\right]
\times\mathbf{T}^{\mathcal{W}_{D^{p}}}\left(  \left[  \mathbf{T}%
^{\mathcal{W}_{D^{q}}}M\rightarrow M\right]  \times\mathbf{T}^{\mathcal{W}%
_{D^{q}}}M\right) \\
& \underrightarrow{\mathrm{id}_{\left[  \mathbf{T}^{\mathcal{W}_{D^{p}}%
}M\rightarrow M\right]  }\times\mathbf{T}^{\mathcal{W}_{D^{p}}}\mathrm{ev}%
_{\left[  \mathbf{T}^{\mathcal{W}_{D^{q}}}M\rightarrow M\right]
\times\mathbf{T}^{\mathcal{W}_{D^{q}}}M}}\\
& \left[  \mathbf{T}^{\mathcal{W}_{D^{p}}}M\rightarrow M\right]
\times\mathbf{T}^{\mathcal{W}_{D^{p}}}M\\
& \underrightarrow{\mathrm{ev}_{\left[  \mathbf{T}^{\mathcal{W}_{D^{p}}%
}M\rightarrow M\right]  \times\mathbf{T}^{\mathcal{W}_{D^{p}}}M}}\\
& M\text{,}%
\end{align*}
while we define another morphism
\[
\overline{\mathrm{Conv}}_{p,q}^{M}:\left[  \mathbf{T}^{\mathcal{W}_{D^{p}}%
}M\rightarrow M\right]  \times\left[  \mathbf{T}^{\mathcal{W}_{D^{q}}%
}M\rightarrow M\right]  \rightarrow\left[  \mathbf{T}^{\mathcal{W}_{D^{p+q}}%
}M\rightarrow M\right]
\]
in the category $\mathcal{K}$\ to be the exponential transpose of
\begin{align*}
& \left[  \mathbf{T}^{\mathcal{W}_{D^{p}}}M\rightarrow M\right]  \times\left[
\mathbf{T}^{\mathcal{W}_{D^{q}}}M\rightarrow M\right]  \times\mathbf{T}%
^{\mathcal{W}_{D^{p+q}}}M\\
& =\left[  \mathbf{T}^{\mathcal{W}_{D^{p}}}M\rightarrow M\right]
\times\left[  \mathbf{T}^{\mathcal{W}_{D^{q}}}M\rightarrow M\right]
\times\mathbf{T}^{\mathcal{W}_{D^{p}}\otimes_{k}\mathcal{W}_{D^{q}}}M\\
& =\left[  \mathbf{T}^{\mathcal{W}_{D^{q}}}M\rightarrow M\right]
\times\left[  \mathbf{T}^{\mathcal{W}_{D^{p}}}M\rightarrow M\right]
\times\mathbf{T}^{\mathcal{W}_{D^{q}}}\left(  \mathbf{T}^{\mathcal{W}_{D^{p}}%
}M\right) \\
& \underrightarrow{\mathrm{id}_{\left[  \mathbf{T}^{\mathcal{W}_{D^{q}}%
}M\rightarrow M\right]  }\times\alpha_{\mathcal{W}_{D^{q}\rightarrow1}}\left(
\left[  \mathbf{T}^{\mathcal{W}_{D^{p}}}M\rightarrow M\right]  \right)
\times\mathrm{id}_{\mathbf{T}^{\mathcal{W}_{D^{q}}}\left(  \mathbf{T}%
^{\mathcal{W}_{D^{p}}}M\right)  }}\\
& \left[  \mathbf{T}^{\mathcal{W}_{D^{q}}}M\rightarrow M\right]
\times\mathbf{T}^{\mathcal{W}_{D^{q}}}\left[  \mathbf{T}^{\mathcal{W}_{D^{p}}%
}M\rightarrow M\right]  \times\mathbf{T}^{\mathcal{W}_{D^{q}}}\left(
\mathbf{T}^{\mathcal{W}_{D^{p}}}M\right) \\
& =\left[  \mathbf{T}^{\mathcal{W}_{D^{q}}}M\rightarrow M\right]
\times\mathbf{T}^{\mathcal{W}_{D^{q}}}\left(  \left[  \mathbf{T}%
^{\mathcal{W}_{D^{p}}}M\rightarrow M\right]  \times\mathbf{T}^{\mathcal{W}%
_{D^{p}}}M\right) \\
& \underrightarrow{\mathrm{id}_{\left[  \mathbf{T}^{\mathcal{W}_{D^{q}}%
}M\rightarrow M\right]  }\times\mathbf{T}^{\mathcal{W}_{D^{q}}}\mathrm{ev}%
_{\left[  \mathbf{T}^{\mathcal{W}_{D^{p}}}M\rightarrow M\right]
\times\mathbf{T}^{\mathcal{W}_{D^{p}}}M}}\\
& \left[  \mathbf{T}^{\mathcal{W}_{D^{q}}}M\rightarrow M\right]
\times\mathbf{T}^{\mathcal{W}_{D^{q}}}M\\
& \underrightarrow{\mathrm{ev}_{\left[  \mathbf{T}^{\mathcal{W}_{D^{q}}%
}M\rightarrow M\right]  \times\mathbf{T}^{\mathcal{W}_{D^{q}}}M}}\\
& M
\end{align*}

\end{definition}

\begin{remark}
\thinspace

\begin{enumerate}
\item Our two convolutions $\underline{\mathrm{Conv}}_{p,q}^{M}$\ and
$\overline{\mathrm{Conv}}_{p,q}^{M}$\ are reminiscent of the familiar ones in
abstract harmonic analysis and the theory of Schwartz distributions.

\item In case that $p=q=0$, it obtains that
\[
M\otimes\mathcal{W}_{D^{p}}=M\otimes\mathcal{W}_{D^{q}}=M\otimes
\mathcal{W}_{D^{p+q}}=M
\]
so that
\[
\left[  M\otimes\mathcal{W}_{D^{p}}\rightarrow M\right]  =\left[
M\otimes\mathcal{W}_{D^{q}}\rightarrow M\right]  =\left[  M\otimes
\mathcal{W}_{D^{p+q}}\rightarrow M\right]  =[M\rightarrow M]\text{,}%
\]
in which we have
\[
\underline{\mathrm{Conv}}_{p,q}^{M}=\mathrm{ass}_{M}%
\]
and the morphism $\overline{\mathrm{Conv}}_{p,q}^{M}$ is identical to the
morphism
\begin{align*}
& \underset{1}{[M\rightarrow M]}\times\underset{2}{[M\rightarrow M]}\\
& =\underset{2}{[M\rightarrow M]}\times\underset{1}{[M\rightarrow M]}\\
& \underrightarrow{\mathrm{ass}_{M}}\\
\lbrack M  & \rightarrow M]
\end{align*}
where $\mathrm{ass}_{M}$ stands for composition.
\end{enumerate}
\end{remark}

It should be obvious that

\begin{proposition}
\label{t5.1.1}Given natural numbers $p,q$, the morphism
\begin{align*}
& \left[  \mathbf{T}^{\mathcal{W}_{D^{p}}}M\rightarrow M\right]  \times\left[
\mathbf{T}^{\mathcal{W}_{D^{q}}}M\rightarrow M\right] \\
& \underrightarrow{\underline{\mathrm{Conv}}_{p,q}^{M}}\\
& \left[  \mathbf{T}^{\mathcal{W}_{D^{p+q}}}M\rightarrow M\right] \\
& \underrightarrow{\left(  \cdot^{\sigma_{p,q}}\right)  _{\left[
\mathbf{T}^{\mathcal{W}_{D^{p+q}}}M\rightarrow M\right]  }}\\
& \left[  \mathbf{T}^{\mathcal{W}_{D^{p+q}}}M\rightarrow M\right]
\end{align*}
is identical to the morphism
\begin{align*}
& \left[  \mathbf{T}^{\mathcal{W}_{D^{p}}}M\rightarrow M\right]  \times\left[
\mathbf{T}^{\mathcal{W}_{D^{q}}}M\rightarrow M\right] \\
& =\left[  \mathbf{T}^{\mathcal{W}_{D^{q}}}M\rightarrow M\right]
\times\left[  \mathbf{T}^{\mathcal{W}_{D^{p}}}M\rightarrow M\right] \\
& \underrightarrow{\overline{\mathrm{Conv}}_{q,p}^{M}}\\
& \left[  \mathbf{T}^{\mathcal{W}_{D^{p+q}}}M\rightarrow M\right]  \text{, }%
\end{align*}
while the morphism
\begin{align*}
& \left[  \mathbf{T}^{\mathcal{W}_{D^{p}}}M\rightarrow M\right]  \times\left[
\mathbf{T}^{\mathcal{W}_{D^{q}}}M\rightarrow M\right] \\
& \underrightarrow{\overline{\mathrm{Conv}}_{p,q}^{M}}\\
& \left[  \mathbf{T}^{\mathcal{W}_{D^{p+q}}}M\rightarrow M\right] \\
& \left(  \cdot^{\sigma_{p,q}}\right)  _{\mathbf{T}^{\mathcal{W}_{D^{p+q}}}%
M}\rightarrow\\
& \left[  \mathbf{T}^{\mathcal{W}_{D^{p+q}}}M\rightarrow M\right]
\end{align*}
is identical to the morphism
\begin{align*}
& \left[  \mathbf{T}^{\mathcal{W}_{D^{p}}}M\rightarrow M\right]  \times\left[
\mathbf{T}^{\mathcal{W}_{D^{q}}}M\rightarrow M\right] \\
& =\left[  \mathbf{T}^{\mathcal{W}_{D^{q}}}M\rightarrow M\right]
\times\left[  \mathbf{T}^{\mathcal{W}_{D^{p}}}M\rightarrow M\right] \\
& \underrightarrow{\underline{\mathrm{Conv}}_{q,p}^{M}}\\
& \left[  \mathbf{T}^{\mathcal{W}_{D^{p+q}}}M\rightarrow M\right]
\end{align*}
where $\sigma_{p,q}$ is the permutation mapping the sequence
$1,...,q,q+1,...,p+q$ to the sequence $q+1,...,q+p,1,...,q$, namely,
\[
\sigma_{p,q}=\left(
\begin{array}
[c]{cccccc}%
1 & ... & p & p+1 & ... & p+q\\
q+1 & ... & q+p & 1 & ... & q
\end{array}
\right)
\]

\end{proposition}

It should be obvious that

\begin{proposition}
\label{t5.1.2}Both $\underline{\mathrm{Conv}}^{M}$\ and $\overline
{\mathrm{Conv}}^{M}$ are associative in the sense that, given an object
$M$\ in the category $\mathcal{K}$ and natural numbers $p,q,r$, the morphism
\begin{align*}
& \left[  \mathbf{T}^{\mathcal{W}_{D^{p}}}M\rightarrow M\right]  \times\left[
\mathbf{T}^{\mathcal{W}_{D^{q}}}M\rightarrow M\right]  \times\left[
\mathbf{T}^{\mathcal{W}_{D^{r}}}M\rightarrow M\right] \\
& \underrightarrow{\underline{\mathrm{Conv}}_{p,q}^{M}\times\mathrm{id}%
_{\left[  \mathbf{T}^{\mathcal{W}_{D^{r}}}M\rightarrow M\right]  }}\\
& \left[  \mathbf{T}^{\mathcal{W}_{D^{p+q}}}M\rightarrow M\right]
\times\left[  \mathbf{T}^{\mathcal{W}_{D^{r}}}M\rightarrow M\right] \\
& \underline{\mathrm{Conv}}_{p+q,r}^{M}\\
& \left[  \mathbf{T}^{\mathcal{W}_{D^{p+q+r}}}M\rightarrow M\right]
\end{align*}
is identical to the morphism
\begin{align*}
& \left[  \mathbf{T}^{\mathcal{W}_{D^{p}}}M\rightarrow M\right]  \times\left[
\mathbf{T}^{\mathcal{W}_{D^{q}}}M\rightarrow M\right]  \times\left[
\mathbf{T}^{\mathcal{W}_{D^{r}}}M\rightarrow M\right] \\
& \underrightarrow{\mathrm{id}_{\left[  \mathbf{T}^{\mathcal{W}_{D^{p}}%
}M\rightarrow M\right]  }\times\underline{\mathrm{Conv}}_{q,r}^{M}}\\
& \left[  \mathbf{T}^{\mathcal{W}_{D^{p}}}M\rightarrow M\right]  \times\left[
\mathbf{T}^{\mathcal{W}_{D^{q+r}}}M\rightarrow M\right] \\
& \underrightarrow{\underline{\mathrm{Conv}}_{p,q+r}^{M}}\\
& \left[  \mathbf{T}^{\mathcal{W}_{D^{p+q+r}}}M\rightarrow M\right]  \text{,}%
\end{align*}
and we have a similar identification for $\overline{\mathrm{Conv}}^{M}$.
\end{proposition}

\begin{remark}
This proposition enables us to write
\begin{align*}
\underline{\mathrm{Conv}}_{p.q.r}^{M}  & :\left[  \mathbf{T}^{\mathcal{W}%
_{D^{p}}}M\rightarrow M\right]  \times\left[  \mathbf{T}^{\mathcal{W}_{D^{q}}%
}M\rightarrow M\right]  \times\left[  \mathbf{T}^{\mathcal{W}_{D^{r}}%
}M\rightarrow M\right] \\
& \rightarrow\left[  \mathbf{T}^{\mathcal{W}_{D^{p+q+r}}}M\rightarrow
M\right]
\end{align*}
to denote one of the above two identical morphisms without any ambiguity, and
similarly for
\begin{align*}
\overline{\mathrm{Conv}}_{p.q.r}^{M}  & :\left[  \mathbf{T}^{\mathcal{W}%
_{D^{p}}}M\rightarrow M\right]  \times\left[  \mathbf{T}^{\mathcal{W}_{D^{q}}%
}M\rightarrow M\right]  \times\left[  \mathbf{T}^{\mathcal{W}_{D^{r}}%
}M\rightarrow M\right] \\
& \rightarrow\left[  \mathbf{T}^{\mathcal{W}_{D^{p+q+r}}}M\rightarrow
M\right]
\end{align*}

\end{remark}

\begin{notation}
Given a natural number $p$, we write
\[
\left[  \mathbf{T}^{\mathcal{W}_{D^{p}}}M\rightarrow M\right]  _{\mathrm{id}%
_{M}}%
\]
for the pullback of
\[%
\begin{array}
[c]{ccc}%
\left[  \mathbf{T}^{\mathcal{W}_{D^{p}}}M\rightarrow M\right]  _{\mathrm{id}%
_{M}} & \rightarrow & \left[  \mathbf{T}^{\mathcal{W}_{D^{p}}}M\rightarrow
M\right] \\
\downarrow &  & \downarrow\\
1 & \rightarrow & \left[  M\rightarrow M\right]
\end{array}
\text{,}%
\]
where the right vertical arrow is
\[
\left[  \alpha_{\mathcal{W}_{D^{p}\rightarrow1}}\rightarrow\mathrm{id}%
_{M}\right]  :\left[  \mathbf{T}^{\mathcal{W}_{D^{p}}}M\rightarrow M\right]
\rightarrow\left[  M\rightarrow M\right]
\]
is the canonical projection, while the bottom horizontal arrow is the
exponential transpose of
\[
\mathrm{id}_{M}:1\times M=M\rightarrow M\text{.}%
\]

\end{notation}

The following proposition should be evident.

\begin{proposition}
\label{t5.1.3}The morphism
\begin{align*}
& \left[  \mathbf{T}^{\mathcal{W}_{D^{p}}}M\rightarrow M\right]
_{\mathrm{id}_{M}}\times\left[  \mathbf{T}^{\mathcal{W}_{D^{q}}}M\rightarrow
M\right] \\
& \underrightarrow{i_{\left[  \mathbf{T}^{\mathcal{W}_{D^{p}}}M\rightarrow
M\right]  _{\mathrm{id}_{M}}}^{\left[  \mathbf{T}^{\mathcal{W}_{D^{p}}%
}M\rightarrow M\right]  }\times\mathrm{id}_{\left[  \mathbf{T}^{\mathcal{W}%
_{D^{q}}}M\rightarrow M\right]  }}\\
& \left[  \mathbf{T}^{\mathcal{W}_{D^{p}}}M\rightarrow M\right]  \times\left[
\mathbf{T}^{\mathcal{W}_{D^{q}}}M\rightarrow M\right] \\
& \underrightarrow{\underline{\mathrm{Conv}}_{p,q}^{M}}\\
& \left[  \mathbf{T}^{\mathcal{W}_{D^{p+q}}}M\rightarrow M\right]
\end{align*}
is identical to the morphism
\begin{align*}
& \left[  \mathbf{T}^{\mathcal{W}_{D^{p}}}M\rightarrow M\right]
_{\mathrm{id}_{M}}\times\left[  \mathbf{T}^{\mathcal{W}_{D^{q}}}M\rightarrow
M\right] \\
& \underrightarrow{i_{\left[  \mathbf{T}^{\mathcal{W}_{D^{p}}}M\rightarrow
M\right]  _{\mathrm{id}_{M}}}^{\left[  \mathbf{T}^{\mathcal{W}_{D^{p}}%
}M\rightarrow M\right]  }\times\mathrm{id}_{\left[  \mathbf{T}^{\mathcal{W}%
_{D^{q}}}M\rightarrow M\right]  }}\\
& \left[  \mathbf{T}^{\mathcal{W}_{D^{p}}}M\rightarrow M\right]  \times\left[
\mathbf{T}^{\mathcal{W}_{D^{q}}}M\rightarrow M\right] \\
& \underrightarrow{\overline{\mathrm{Conv}}_{p,q}^{M}}\\
& \left[  \mathbf{T}^{\mathcal{W}_{D^{p+q}}}M\rightarrow M\right]
\end{align*}
Similarly, both of the morphisms
\begin{align*}
& \left[  \mathbf{T}^{\mathcal{W}_{D^{p}}}M\rightarrow M\right]  \times\left[
\mathbf{T}^{\mathcal{W}_{D^{q}}}M\rightarrow M\right]  _{\mathrm{id}_{M}}\\
& \underrightarrow{\mathrm{id}_{\left[  \mathbf{T}^{\mathcal{W}_{D^{p}}%
}M\rightarrow M\right]  }\times i_{\left[  \mathbf{T}^{\mathcal{W}_{D^{q}}%
}M\rightarrow M\right]  _{\mathrm{id}_{M}}}^{\left[  \mathbf{T}^{\mathcal{W}%
_{D^{q}}}M\rightarrow M\right]  }}\\
& \left[  \mathbf{T}^{\mathcal{W}_{D^{p}}}M\rightarrow M\right]  \times\left[
\mathbf{T}^{\mathcal{W}_{D^{q}}}M\rightarrow M\right] \\
& \underrightarrow{\underline{\mathrm{Conv}}_{p,q}^{M}}\\
& \left[  \mathbf{T}^{\mathcal{W}_{D^{p+q}}}M\rightarrow M\right]
\end{align*}
and
\begin{align*}
& \left[  \mathbf{T}^{\mathcal{W}_{D^{p}}}M\rightarrow M\right]  \times\left[
\mathbf{T}^{\mathcal{W}_{D^{q}}}M\rightarrow M\right]  _{\mathrm{id}_{M}}\\
& \underrightarrow{\mathrm{id}_{\left[  \mathbf{T}^{\mathcal{W}_{D^{p}}%
}M\rightarrow M\right]  }\times i_{\left[  \mathbf{T}^{\mathcal{W}_{D^{q}}%
}M\rightarrow M\right]  _{\mathrm{id}_{M}}}^{\left[  \mathbf{T}^{\mathcal{W}%
_{D^{q}}}M\rightarrow M\right]  }}\\
& \left[  \mathbf{T}^{\mathcal{W}_{D^{p}}}M\rightarrow M\right]  \times\left[
\mathbf{T}^{\mathcal{W}_{D^{q}}}M\rightarrow M\right] \\
& \underrightarrow{\overline{\mathrm{Conv}}_{p,q}^{M}}\\
& \left[  \mathbf{T}^{\mathcal{W}_{D^{p+q}}}M\rightarrow M\right]
\end{align*}
are identical. Besides, the morphism
\begin{align*}
& \left[  \mathbf{T}^{\mathcal{W}_{D^{p}}}M\rightarrow M\right]
_{\mathrm{id}_{M}}\times\left[  \mathbf{T}^{\mathcal{W}_{D^{q}}}M\rightarrow
M\right]  _{\mathrm{id}_{M}}\\
& \underrightarrow{i_{\left[  \mathbf{T}^{\mathcal{W}_{D^{p}}}M\rightarrow
M\right]  _{\mathrm{id}_{M}}}^{\left[  \mathbf{T}^{\mathcal{W}_{D^{p}}%
}M\rightarrow M\right]  }\times i_{\left[  \mathbf{T}^{\mathcal{W}_{D^{q}}%
}M\rightarrow M\right]  _{\mathrm{id}_{M}}}^{\left[  \mathbf{T}^{\mathcal{W}%
_{D^{q}}}M\rightarrow M\right]  }}\\
& \left[  \mathbf{T}^{\mathcal{W}_{D^{p}}}M\rightarrow M\right]  \times\left[
\mathbf{T}^{\mathcal{W}_{D^{q}}}M\rightarrow M\right] \\
& \underrightarrow{\underline{\mathrm{Conv}}_{p,q}^{M}}\\
& \left[  \mathbf{T}^{\mathcal{W}_{D^{p+q}}}M\rightarrow M\right]  \text{,}%
\end{align*}
which is identical to the morphism
\begin{align*}
& \left[  \mathbf{T}^{\mathcal{W}_{D^{p}}}M\rightarrow M\right]
_{\mathrm{id}_{M}}\times\left[  \mathbf{T}^{\mathcal{W}_{D^{q}}}M\rightarrow
M\right]  _{\mathrm{id}_{M}}\\
& \underrightarrow{i_{\left[  \mathbf{T}^{\mathcal{W}_{D^{p}}}M\rightarrow
M\right]  _{\mathrm{id}_{M}}}^{\left[  \mathbf{T}^{\mathcal{W}_{D^{p}}%
}M\rightarrow M\right]  }\times i_{\left[  \mathbf{T}^{\mathcal{W}_{D^{q}}%
}M\rightarrow M\right]  _{\mathrm{id}_{M}}}^{\left[  \mathbf{T}^{\mathcal{W}%
_{D^{q}}}M\rightarrow M\right]  }}\\
& \left[  \mathbf{T}^{\mathcal{W}_{D^{p}}}M\rightarrow M\right]  \times\left[
\mathbf{T}^{\mathcal{W}_{D^{q}}}M\rightarrow M\right] \\
& \underrightarrow{\overline{\mathrm{Conv}}_{p,q}^{M}}\\
& \left[  \mathbf{T}^{\mathcal{W}_{D^{p+q}}}M\rightarrow M\right]  \text{,}%
\end{align*}
is to be factored through the canonical injection
\[
i_{\left[  \mathbf{T}^{\mathcal{W}_{D^{p+q}}}M\rightarrow M\right]
_{\mathrm{id}_{M}}}^{\left[  \mathbf{T}^{\mathcal{W}_{D^{p+q}}}M\rightarrow
M\right]  }:\left[  \mathbf{T}^{\mathcal{W}_{D^{p+q}}}M\rightarrow M\right]
_{\mathrm{id}_{M}}\rightarrow\left[  \mathbf{T}^{\mathcal{W}_{D^{p+q}}%
}M\rightarrow M\right]
\]

\end{proposition}

\begin{definition}
We define a morphism
\begin{align*}
\underline{\mathrm{\Pr od}}_{\left(  p,m\right)  ,\left(  q,n\right)  }^{M}  &
:\mathbf{T}^{\mathcal{W}_{D^{m}}}\left[  \mathbf{T}^{\mathcal{W}_{D^{p}}%
}M\rightarrow M\right]  \times\mathbf{T}^{\mathcal{W}_{D^{n}}}\left[
\mathbf{T}^{\mathcal{W}_{D^{q}}}M\rightarrow M\right] \\
& \rightarrow\mathbf{T}^{\mathcal{W}_{D^{m+n}}}\left[  \mathbf{T}%
^{\mathcal{W}_{D^{p+q}}}M\rightarrow M\right]
\end{align*}
to be
\begin{align*}
& \mathbf{T}^{\mathcal{W}_{D^{m}}}\left[  \mathbf{T}^{\mathcal{W}_{D^{p}}%
}M\rightarrow M\right]  \times\mathbf{T}^{\mathcal{W}_{D^{n}}}\left[
\mathbf{T}^{\mathcal{W}_{D^{q}}}M\rightarrow M\right] \\
& \underline{\alpha_{\mathcal{W}_{\left(  d_{1},...,d_{m},d_{m+1}%
,...,d_{m+n}\right)  \in D^{m+n}\mapsto\left(  d_{1},...,d_{m}\right)  \in
D^{m}}}\left(  \left[  \mathbf{T}^{\mathcal{W}_{D^{p}}}M\rightarrow M\right]
\right)  \times}\\
& \underrightarrow{\alpha_{\mathcal{W}_{\left(  d_{1},...,d_{m},d_{m+1}%
,...,d_{m+n}\right)  \in D^{m+n}\mapsto\left(  d_{m+1},...,d_{m+n}\right)  \in
D^{n}}}\left(  \left[  \mathbf{T}^{\mathcal{W}_{D^{q}}}M\rightarrow M\right]
\right)  }\\
& \mathbf{T}^{\mathcal{W}_{D^{m+n}}}\left[  \mathbf{T}^{\mathcal{W}_{D^{p}}%
}M\rightarrow M\right]  \times\mathbf{T}^{\mathcal{W}_{D^{m+n}}}\left[
\mathbf{T}^{\mathcal{W}_{D^{q}}}M\rightarrow M\right] \\
& =\mathbf{T}^{\mathcal{W}_{D^{m+n}}}\left(  \left[  \mathbf{T}^{\mathcal{W}%
_{D^{p}}}M\rightarrow M\right]  \times\left[  \mathbf{T}^{\mathcal{W}_{D^{q}}%
}M\rightarrow M\right]  \right) \\
& \underrightarrow{\mathbf{T}^{\mathcal{W}_{D^{m+n}}}\underline{\mathrm{Conv}%
}_{p,q}^{M}}\\
& \mathbf{T}^{\mathcal{W}_{D^{m+n}}}\left[  \mathbf{T}^{\mathcal{W}_{D^{p+q}}%
}M\rightarrow M\right]  \text{,}%
\end{align*}
while we define a morphism
\begin{align*}
\overline{\mathrm{\Pr od}}_{\left(  p,m\right)  ,\left(  q,n\right)  }^{M}  &
:\mathbf{T}^{\mathcal{W}_{D^{m}}}\left[  \mathbf{T}^{\mathcal{W}_{D^{p}}%
}M\rightarrow M\right]  \times\mathbf{T}^{\mathcal{W}_{D^{n}}}\left[
\mathbf{T}^{\mathcal{W}_{D^{q}}}M\rightarrow M\right] \\
& \rightarrow\mathbf{T}^{\mathcal{W}_{D^{m+n}}}\left[  \mathbf{T}%
^{\mathcal{W}_{D^{p+q}}}M\rightarrow M\right]
\end{align*}
to be
\begin{align*}
& \mathbf{T}^{\mathcal{W}_{D^{m}}}\left[  \mathbf{T}^{\mathcal{W}_{D^{p}}%
}M\rightarrow M\right]  \times\mathbf{T}^{\mathcal{W}_{D^{n}}}\left[
\mathbf{T}^{\mathcal{W}_{D^{q}}}M\rightarrow M\right] \\
& \underline{\alpha_{\mathcal{W}_{\left(  d_{1},...,d_{m},d_{m+1}%
,...,d_{m+n}\right)  \in D^{m+n}\mapsto\left(  d_{1},...,d_{m}\right)  \in
D^{m}}}\left(  \left[  \mathbf{T}^{\mathcal{W}_{D^{p}}}M\rightarrow M\right]
\right)  \times}\\
& \underrightarrow{\alpha_{\mathcal{W}_{\left(  d_{1},...,d_{m},d_{m+1}%
,...,d_{m+n}\right)  \in D^{m+n}\mapsto\left(  d_{m+1},...,d_{m+n}\right)  \in
D^{n}}}\left(  \left[  \mathbf{T}^{\mathcal{W}_{D^{q}}}M\rightarrow M\right]
\right)  }\\
& \mathbf{T}^{\mathcal{W}_{D^{m+n}}}\left[  \mathbf{T}^{\mathcal{W}_{D^{p}}%
}M\rightarrow M\right]  \times\mathbf{T}^{\mathcal{W}_{D^{m+n}}}\left[
\mathbf{T}^{\mathcal{W}_{D^{q}}}M\rightarrow M\right] \\
& =\mathbf{T}^{\mathcal{W}_{D^{m+n}}}\left(  \left[  \mathbf{T}^{\mathcal{W}%
_{D^{p}}}M\rightarrow M\right]  \times\left[  \mathbf{T}^{\mathcal{W}_{D^{q}}%
}M\rightarrow M\right]  \right) \\
& \underrightarrow{\mathbf{T}^{\mathcal{W}_{D^{m+n}}}\overline{\mathrm{Conv}%
}_{p,q}^{M}}\\
& \mathbf{T}^{\mathcal{W}_{D^{m+n}}}\left[  \mathbf{T}^{\mathcal{W}_{D^{p+q}}%
}M\rightarrow M\right]
\end{align*}

\end{definition}

It should be obvious that

\begin{proposition}
\label{t5.1.4}Both $\underline{\mathrm{\Pr od}}^{M}$\ and $\overline
{\mathrm{\Pr od}}^{M}$ are associative in the sense that, given an object
$M$\ in the category $\mathcal{K}$ and natural numbers $l,m,n,p,q,r$, the
morphism
\begin{align}
& \mathbf{T}^{\mathcal{W}_{D^{l}}}\left[  \mathbf{T}^{\mathcal{W}_{D^{p}}%
}M\rightarrow M\right]  \times\mathbf{T}^{\mathcal{W}_{D^{m}}}\left[
\mathbf{T}^{\mathcal{W}_{D^{q}}}M\rightarrow M\right]  \times\mathbf{T}%
^{\mathcal{W}_{D^{n}}}\left[  \mathbf{T}^{\mathcal{W}_{D^{r}}}M\rightarrow
M\right] \nonumber\\
& \underrightarrow{\underline{\mathrm{\Pr od}}_{\left(  p,l\right)  ,\left(
q,m\right)  }^{M}\times\mathrm{id}_{\mathbf{T}^{\mathcal{W}_{D^{n}}}\left[
\mathbf{T}^{\mathcal{W}_{D^{r}}}M\rightarrow M\right]  }}\nonumber\\
& \mathbf{T}^{\mathcal{W}_{D^{m+n}}}\left[  \mathbf{T}^{\mathcal{W}_{D^{p+q}}%
}M\rightarrow M\right]  \times\mathbf{T}^{\mathcal{W}_{D^{n}}}\left[
\mathbf{T}^{\mathcal{W}_{D^{r}}}M\rightarrow M\right] \nonumber\\
& \underline{\mathrm{\Pr od}}_{\left(  p+q,l+m\right)  ,\left(  r,n\right)
}^{M}\nonumber\\
& \mathbf{T}^{\mathcal{W}_{D^{l+m+n}}}\left[  \mathbf{T}^{\mathcal{W}%
_{D^{p+q+r}}}M\rightarrow M\right] \label{5.1.4.1}%
\end{align}
is identical to the morphism
\begin{align}
& \mathbf{T}^{\mathcal{W}_{D^{l}}}\left[  \mathbf{T}^{\mathcal{W}_{D^{p}}%
}M\rightarrow M\right]  \times\mathbf{T}^{\mathcal{W}_{D^{m}}}\left[
\mathbf{T}^{\mathcal{W}_{D^{q}}}M\rightarrow M\right]  \times\mathbf{T}%
^{\mathcal{W}_{D^{n}}}\left[  \mathbf{T}^{\mathcal{W}_{D^{r}}}M\rightarrow
M\right] \nonumber\\
& \underrightarrow{\mathrm{id}_{\mathbf{T}^{\mathcal{W}_{D^{l}}}\left[
\mathbf{T}^{\mathcal{W}_{D^{p}}}M\rightarrow M\right]  }\times\underline
{\mathrm{\Pr od}}_{\left(  q,m\right)  ,\left(  r,n\right)  }^{M}}\nonumber\\
& \mathbf{T}^{\mathcal{W}_{D^{l}}}\left[  \mathbf{T}^{\mathcal{W}_{D^{p}}%
}M\rightarrow M\right]  \times\mathbf{T}^{\mathcal{W}_{D^{m+n}}}\left[
\mathbf{T}^{\mathcal{W}_{D^{q+r}}}M\rightarrow M\right] \nonumber\\
& \underrightarrow{\underline{\mathrm{\Pr od}}_{\left(  p,l\right)  ,\left(
q+r,m+n\right)  }^{M}}\nonumber\\
& \mathbf{T}^{\mathcal{W}_{D^{l+m+n}}}\left[  \mathbf{T}^{\mathcal{W}%
_{D^{p+q+r}}}M\rightarrow M\right]  \text{,}\label{5.1.4.2}%
\end{align}
and similarly for $\overline{\mathrm{\Pr od}}^{M}$.
\end{proposition}

\begin{proof}
To prove the first statement, we note that the morphism (\ref{5.1.4.1})\ is
identical to the morphism
\begin{align*}
& \mathbf{T}^{\mathcal{W}_{D^{l}}}\left[  \mathbf{T}^{\mathcal{W}_{D^{p}}%
}M\rightarrow M\right]  \times\mathbf{T}^{\mathcal{W}_{D^{m}}}\left[
\mathbf{T}^{\mathcal{W}_{D^{q}}}M\rightarrow M\right]  \times\mathbf{T}%
^{\mathcal{W}_{D^{n}}}\left[  \mathbf{T}^{\mathcal{W}_{D^{r}}}M\rightarrow
M\right] \\
& \underline{\alpha_{\mathcal{W}_{\left(  d_{1},...,d_{l},d_{l+1}%
,...,d_{l+m},d_{l+m+1},...,d_{l+m+n}\right)  \in D^{l+m+n}\mapsto\left(
d_{1},...,d_{l}\right)  \in D^{l}}}\left(  \left[  \mathbf{T}^{\mathcal{W}%
_{D^{p}}}M\rightarrow M\right]  \right)  \times}\\
& \underline{\alpha_{\mathcal{W}_{\left(  d_{1},...,d_{l},d_{l+1}%
,...,d_{l+m},d_{l+m+1},...,d_{l+m+n}\right)  \in D^{l+m+n}\mapsto\left(
d_{l+1},...,d_{l+m}\right)  \in D^{m}}}\left(  \left[  \mathbf{T}%
^{\mathcal{W}_{D^{q}}}M\rightarrow M\right]  \right)  \times}\\
& \underrightarrow{\alpha_{\mathcal{W}_{\left(  d_{1},...,d_{l},d_{l+1}%
,...,d_{l+m},d_{l+m+1},...,d_{l+m+n}\right)  \in D^{l+m+n}\mapsto\left(
d_{l+m+1},...,d_{l+m+n}\right)  \in D^{n}}}\left(  \left[  \mathbf{T}%
^{\mathcal{W}_{D^{r}}}M\rightarrow M\right]  \right)  }\\
& \mathbf{T}^{\mathcal{W}_{D^{l+m+n}}}\left[  \mathbf{T}^{\mathcal{W}_{D^{p}}%
}M\rightarrow M\right]  \times\mathbf{T}^{\mathcal{W}_{D^{l+m+n}}}\left[
\mathbf{T}^{\mathcal{W}_{D^{q}}}M\rightarrow M\right]  \times\mathbf{T}%
^{\mathcal{W}_{D^{l+m+n}}}\left[  \mathbf{T}^{\mathcal{W}_{D^{r}}}M\rightarrow
M\right] \\
& =\mathbf{T}^{\mathcal{W}_{D^{l+m+n}}}\left(  \left[  \mathbf{T}%
^{\mathcal{W}_{D^{p}}}M\rightarrow M\right]  \times\left[  \mathbf{T}%
^{\mathcal{W}_{D^{q}}}M\rightarrow M\right]  \times\left[  \mathbf{T}%
^{\mathcal{W}_{D^{r}}}M\rightarrow M\right]  \right) \\
& \underrightarrow{\mathbf{T}^{\mathcal{W}_{D^{l+m+n}}}\left(  \underline
{\mathrm{Conv}}_{p,q}^{M}\times\mathrm{id}_{\left[  \mathbf{T}^{\mathcal{W}%
_{D^{r}}}M\rightarrow M\right]  }\right)  }\\
& \mathbf{T}^{\mathcal{W}_{D^{l+m+n}}}\left(  \left[  \mathbf{T}%
^{\mathcal{W}_{D^{p+q}}}M\rightarrow M\right]  \times\left[  \mathbf{T}%
^{\mathcal{W}_{D^{r}}}M\rightarrow M\right]  \right) \\
& \underrightarrow{\mathbf{T}^{\mathcal{W}_{D^{l+m+n}}}\left(  \underline
{\mathrm{Conv}}_{p+q,r}^{M}\right)  }\\
& \mathbf{T}^{\mathcal{W}_{D^{l+m+n}}}\left[  \mathbf{T}^{\mathcal{W}%
_{D^{p+q+r}}}M\rightarrow M\right]
\end{align*}
while the morphism (\ref{5.1.4.2})\ is identical to the morphism
\begin{align*}
& \mathbf{T}^{\mathcal{W}_{D^{l}}}\left[  \mathbf{T}^{\mathcal{W}_{D^{p}}%
}M\rightarrow M\right]  \times\mathbf{T}^{\mathcal{W}_{D^{m}}}\left[
\mathbf{T}^{\mathcal{W}_{D^{q}}}M\rightarrow M\right]  \times\mathbf{T}%
^{\mathcal{W}_{D^{n}}}\left[  \mathbf{T}^{\mathcal{W}_{D^{r}}}M\rightarrow
M\right] \\
& \underline{\alpha_{\mathcal{W}_{\left(  d_{1},...,d_{l},d_{l+1}%
,...,d_{l+m},d_{l+m+1},...,d_{l+m+n}\right)  \in D^{l+m+n}\mapsto\left(
d_{1},...,d_{l}\right)  \in D^{l}}}\left(  \left[  \mathbf{T}^{\mathcal{W}%
_{D^{p}}}M\rightarrow M\right]  \right)  \times}\\
& \underline{\alpha_{\mathcal{W}_{\left(  d_{1},...,d_{l},d_{l+1}%
,...,d_{l+m},d_{l+m+1},...,d_{l+m+n}\right)  \in D^{l+m+n}\mapsto\left(
d_{l+1},...,d_{l+m}\right)  \in D^{m}}}\left(  \left[  \mathbf{T}%
^{\mathcal{W}_{D^{q}}}M\rightarrow M\right]  \right)  \times}\\
& \underrightarrow{\alpha_{\mathcal{W}_{\left(  d_{1},...,d_{l},d_{l+1}%
,...,d_{l+m},d_{l+m+1},...,d_{l+m+n}\right)  \in D^{l+m+n}\mapsto\left(
d_{l+m+1},...,d_{l+m+n}\right)  \in D^{n}}}\left(  \left[  \mathbf{T}%
^{\mathcal{W}_{D^{r}}}M\rightarrow M\right]  \right)  }\\
& \mathbf{T}^{\mathcal{W}_{D^{l+m+n}}}\left[  \mathbf{T}^{\mathcal{W}_{D^{p}}%
}M\rightarrow M\right]  \times\mathbf{T}^{\mathcal{W}_{D^{l+m+n}}}\left[
\mathbf{T}^{\mathcal{W}_{D^{q}}}M\rightarrow M\right]  \times\mathbf{T}%
^{\mathcal{W}_{D^{l+m+n}}}\left[  \mathbf{T}^{\mathcal{W}_{D^{r}}}M\rightarrow
M\right] \\
& =\mathbf{T}^{\mathcal{W}_{D^{l+m+n}}}\left(  \left[  \mathbf{T}%
^{\mathcal{W}_{D^{p}}}M\rightarrow M\right]  \times\left[  \mathbf{T}%
^{\mathcal{W}_{D^{q}}}M\rightarrow M\right]  \times\left[  \mathbf{T}%
^{\mathcal{W}_{D^{r}}}M\rightarrow M\right]  \right) \\
& \underrightarrow{\mathbf{T}^{\mathcal{W}_{D^{l+m+n}}}\left(  \mathrm{id}%
_{\left[  \mathbf{T}^{\mathcal{W}_{D^{p}}}M\rightarrow M\right]  }%
\times\underline{\mathrm{Conv}}_{q,r}^{M}\right)  }\\
& \mathbf{T}^{\mathcal{W}_{D^{l+m+n}}}\left(  \left[  \mathbf{T}%
^{\mathcal{W}_{D^{p}}}M\rightarrow M\right]  \times\left[  \mathbf{T}%
^{\mathcal{W}_{D^{q+r}}}M\rightarrow M\right]  \right) \\
& \underrightarrow{\mathbf{T}^{\mathcal{W}_{D^{l+m+n}}}\left(  \underline
{\mathrm{Conv}}_{p,q+r}^{M}\right)  }\\
& \mathbf{T}^{\mathcal{W}_{D^{l+m+n}}}\left[  \mathbf{T}^{\mathcal{W}%
_{D^{p+q+r}}}M\rightarrow M\right]
\end{align*}
Therefore the desired result follows directly from Proposition \ref{t5.1.2}.
Similarly for the second statement.
\end{proof}

\begin{remark}
This proposition enables us to write
\begin{align*}
& \underline{\mathrm{\Pr od}}_{\left(  p,l\right)  ,\left(  q,m\right)
,\left(  r,n\right)  }^{M}\\
& :\mathbf{T}^{\mathcal{W}_{D^{l}}}\left[  \mathbf{T}^{\mathcal{W}_{D^{p}}%
}M\rightarrow M\right]  \times\mathbf{T}^{\mathcal{W}_{D^{m}}}\left[
\mathbf{T}^{\mathcal{W}_{D^{q}}}M\rightarrow M\right]  \times\mathbf{T}%
^{\mathcal{W}_{D^{n}}}\left[  \mathbf{T}^{\mathcal{W}_{D^{r}}}M\rightarrow
M\right] \\
& \rightarrow\mathbf{T}^{\mathcal{W}_{D^{l+m+n}}}\left[  \mathbf{T}%
^{\mathcal{W}_{D^{p+q+r}}}M\rightarrow M\right]
\end{align*}
to denote one of the above two identical morphisms without any ambiguity, and
similarly for
\begin{align*}
& \overline{\mathrm{\Pr od}}_{\left(  p,l\right)  ,\left(  q,m\right)
,\left(  r,n\right)  }^{M}\\
& :\mathbf{T}^{\mathcal{W}_{D^{l}}}\left[  \mathbf{T}^{\mathcal{W}_{D^{p}}%
}M\rightarrow M\right]  \times\mathbf{T}^{\mathcal{W}_{D^{m}}}\left[
\mathbf{T}^{\mathcal{W}_{D^{q}}}M\rightarrow M\right]  \times\mathbf{T}%
^{\mathcal{W}_{D^{n}}}\left[  \mathbf{T}^{\mathcal{W}_{D^{r}}}M\rightarrow
M\right] \\
& \rightarrow\mathbf{T}^{\mathcal{W}_{D^{l+m+n}}}\left[  \mathbf{T}%
^{\mathcal{W}_{D^{p+q+r}}}M\rightarrow M\right]
\end{align*}

\end{remark}

\subsection{\label{s5.2}The First Consideration}

In this subsection we are concerned with the Lie algebra structure of
$\overline{\Omega}_{\left(  1\right)  }^{\left(  p,1\right)  }\left(
M\right)  $'s. Let us begin with

\begin{lemma}
\label{t5.2.1}The morphism
\begin{align*}
& \overline{\Omega}_{\left(  1\right)  }^{\left(  p,1\right)  }\left(
M\right)  \times\overline{\Omega}_{\left(  1\right)  }^{\left(  q,1\right)
}\left(  M\right) \\
& \underrightarrow{i_{\overline{\Omega}_{\left(  1\right)  }^{\left(
p,1\right)  }\left(  M\right)  }^{\overline{\Omega}_{\left(  0\right)
}^{\left(  p,1\right)  }\left(  M\right)  }\times i_{\overline{\Omega
}_{\left(  1\right)  }^{\left(  q,1\right)  }\left(  M\right)  }%
^{\overline{\Omega}_{\left(  0\right)  }^{\left(  q,1\right)  }\left(
M\right)  }}\\
& \overline{\Omega}_{\left(  0\right)  }^{\left(  p,1\right)  }\left(
M\right)  \times\overline{\Omega}_{\left(  0\right)  }^{\left(  q,1\right)
}\left(  M\right) \\
& \underrightarrow{\underline{\mathrm{\Pr od}}_{\left(  p,1\right)  ,\left(
q,1\right)  }^{M}}\\
& \mathbf{T}^{\mathcal{W}_{D^{2}}}\left[  \mathbf{T}^{\mathcal{W}_{D^{p+q}}%
}M\rightarrow M\right] \\
& \underrightarrow{\alpha_{\mathcal{W}_{\left(  d_{1},d_{2}\right)  \in
D(2)\mapsto\left(  d_{1},d_{2}\right)  \in D^{2}}}\left(  \left[
\mathbf{T}^{\mathcal{W}_{D^{p+q}}}M\rightarrow M\right]  \right)  }\\
& \mathbf{T}^{\mathcal{W}_{D(2)}}\left[  \mathbf{T}^{\mathcal{W}_{D^{p+q}}%
}M\rightarrow M\right]
\end{align*}
is identical to the morphism
\begin{align*}
& \overline{\Omega}_{\left(  1\right)  }^{\left(  p,1\right)  }\left(
M\right)  \times\overline{\Omega}_{\left(  1\right)  }^{\left(  q,1\right)
}\left(  M\right) \\
& \underrightarrow{i_{\overline{\Omega}_{\left(  1\right)  }^{\left(
p,1\right)  }\left(  M\right)  }^{\overline{\Omega}_{\left(  0\right)
}^{\left(  p,1\right)  }\left(  M\right)  }\times i_{\overline{\Omega
}_{\left(  1\right)  }^{\left(  q,1\right)  }\left(  M\right)  }%
^{\overline{\Omega}_{\left(  0\right)  }^{\left(  q,1\right)  }\left(
M\right)  }}\\
& \overline{\Omega}_{\left(  0\right)  }^{\left(  p,1\right)  }\left(
M\right)  \times\overline{\Omega}_{\left(  0\right)  }^{\left(  q,1\right)
}\left(  M\right) \\
& \underrightarrow{\overline{\mathrm{\Pr od}}_{\left(  p,1\right)  ,\left(
q,1\right)  }^{M}}\\
& \mathbf{T}^{\mathcal{W}_{D^{2}}}\left[  \mathbf{T}^{\mathcal{W}_{D^{p+q}}%
}M\rightarrow M\right] \\
& \underrightarrow{\alpha_{\mathcal{W}_{\left(  d_{1},d_{2}\right)  \in
D(2)\mapsto\left(  d_{1},d_{2}\right)  \in D^{2}}}\left(  \left[
\mathbf{T}^{\mathcal{W}_{D^{p+q}}}M\rightarrow M\right]  \right)  }\\
& \mathbf{T}^{\mathcal{W}_{D(2)}}\left[  \mathbf{T}^{\mathcal{W}_{D^{p+q}}%
}M\rightarrow M\right]
\end{align*}

\end{lemma}

\begin{proof}
By Proposition \ref{t5.1.3}, the morphism
\begin{align*}
& \overline{\Omega}_{\left(  1\right)  }^{\left(  p,1\right)  }\left(
M\right)  \times\overline{\Omega}_{\left(  1\right)  }^{\left(  q,1\right)
}\left(  M\right) \\
& \underrightarrow{i_{\overline{\Omega}_{\left(  1\right)  }^{\left(
p,1\right)  }\left(  M\right)  }^{\overline{\Omega}_{\left(  0\right)
}^{\left(  p,1\right)  }\left(  M\right)  }\times i_{\overline{\Omega
}_{\left(  1\right)  }^{\left(  q,1\right)  }\left(  M\right)  }%
^{\overline{\Omega}_{\left(  0\right)  }^{\left(  q,1\right)  }\left(
M\right)  }}\\
& \overline{\Omega}_{\left(  0\right)  }^{\left(  p,1\right)  }\left(
M\right)  \times\overline{\Omega}_{\left(  0\right)  }^{\left(  q,1\right)
}\left(  M\right) \\
& \underrightarrow{\underline{\mathrm{\Pr od}}_{\left(  p,1\right)  ,\left(
q,1\right)  }^{M}}\\
& \mathbf{T}^{\mathcal{W}_{D^{2}}}\left[  \mathbf{T}^{\mathcal{W}_{D^{p+q}}%
}M\rightarrow M\right]
\end{align*}
is identical to the morphism
\begin{align*}
& \overline{\Omega}_{\left(  1\right)  }^{\left(  p,1\right)  }\left(
M\right)  \times\overline{\Omega}_{\left(  1\right)  }^{\left(  q,1\right)
}\left(  M\right) \\
& \underrightarrow{i_{\overline{\Omega}_{\left(  1\right)  }^{\left(
p,1\right)  }\left(  M\right)  }^{\overline{\Omega}_{\left(  0\right)
}^{\left(  p,1\right)  }\left(  M\right)  }\times i_{\overline{\Omega
}_{\left(  1\right)  }^{\left(  q,1\right)  }\left(  M\right)  }%
^{\overline{\Omega}_{\left(  0\right)  }^{\left(  q,1\right)  }\left(
M\right)  }}\\
& \overline{\Omega}_{\left(  0\right)  }^{\left(  p,1\right)  }\left(
M\right)  \times\overline{\Omega}_{\left(  0\right)  }^{\left(  q,1\right)
}\left(  M\right) \\
& \underrightarrow{\overline{\mathrm{\Pr od}}_{\left(  p,1\right)  ,\left(
q,1\right)  }^{M}}\\
& \mathbf{T}^{\mathcal{W}_{D^{2}}}\left[  \mathbf{T}^{\mathcal{W}_{D^{p+q}}%
}M\rightarrow M\right]
\end{align*}
so that their compositions with the morphism
\begin{align*}
& \mathbf{T}^{\mathcal{W}_{D^{2}}}\left[  \mathbf{T}^{\mathcal{W}_{D^{p+q}}%
}M\rightarrow M\right] \\
& \underrightarrow{\alpha_{\mathcal{W}_{\left(  d_{1},d_{2}\right)  \in
D(2)\mapsto\left(  d_{1},d_{2}\right)  \in D^{2}}}\left(  \left[
\mathbf{T}^{\mathcal{W}_{D^{p+q}}}M\rightarrow M\right]  \right)  }\\
& \mathbf{T}^{\mathcal{W}_{D(2)}}\left[  \mathbf{T}^{\mathcal{W}_{D^{p+q}}%
}M\rightarrow M\right]
\end{align*}
should evidently be identical.
\end{proof}

\begin{corollary}
\label{t5.2.1'}The morphism
\begin{align*}
& \overline{\Omega}_{\left(  1\right)  }^{\left(  p,1\right)  }\left(
M\right)  \times\overline{\Omega}_{\left(  1\right)  }^{\left(  q,1\right)
}\left(  M\right) \\
& \underrightarrow{i_{\overline{\Omega}_{\left(  1\right)  }^{\left(
p,1\right)  }\left(  M\right)  }^{\overline{\Omega}_{\left(  0\right)
}^{\left(  p,1\right)  }\left(  M\right)  }\times i_{\overline{\Omega
}_{\left(  1\right)  }^{\left(  q,1\right)  }\left(  M\right)  }%
^{\overline{\Omega}_{\left(  0\right)  }^{\left(  q,1\right)  }\left(
M\right)  }}\\
& \overline{\Omega}_{\left(  0\right)  }^{\left(  p,1\right)  }\left(
M\right)  \times\overline{\Omega}_{\left(  0\right)  }^{\left(  q,1\right)
}\left(  M\right) \\
& \underrightarrow{\left(  \underline{\mathrm{\Pr od}}_{\left(  p,1\right)
,\left(  q,1\right)  }^{M},\overline{\mathrm{\Pr od}}_{\left(  p,1\right)
,\left(  q,1\right)  }^{M}\right)  }\\
& \mathbf{T}^{\mathcal{W}_{D^{2}}}\left[  \mathbf{T}^{\mathcal{W}_{D^{p+q}}%
}M\rightarrow M\right]  \times\mathbf{T}^{\mathcal{W}_{D^{2}}}\left[
\mathbf{T}^{\mathcal{W}_{D^{p+q}}}M\rightarrow M\right]
\end{align*}
is to be factored through the canonical injection
\begin{align*}
& \mathbf{T}^{\mathcal{W}_{D^{2}}}\left[  \mathbf{T}^{\mathcal{W}_{D^{p+q}}%
}M\rightarrow M\right]  \times_{\mathbf{T}^{\mathcal{W}_{D(2)}}\left[
\mathbf{T}^{\mathcal{W}_{D^{p+q}}}M\rightarrow M\right]  }\mathbf{T}%
^{\mathcal{W}_{D^{2}}}\left[  \mathbf{T}^{\mathcal{W}_{D^{p+q}}}M\rightarrow
M\right] \\
& \rightarrow\mathbf{T}^{\mathcal{W}_{D^{2}}}\left[  \mathbf{T}^{\mathcal{W}%
_{D^{p+q}}}M\rightarrow M\right]  \times\mathbf{T}^{\mathcal{W}_{D^{2}}%
}\left[  \mathbf{T}^{\mathcal{W}_{D^{p+q}}}M\rightarrow M\right]
\end{align*}

\end{corollary}

It is easy to see that

\begin{proposition}
\label{t5.2.2}The factored morphism
\begin{align*}
& \overline{\Omega}_{\left(  1\right)  }^{\left(  p,1\right)  }\left(
M\right)  \times\overline{\Omega}_{\left(  1\right)  }^{\left(  q,1\right)
}\left(  M\right) \\
& \rightarrow\mathbf{T}^{\mathcal{W}_{D^{2}}}\left[  \mathbf{T}^{\mathcal{W}%
_{D^{p+q}}}M\rightarrow M\right]  \times_{\mathbf{T}^{\mathcal{W}_{D(2)}%
}\left[  \mathbf{T}^{\mathcal{W}_{D^{p+q}}}M\rightarrow M\right]  }%
\mathbf{T}^{\mathcal{W}_{D^{2}}}\left[  \mathbf{T}^{\mathcal{W}_{D^{p+q}}%
}M\rightarrow M\right]
\end{align*}
in Corollary \ref{t5.2.1'} followed by the morphism
\begin{align*}
& \mathbf{T}^{\mathcal{W}_{D^{2}}}\left[  \mathbf{T}^{\mathcal{W}_{D^{p+q}}%
}M\rightarrow M\right]  \times_{\mathbf{T}^{\mathcal{W}_{D(2)}}\left[
\mathbf{T}^{\mathcal{W}_{D^{p+q}}}M\rightarrow M\right]  }\mathbf{T}%
^{\mathcal{W}_{D^{2}}}\left[  \mathbf{T}^{\mathcal{W}_{D^{p+q}}}M\rightarrow
M\right] \\
& \underrightarrow{\zeta^{\overset{\cdot}{-}}\left(  \left[  \mathbf{T}%
^{\mathcal{W}_{D^{p+q}}}M\rightarrow M\right]  \right)  }\\
& \mathbf{T}^{\mathcal{W}_{D}}\left[  \mathbf{T}^{\mathcal{W}_{D^{p+q}}%
}M\rightarrow M\right]
\end{align*}
is to be factored uniquely into a morphism
\[
\zeta_{p,q}^{L_{1}}\left(  M\right)  :\overline{\Omega}_{\left(  1\right)
}^{\left(  p,1\right)  }\left(  M\right)  \times\overline{\Omega}_{\left(
1\right)  }^{\left(  q,1\right)  }\left(  M\right)  \rightarrow\overline
{\Omega}_{\left(  1\right)  }^{\left(  p+q,1\right)  }\left(  M\right)
\]
and the canonical injection
\[
i_{\overline{\Omega}_{\left(  1\right)  }^{\left(  ^{p+q},1\right)  }\left(
M\right)  }^{\mathbf{T}^{\mathcal{W}_{D}}\left[  \mathbf{T}^{\mathcal{W}%
_{D^{p+q}}}M\rightarrow M\right]  }:\overline{\Omega}_{\left(  1\right)
}^{\left(  p+q,1\right)  }\left(  M\right)  \rightarrow\mathbf{T}%
^{\mathcal{W}_{D}}\left[  \mathbf{T}^{\mathcal{W}_{D^{p+q}}}M\rightarrow
M\right]
\]

\end{proposition}

\begin{notation}
Given $\xi\in\left[  M\otimes\mathcal{W}_{D^{p}}\rightarrow M\right]
\otimes\mathcal{W}_{D^{n}}$ and $\sigma\in\mathbb{S}_{p}$, $\xi^{\sigma} $
denotes
\[
\left(  \left(  \,\right)  _{\left[  M\otimes\mathcal{W}_{D^{p}}\rightarrow
M\right]  }^{\sigma}\otimes\mathrm{id}_{\mathcal{W}_{D^{n}}}\right)  \left(
\xi\right)
\]
where $\left(  \,\right)  _{\left[  M\otimes\mathcal{W}_{D^{p}}\rightarrow
M\right]  }^{\sigma}:\left[  M\otimes\mathcal{W}_{D^{p}}\rightarrow M\right]
\rightarrow\left[  M\otimes\mathcal{W}_{D^{p}}\rightarrow M\right]  $ denotes
the operation
\[
\eta\in\left[  M\otimes\mathcal{W}_{D^{p}}\rightarrow M\right]  \mapsto
\eta\circ\left(  \mathrm{id}_{M}\otimes\mathcal{W}_{\left(  d_{1}%
,...,d_{p}\right)  \in D^{p}\mapsto\left(  d_{\sigma\left(  1\right)
},...,d_{\sigma\left(  p\right)  }\right)  \in D^{p}}\right)
\]

\end{notation}

We will show that the morphism
\[
\zeta_{p,q}^{L_{1}}\left(  M\right)  :\overline{\Omega}_{\left(  1\right)
}^{\left(  p,1\right)  }\left(  M\right)  \times\overline{\Omega}_{\left(
1\right)  }^{\left(  q,1\right)  }\left(  M\right)  \rightarrow\overline
{\Omega}_{\left(  1\right)  }^{\left(  p+q,1\right)  }\left(  M\right)
\]
is antisymmetric in the following sense.

\begin{proposition}
\label{t5.2.3}The morphism
\begin{align}
& \overline{\Omega}_{\left(  1\right)  }^{\left(  p,1\right)  }\left(
M\right)  \times\overline{\Omega}_{\left(  1\right)  }^{\left(  q,1\right)
}\left(  M\right) \nonumber\\
& \underrightarrow{\zeta_{p,q}^{L_{1}}\left(  M\right)  }\nonumber\\
& \overline{\Omega}_{\left(  1\right)  }^{\left(  p+q,1\right)  }\left(
M\right) \label{5.2.3.1}%
\end{align}
and the morphism
\begin{align}
& \overline{\Omega}_{\left(  1\right)  }^{\left(  p,1\right)  }\left(
M\right)  \times\overline{\Omega}_{\left(  1\right)  }^{\left(  q,1\right)
}\left(  M\right) \nonumber\\
& =\overline{\Omega}_{\left(  1\right)  }^{\left(  q,1\right)  }\left(
M\right)  \times\overline{\Omega}_{\left(  1\right)  }^{\left(  p,1\right)
}\left(  M\right) \nonumber\\
& \underrightarrow{\zeta_{q,p}^{L_{1}}\left(  M\right)  }\nonumber\\
& \overline{\Omega}_{\left(  1\right)  }^{\left(  p+q,1\right)  }\left(
M\right) \nonumber\\
& \underrightarrow{\left(  \cdot^{\sigma_{p,q}}\right)  _{\overline{\Omega
}_{\left(  1\right)  }^{\left(  p+q,1\right)  }\left(  M\right)  }}\nonumber\\
& \overline{\Omega}_{\left(  1\right)  }^{\left(  p+q,1\right)  }\left(
M\right) \label{5.2.3.2}%
\end{align}
sum up only to vanish.
\end{proposition}

\begin{proof}
This follows easily from Propositions \ref{t3.2} and \ref{t5.1.1}.It is easy
to see that the morphism
\begin{align*}
& \overline{\Omega}_{\left(  1\right)  }^{\left(  p,1\right)  }\left(
M\right)  \times\overline{\Omega}_{\left(  1\right)  }^{\left(  q,1\right)
}\left(  M\right) \\
& \underrightarrow{i_{\overline{\Omega}_{\left(  1\right)  }^{\left(
p,1\right)  }\left(  M\right)  }^{\overline{\Omega}_{\left(  0\right)
}^{\left(  p,1\right)  }\left(  M\right)  }\times i_{\overline{\Omega
}_{\left(  1\right)  }^{\left(  q,1\right)  }\left(  M\right)  }%
^{\overline{\Omega}_{\left(  0\right)  }^{\left(  q,1\right)  }\left(
M\right)  }}\\
& \overline{\Omega}_{\left(  0\right)  }^{\left(  q,1\right)  }\left(
M\right)  \times\overline{\Omega}_{\left(  0\right)  }^{\left(  p,1\right)
}\left(  M\right) \\
& \underrightarrow{\left(  \underline{\mathrm{\Pr od}}_{\left(  q,1\right)
,\left(  p,1\right)  }^{M},\overline{\mathrm{\Pr od}}_{\left(  q,1\right)
,\left(  p,1\right)  }^{M}\right)  }\\
& \mathbf{T}^{\mathcal{W}_{D^{2}}}\left[  \mathbf{T}^{\mathcal{W}_{D^{p+q}}%
}M\rightarrow M\right]  \times_{\mathbf{T}^{\mathcal{W}_{D(2)}}\left[
\mathbf{T}^{\mathcal{W}_{D^{p+q}}}M\rightarrow M\right]  }\mathbf{T}%
^{\mathcal{W}_{D^{2}}}\left[  \mathbf{T}^{\mathcal{W}_{D^{p+q}}}M\rightarrow
M\right] \\
& \underrightarrow{\zeta^{\overset{\cdot}{-}}\left(  \left[  \mathbf{T}%
^{\mathcal{W}_{D^{p+q}}}M\rightarrow M\right]  \right)  }\\
& \overline{\Omega}_{\left(  0\right)  }^{\left(  p+q,1\right)  }\left(
M\right) \\
& \underrightarrow{\left(  \cdot^{\sigma_{p,q}}\right)  _{\overline{\Omega
}_{\left(  0\right)  }^{\left(  p+q,1\right)  }\left(  M\right)  }}\\
& \overline{\Omega}_{\left(  0\right)  }^{\left(  p+q,1\right)  }\left(
M\right)
\end{align*}
is identical to the morphism
\begin{align*}
& \overline{\Omega}_{\left(  1\right)  }^{\left(  p,1\right)  }\left(
M\right)  \times\overline{\Omega}_{\left(  1\right)  }^{\left(  q,1\right)
}\left(  M\right) \\
& \underrightarrow{i_{\overline{\Omega}_{\left(  1\right)  }^{\left(
p,1\right)  }\left(  M\right)  }^{\overline{\Omega}_{\left(  0\right)
}^{\left(  p,1\right)  }\left(  M\right)  }\times i_{\overline{\Omega
}_{\left(  1\right)  }^{\left(  q,1\right)  }\left(  M\right)  }%
^{\overline{\Omega}_{\left(  0\right)  }^{\left(  q,1\right)  }\left(
M\right)  }}\\
& \overline{\Omega}_{\left(  0\right)  }^{\left(  q,1\right)  }\left(
M\right)  \times\overline{\Omega}_{\left(  0\right)  }^{\left(  p,1\right)
}\left(  M\right) \\
& \underrightarrow{\left(  \left(  \cdot^{\sigma_{p,q}}\right)  _{\overline
{\Omega}_{\left(  0\right)  }^{\left(  p+q,1\right)  }\left(  M\right)  }%
\circ\underline{\mathrm{\Pr od}}_{\left(  q,1\right)  ,\left(  p,1\right)
}^{M},\left(  \cdot^{\sigma_{p,q}}\right)  _{\overline{\Omega}_{\left(
0\right)  }^{\left(  p+q,1\right)  }\left(  M\right)  }\circ\overline
{\mathrm{\Pr od}}_{\left(  q,1\right)  ,\left(  p,1\right)  }^{M}\right)  }\\
& \mathbf{T}^{\mathcal{W}_{D^{2}}}\left[  \mathbf{T}^{\mathcal{W}_{D^{p+q}}%
}M\rightarrow M\right]  \times_{\mathbf{T}^{\mathcal{W}_{D(2)}}\left[
\mathbf{T}^{\mathcal{W}_{D^{p+q}}}M\rightarrow M\right]  }\mathbf{T}%
^{\mathcal{W}_{D^{2}}}\left[  \mathbf{T}^{\mathcal{W}_{D^{p+q}}}M\rightarrow
M\right] \\
& \underrightarrow{\zeta^{\overset{\cdot}{-}}\left(  \left[  \mathbf{T}%
^{\mathcal{W}_{D^{p+q}}}M\rightarrow M\right]  \right)  }\\
& \overline{\Omega}_{\left(  0\right)  }^{\left(  p+q,1\right)  }\left(
M\right)
\end{align*}
However we know well by Proposition \ref{t5.1.1} that
\[
\]
This follows from Propositions 4 and 6 in \S 3.4 of Lavendhomme \cite{lav}.
More specifically we have
\begin{align*}
&  \left[  \xi_{1},\xi_{2}\right]  _{L}+\left(  \left[  \xi_{2},\xi
_{1}\right]  _{L}\right)  ^{\sigma_{p,q}}\\
&  =(\xi_{1}\widetilde{\circledast}\xi_{2}\overset{\cdot}{-}\xi_{1}%
\circledast\xi_{2})+(\left(  \xi_{2}\widetilde{\circledast}\xi_{1}\right)
^{\sigma_{p,q}}\overset{\cdot}{-}\left(  \xi_{2}\circledast\xi_{1}\right)
^{\sigma_{p,q}})\\
&  =(\xi_{1}\widetilde{\circledast}\xi_{2}\overset{\cdot}{-}\xi_{1}%
\circledast\xi_{2})+\left(  \xi_{1}\circledast\xi_{2}\overset{\cdot}{-}\xi
_{1}\widetilde{\circledast}\xi_{2}\right) \\
&  \text{[By Proposition \ref{t4.1.1}]}\\
&  =0\text{ \ }%
\end{align*}

\end{proof}

We have the following Jacobi identity.

\begin{theorem}
\label{t5.2.4}The three morphisms
\begin{align}
& \overline{\Omega}_{\left(  1\right)  }^{\left(  p,1\right)  }\left(
M\right)  \times\overline{\Omega}_{\left(  1\right)  }^{\left(  q,1\right)
}\left(  M\right)  \times\overline{\Omega}_{\left(  1\right)  }^{\left(
r,1\right)  }\left(  M\right) \nonumber\\
& \underrightarrow{\mathrm{id}_{\overline{\Omega}_{\left(  1\right)
}^{\left(  p,1\right)  }\left(  M\right)  }\times\zeta_{q,r}^{L_{1}}\left(
M\right)  }\nonumber\\
& \overline{\Omega}_{\left(  1\right)  }^{\left(  p,1\right)  }\left(
M\right)  \times\overline{\Omega}_{\left(  1\right)  }^{\left(  q+r,1\right)
}\left(  M\right) \nonumber\\
& \underrightarrow{\zeta_{p,q+r}^{L_{1}}\left(  M\right)  }\nonumber\\
& \overline{\Omega}_{\left(  1\right)  }^{\left(  p+q+r,1\right)  }\left(
M\right)  \text{,}\label{5.2.4.1}%
\end{align}
\begin{align}
& \overline{\Omega}_{\left(  1\right)  }^{\left(  p,1\right)  }\left(
M\right)  \times\overline{\Omega}_{\left(  1\right)  }^{\left(  q,1\right)
}\left(  M\right)  \times\overline{\Omega}_{\left(  1\right)  }^{\left(
r,1\right)  }\left(  M\right) \nonumber\\
& =\overline{\Omega}_{\left(  1\right)  }^{\left(  q,1\right)  }\left(
M\right)  \times\overline{\Omega}_{\left(  1\right)  }^{\left(  r,1\right)
}\left(  M\right)  \times\overline{\Omega}_{\left(  1\right)  }^{\left(
p,1\right)  }\left(  M\right) \nonumber\\
& \underrightarrow{\mathrm{id}_{\overline{\Omega}_{\left(  1\right)
}^{\left(  q,1\right)  }\left(  M\right)  }\times\zeta_{r,p}^{L_{1}}\left(
M\right)  }\nonumber\\
& \overline{\Omega}_{\left(  1\right)  }^{\left(  q,1\right)  }\left(
M\right)  \times\overline{\Omega}_{\left(  1\right)  }^{\left(  r+p,1\right)
}\left(  M\right) \nonumber\\
& \underrightarrow{\zeta_{q,r+p}^{L_{1}}\left(  M\right)  }\nonumber\\
& \overline{\Omega}_{\left(  1\right)  }^{\left(  p+q+r,1\right)  }\left(
M\right) \nonumber\\
& \underrightarrow{\left(  \cdot^{\sigma_{p,q+r}}\right)  _{\overline{\Omega
}_{\left(  1\right)  }^{\left(  p+q+r,1\right)  }\left(  M\right)  }%
}\nonumber\\
& \overline{\Omega}_{\left(  1\right)  }^{\left(  p+q+r,1\right)  }\left(
M\right) \label{5.2.4.2}%
\end{align}
and
\begin{align}
& \overline{\Omega}_{\left(  1\right)  }^{\left(  p,1\right)  }\left(
M\right)  \times\overline{\Omega}_{\left(  1\right)  }^{\left(  q,1\right)
}\left(  M\right)  \times\overline{\Omega}_{\left(  1\right)  }^{\left(
r,1\right)  }\left(  M\right) \nonumber\\
& =\overline{\Omega}_{\left(  1\right)  }^{\left(  r,1\right)  }\left(
M\right)  \times\overline{\Omega}_{\left(  1\right)  }^{\left(  p,1\right)
}\left(  M\right)  \times\overline{\Omega}_{\left(  1\right)  }^{\left(
q,1\right)  }\left(  M\right) \nonumber\\
& \underrightarrow{\mathrm{id}_{\overline{\Omega}_{\left(  1\right)
}^{\left(  r,1\right)  }\left(  M\right)  }\times\zeta_{p,q}^{L_{1}}\left(
M\right)  }\nonumber\\
& \overline{\Omega}_{\left(  1\right)  }^{\left(  r,1\right)  }\left(
M\right)  \times\overline{\Omega}_{\left(  1\right)  }^{\left(  p+q,1\right)
}\left(  M\right) \nonumber\\
& \underrightarrow{\zeta_{r,p+q}^{L_{1}}\left(  M\right)  }\nonumber\\
& \overline{\Omega}_{\left(  1\right)  }^{\left(  p+q+r,1\right)  }\left(
M\right) \nonumber\\
& \underrightarrow{\left(  \cdot^{\sigma_{r,p+q}}\right)  _{\overline{\Omega
}_{\left(  1\right)  }^{\left(  p+q+r,1\right)  }\left(  M\right)  }%
}\nonumber\\
& \overline{\Omega}_{\left(  1\right)  }^{\left(  p+q+r,1\right)  }\left(
M\right) \label{5.2.4.3}%
\end{align}
sum up only to vanish.
\end{theorem}

In order to establish the above theorem, we need the following simple lemma,
which is a tiny generalization of Proposition 2.6 of \cite{nishi-a}.

\begin{lemma}
\label{t5.2.5}The morphism
\begin{align*}
& \overline{\Omega}_{\left(  1\right)  }^{\left(  p,1\right)  }\left(
M\right)  \times\left(  \overline{\Omega}_{\left(  1\right)  }^{\left(
q,2\right)  }\left(  M\right)  \times_{\mathbf{T}^{\mathcal{W}_{D(2)}}\left[
\mathbf{T}^{\mathcal{W}_{D^{q}}}M\rightarrow M\right]  }\overline{\Omega
}_{\left(  1\right)  }^{\left(  q,2\right)  }\left(  M\right)  \right) \\
& \underrightarrow{\mathrm{id}_{\overline{\Omega}_{\left(  1\right)
}^{\left(  p,1\right)  }\left(  M\right)  }\times\zeta^{\overset{\cdot}{-}%
}\left(  \left[  \mathbf{T}^{\mathcal{W}_{D^{q}}}M\rightarrow M\right]
\right)  }\\
& \overline{\Omega}_{\left(  1\right)  }^{\left(  p,1\right)  }\left(
M\right)  \times\overline{\Omega}_{\left(  1\right)  }^{\left(  q,1\right)
}\left(  M\right) \\
& \underrightarrow{\mathrm{\Pr od}_{\left(  p,1\right)  ,\left(  q,1\right)
}^{M}}\\
& \overline{\Omega}_{\left(  1\right)  }^{\left(  p+q,2\right)  }\left(
M\right)
\end{align*}
is identical to
\begin{align*}
& \overline{\Omega}_{\left(  1\right)  }^{\left(  p,1\right)  }\left(
M\right)  \times\left(  \overline{\Omega}_{\left(  1\right)  }^{\left(
q,2\right)  }\left(  M\right)  \times_{\mathbf{T}^{\mathcal{W}_{D(2)}}\left[
\mathbf{T}^{\mathcal{W}_{D^{q}}}M\rightarrow M\right]  }\overline{\Omega
}_{\left(  1\right)  }^{\left(  q,2\right)  }\left(  M\right)  \right) \\
& =\left(  \overline{\Omega}_{\left(  1\right)  }^{\left(  p,1\right)
}\left(  M\right)  \times\overline{\Omega}_{\left(  1\right)  }^{\left(
q,2\right)  }\left(  M\right)  \right)  \times_{\mathbf{T}^{\mathcal{W}%
_{D(2)}}\left[  \mathbf{T}^{\mathcal{W}_{D^{q}}}M\rightarrow M\right]
}\left(  \overline{\Omega}_{\left(  1\right)  }^{\left(  p,1\right)  }\left(
M\right)  \times\overline{\Omega}_{\left(  1\right)  }^{\left(  q,2\right)
}\left(  M\right)  \right) \\
& \underrightarrow{\mathrm{\Pr od}_{\left(  p,1\right)  ,\left(  q,2\right)
}^{M}\times_{\mathbf{T}^{\mathcal{W}_{D(2)}}\left[  \mathbf{T}^{\mathcal{W}%
_{D^{q}}}M\rightarrow M\right]  }\mathrm{\Pr od}_{\left(  p,1\right)  ,\left(
q,2\right)  }^{M}}\\
& \overline{\Omega}_{\left(  1\right)  }^{\left(  p+q,3\right)  }\left(
M\right)  \times_{\mathbf{T}^{\mathcal{W}_{D^{3}\{(2,3)\}}}\left[
\mathbf{T}^{\mathcal{W}_{D^{p+q}}}M\rightarrow M\right]  }\overline{\Omega
}_{\left(  1\right)  }^{\left(  p+q,3\right)  }\left(  M\right) \\
& \underrightarrow{\zeta^{\underset{1}{\overset{\cdot}{-}}}\left(  \left[
\mathbf{T}^{\mathcal{W}_{D^{p+q}}}M\rightarrow M\right]  \right)  }\\
& \overline{\Omega}_{\left(  1\right)  }^{\left(  p+q,2\right)  }\left(
M\right)  \text{,}%
\end{align*}
while the morphism
\begin{align*}
& \left(  \overline{\Omega}_{\left(  1\right)  }^{\left(  p,2\right)  }\left(
M\right)  \times_{\mathbf{T}^{\mathcal{W}_{D(2)}}\left[  \mathbf{T}%
^{\mathcal{W}_{D^{p}}}M\rightarrow M\right]  }\overline{\Omega}_{\left(
1\right)  }^{\left(  p,2\right)  }\left(  M\right)  \right)  \times
\overline{\Omega}_{\left(  1\right)  }^{\left(  q,1\right)  }\left(  M\right)
\\
& \underrightarrow{\zeta^{\overset{\cdot}{-}}\left(  \left[  \mathbf{T}%
^{\mathcal{W}_{D^{p}}}M\rightarrow M\right]  \right)  \times\mathrm{id}%
_{\overline{\Omega}_{\left(  1\right)  }^{\left(  q,1\right)  }\left(
M\right)  }}\\
& \overline{\Omega}_{\left(  1\right)  }^{\left(  p,1\right)  }\left(
M\right)  \times\overline{\Omega}_{\left(  1\right)  }^{\left(  q,1\right)
}\left(  M\right) \\
& \underrightarrow{\mathrm{\Pr od}_{\left(  p,1\right)  ,\left(  q,1\right)
}^{M}}\\
& \overline{\Omega}_{\left(  1\right)  }^{\left(  p+q,2\right)  }\left(
M\right)
\end{align*}
is identical to
\begin{align*}
& \left(  \overline{\Omega}_{\left(  1\right)  }^{\left(  p,2\right)  }\left(
M\right)  \times_{\mathbf{T}^{\mathcal{W}_{D(2)}}\left[  \mathbf{T}%
^{\mathcal{W}_{D^{p}}}M\rightarrow M\right]  }\overline{\Omega}_{\left(
1\right)  }^{\left(  p,2\right)  }\left(  M\right)  \right)  \times
\overline{\Omega}_{\left(  1\right)  }^{\left(  q,1\right)  }\left(  M\right)
\\
& =\left(  \overline{\Omega}_{\left(  1\right)  }^{\left(  p,2\right)
}\left(  M\right)  \times\overline{\Omega}_{\left(  1\right)  }^{\left(
q,1\right)  }\left(  M\right)  \right)  \times_{\mathbf{T}^{\mathcal{W}%
_{D(2)}}\left[  \mathbf{T}^{\mathcal{W}_{D^{q}}}M\rightarrow M\right]
}\left(  \overline{\Omega}_{\left(  1\right)  }^{\left(  p,2\right)  }\left(
M\right)  \times\overline{\Omega}_{\left(  1\right)  }^{\left(  q,1\right)
}\left(  M\right)  \right) \\
& \underrightarrow{\mathrm{\Pr od}_{\left(  p,2\right)  ,\left(  q,1\right)
}^{M}\times_{\mathbf{T}^{\mathcal{W}_{D(2)}}\left[  \mathbf{T}^{\mathcal{W}%
_{D^{q}}}M\rightarrow M\right]  }\mathrm{\Pr od}_{\left(  p,2\right)  ,\left(
q,1\right)  }^{M}}\\
& \overline{\Omega}_{\left(  1\right)  }^{\left(  p+q,3\right)  }\left(
M\right)  \times_{\mathbf{T}^{\mathcal{W}_{D^{3}\{1,(2)\}}}\left[
\mathbf{T}^{\mathcal{W}_{D^{p+q}}}M\rightarrow M\right]  }\overline{\Omega
}_{\left(  1\right)  }^{\left(  p+q,3\right)  }\left(  M\right) \\
& \underrightarrow{\zeta^{\underset{3}{\overset{\cdot}{-}}}\left(  \left[
\mathbf{T}^{\mathcal{W}_{D^{p+q}}}M\rightarrow M\right]  \right)  }\\
& \overline{\Omega}_{\left(  1\right)  }^{\left(  p+q,2\right)  }\left(
M\right)
\end{align*}

\end{lemma}

\begin{notation}
For the sake of the proof of Theorem \ref{t5.2.4}, we introduce the following
six notations:

\begin{enumerate}
\item We write $\xi_{123}$\ for the morphism
\begin{align*}
& \overline{\Omega}_{\left(  1\right)  }^{\left(  p,1\right)  }\left(
M\right)  \times\overline{\Omega}_{\left(  1\right)  }^{\left(  q,1\right)
}\left(  M\right)  \times\overline{\Omega}_{\left(  1\right)  }^{\left(
r,1\right)  }\left(  M\right) \\
& \underrightarrow{\underline{\mathrm{\Pr od}}_{\left(  p,1\right)  ,\left(
q,1\right)  ,\left(  r,1\right)  }^{M}}\\
& \overline{\Omega}_{\left(  1\right)  }^{\left(  p+q+r,3\right)  }\left(
M\right)
\end{align*}

\item We write $\xi_{132}$\ for the morphism
\begin{align*}
& \overline{\Omega}_{\left(  1\right)  }^{\left(  p,1\right)  }\left(
M\right)  \times\overline{\Omega}_{\left(  1\right)  }^{\left(  q,1\right)
}\left(  M\right)  \times\overline{\Omega}_{\left(  1\right)  }^{\left(
r,1\right)  }\left(  M\right) \\
& \underrightarrow{\mathrm{id}_{\overline{\Omega}_{\left(  1\right)
}^{\left(  p,1\right)  }\left(  M\right)  }\times\overline{\mathrm{\Pr od}%
}_{\left(  q,1\right)  ,\left(  r,1\right)  }^{M}}\\
& \overline{\Omega}_{\left(  1\right)  }^{\left(  p,1\right)  }\left(
M\right)  \times\overline{\Omega}_{\left(  1\right)  }^{\left(  q+r,1\right)
}\left(  M\right) \\
& \underrightarrow{\underline{\mathrm{\Pr od}}_{\left(  p,1\right)  ,\left(
q+r,1\right)  }^{M}}\\
& \overline{\Omega}_{\left(  1\right)  }^{\left(  p+q+r,3\right)  }\left(
M\right)
\end{align*}

\item We write $\xi_{213}$\ for the morphism
\begin{align*}
& \overline{\Omega}_{\left(  1\right)  }^{\left(  p,1\right)  }\left(
M\right)  \times\overline{\Omega}_{\left(  1\right)  }^{\left(  q,1\right)
}\left(  M\right)  \times\overline{\Omega}_{\left(  1\right)  }^{\left(
r,1\right)  }\left(  M\right) \\
& \underrightarrow{\overline{\mathrm{\Pr od}}_{\left(  p,1\right)  ,\left(
q,1\right)  }^{M}\times\mathrm{id}_{\overline{\Omega}_{\left(  1\right)
}^{\left(  r,1\right)  }\left(  M\right)  }}\\
& \overline{\Omega}_{\left(  1\right)  }^{\left(  p+q,1\right)  }\left(
M\right)  \times\overline{\Omega}_{\left(  1\right)  }^{\left(  r,1\right)
}\left(  M\right) \\
& \underrightarrow{\underline{\mathrm{\Pr od}}_{\left(  p+q,1\right)  ,\left(
r,1\right)  }^{M}}\\
& \overline{\Omega}_{\left(  1\right)  }^{\left(  p+q+r,3\right)  }\left(
M\right)
\end{align*}

\item We write $\xi_{231}$\ for the morphism
\begin{align*}
& \overline{\Omega}_{\left(  1\right)  }^{\left(  p,1\right)  }\left(
M\right)  \times\overline{\Omega}_{\left(  1\right)  }^{\left(  q,1\right)
}\left(  M\right)  \times\overline{\Omega}_{\left(  1\right)  }^{\left(
r,1\right)  }\left(  M\right) \\
& \underrightarrow{\mathrm{id}_{\overline{\Omega}_{\left(  1\right)
}^{\left(  p,1\right)  }\left(  M\right)  }\times\underline{\mathrm{\Pr od}%
}_{\left(  q,1\right)  ,\left(  r,1\right)  }^{M}}\\
& \overline{\Omega}_{\left(  1\right)  }^{\left(  p,1\right)  }\left(
M\right)  \times\overline{\Omega}_{\left(  1\right)  }^{\left(  q+r,1\right)
}\left(  M\right) \\
& \underrightarrow{\overline{\mathrm{\Pr od}}_{\left(  p,1\right)  ,\left(
q+r,1\right)  }^{M}}\\
& \overline{\Omega}_{\left(  1\right)  }^{\left(  p+q+r,3\right)  }\left(
M\right)
\end{align*}

\item We write $\xi_{312}$\ for the morphism
\begin{align*}
& \overline{\Omega}_{\left(  1\right)  }^{\left(  p,1\right)  }\left(
M\right)  \times\overline{\Omega}_{\left(  1\right)  }^{\left(  q,1\right)
}\left(  M\right)  \times\overline{\Omega}_{\left(  1\right)  }^{\left(
r,1\right)  }\left(  M\right) \\
& \underrightarrow{\underline{\mathrm{\Pr od}}_{\left(  p,1\right)  ,\left(
q,1\right)  }^{M}\times\mathrm{id}_{\overline{\Omega}_{\left(  1\right)
}^{\left(  r,1\right)  }\left(  M\right)  }}\\
& \overline{\Omega}_{\left(  1\right)  }^{\left(  p+q,1\right)  }\left(
M\right)  \times\overline{\Omega}_{\left(  1\right)  }^{\left(  r,1\right)
}\left(  M\right) \\
& \underrightarrow{\overline{\mathrm{\Pr od}}_{\left(  p+q,1\right)  ,\left(
r,1\right)  }^{M}}\\
& \overline{\Omega}_{\left(  1\right)  }^{\left(  p+q+r,3\right)  }\left(
M\right)
\end{align*}

\item We write $\xi_{321}$\ for the morphism
\begin{align*}
& \overline{\Omega}_{\left(  1\right)  }^{\left(  p,1\right)  }\left(
M\right)  \times\overline{\Omega}_{\left(  1\right)  }^{\left(  q,1\right)
}\left(  M\right)  \times\overline{\Omega}_{\left(  1\right)  }^{\left(
r,1\right)  }\left(  M\right) \\
& \underrightarrow{\overline{\mathrm{\Pr od}}_{\left(  p,1\right)  ,\left(
q,1\right)  ,\left(  r,1\right)  }^{M}}\\
& \overline{\Omega}_{\left(  1\right)  }^{\left(  p+q+r,3\right)  }\left(
M\right)
\end{align*}

\end{enumerate}
\end{notation}

\begin{lemma}
\label{t5.2.6}We have the following statements:

\begin{enumerate}
\item The morphism
\begin{align*}
& \overline{\Omega}_{\left(  1\right)  }^{\left(  p,1\right)  }\left(
M\right)  \times\overline{\Omega}_{\left(  1\right)  }^{\left(  q,1\right)
}\left(  M\right)  \times\overline{\Omega}_{\left(  1\right)  }^{\left(
r,1\right)  }\left(  M\right) \\
& \underrightarrow{\left(  \xi_{123},\xi_{132}\right)  }\\
& \overline{\Omega}_{\left(  1\right)  }^{\left(  p+q+r,3\right)  }\left(
M\right)  \times\overline{\Omega}_{\left(  1\right)  }^{\left(
p+q+r,3\right)  }\left(  M\right)
\end{align*}
is to be factored uniquely through the canonical injection
\begin{align*}
& \overline{\Omega}_{\left(  1\right)  }^{\left(  p+q+r,3\right)  }\left(
M\right)  \times_{\mathbf{T}^{\mathcal{W}_{D^{3}\{(2,3)\}}}\left[
\mathbf{T}^{\mathcal{W}_{D^{p+q+r}}}M\rightarrow M\right]  }\overline{\Omega
}_{\left(  1\right)  }^{\left(  p+q+r,3\right)  }\left(  M\right) \\
& \rightarrow\overline{\Omega}_{\left(  1\right)  }^{\left(  p+q+r,3\right)
}\left(  M\right)  \times\overline{\Omega}_{\left(  1\right)  }^{\left(
p+q+r,3\right)  }\left(  M\right)
\end{align*}
into
\begin{align}
& \overline{\Omega}_{\left(  1\right)  }^{\left(  p,1\right)  }\left(
M\right)  \times\overline{\Omega}_{\left(  1\right)  }^{\left(  q,1\right)
}\left(  M\right)  \times\overline{\Omega}_{\left(  1\right)  }^{\left(
r,1\right)  }\left(  M\right) \nonumber\\
& \rightarrow\overline{\Omega}_{\left(  1\right)  }^{\left(  p+q+r,3\right)
}\left(  M\right)  \times_{\mathbf{T}^{\mathcal{W}_{D^{3}\{(2,3)\}}}\left[
\mathbf{T}^{\mathcal{W}_{D^{p+q+r}}}M\rightarrow M\right]  }\overline{\Omega
}_{\left(  1\right)  }^{\left(  p+q+r,3\right)  }\left(  M\right)
\label{5.2.6.1}%
\end{align}

\item The morphism
\begin{align*}
& \overline{\Omega}_{\left(  1\right)  }^{\left(  p,1\right)  }\left(
M\right)  \times\overline{\Omega}_{\left(  1\right)  }^{\left(  q,1\right)
}\left(  M\right)  \times\overline{\Omega}_{\left(  1\right)  }^{\left(
r,1\right)  }\left(  M\right) \\
& \underrightarrow{\left(  \xi_{231},\xi_{321}\right)  }\\
& \overline{\Omega}_{\left(  1\right)  }^{\left(  p+q+r,3\right)  }\left(
M\right)  \times\overline{\Omega}_{\left(  1\right)  }^{\left(
p+q+r,3\right)  }\left(  M\right)
\end{align*}
is to be factored uniquely through the canonical injection
\begin{align*}
& \overline{\Omega}_{\left(  1\right)  }^{\left(  p+q+r,3\right)  }\left(
M\right)  \times_{\mathbf{T}^{\mathcal{W}_{D^{3}\{(2,3)\}}}\left[
\mathbf{T}^{\mathcal{W}_{D^{p+q+r}}}M\rightarrow M\right]  }\overline{\Omega
}_{\left(  1\right)  }^{\left(  p+q+r,3\right)  }\left(  M\right) \\
& \rightarrow\overline{\Omega}_{\left(  1\right)  }^{\left(  p+q+r,3\right)
}\left(  M\right)  \times\overline{\Omega}_{\left(  1\right)  }^{\left(
p+q+r,3\right)  }\left(  M\right)
\end{align*}
into
\begin{align}
& \overline{\Omega}_{\left(  1\right)  }^{\left(  p,1\right)  }\left(
M\right)  \times\overline{\Omega}_{\left(  1\right)  }^{\left(  q,1\right)
}\left(  M\right)  \times\overline{\Omega}_{\left(  1\right)  }^{\left(
r,1\right)  }\left(  M\right) \nonumber\\
& \rightarrow\overline{\Omega}_{\left(  1\right)  }^{\left(  p+q+r,3\right)
}\left(  M\right)  \times_{\mathbf{T}^{\mathcal{W}_{D^{3}\{(2,3)\}}}\left[
\mathbf{T}^{\mathcal{W}_{Dp+q+r}}M\rightarrow M\right]  }\overline{\Omega
}_{\left(  1\right)  }^{\left(  p+q+r,3\right)  }\left(  M\right)
\label{5.2.6.2}%
\end{align}

\item The morphism
\begin{align*}
& \overline{\Omega}_{\left(  1\right)  }^{\left(  p,1\right)  }\left(
M\right)  \times\overline{\Omega}_{\left(  1\right)  }^{\left(  q,1\right)
}\left(  M\right)  \times\overline{\Omega}_{\left(  1\right)  }^{\left(
r,1\right)  }\left(  M\right) \\
& \underrightarrow{\left(  \xi_{231},\xi_{213}\right)  }\\
& \overline{\Omega}_{\left(  1\right)  }^{\left(  p+q+r,3\right)  }\left(
M\right)  \times\overline{\Omega}_{\left(  1\right)  }^{\left(
p+q+r,3\right)  }\left(  M\right)
\end{align*}
is to be factored uniquely through the canonical injection
\begin{align*}
& \overline{\Omega}_{\left(  1\right)  }^{\left(  p+q+r,3\right)  }\left(
M\right)  \times_{\mathbf{T}^{\mathcal{W}_{D^{3}\{(1,3)\}}}\left[
\mathbf{T}^{\mathcal{W}_{D^{p+q+r}}}M\rightarrow M\right]  }\overline{\Omega
}_{\left(  1\right)  }^{\left(  p+q+r,3\right)  }\left(  M\right) \\
& \rightarrow\overline{\Omega}_{\left(  1\right)  }^{\left(  p+q+r,3\right)
}\left(  M\right)  \times\overline{\Omega}_{\left(  1\right)  }^{\left(
p+q+r,3\right)  }\left(  M\right)
\end{align*}
into
\begin{align}
& \overline{\Omega}_{\left(  1\right)  }^{\left(  p,1\right)  }\left(
M\right)  \times\overline{\Omega}_{\left(  1\right)  }^{\left(  q,1\right)
}\left(  M\right)  \times\overline{\Omega}_{\left(  1\right)  }^{\left(
r,1\right)  }\left(  M\right) \nonumber\\
& \rightarrow\overline{\Omega}_{\left(  1\right)  }^{\left(  p+q+r,3\right)
}\left(  M\right)  \times_{\mathbf{T}^{\mathcal{W}_{D^{3}\{(1,3)\}}}\left[
\mathbf{T}^{\mathcal{W}_{Dp+q+r}}M\rightarrow M\right]  }\overline{\Omega
}_{\left(  1\right)  }^{\left(  p+q+r,3\right)  }\left(  M\right)
\label{5.2.6.3}%
\end{align}

\item The morphism
\begin{align*}
& \overline{\Omega}_{\left(  1\right)  }^{\left(  p,1\right)  }\left(
M\right)  \times\overline{\Omega}_{\left(  1\right)  }^{\left(  q,1\right)
}\left(  M\right)  \times\overline{\Omega}_{\left(  1\right)  }^{\left(
r,1\right)  }\left(  M\right) \\
& \underrightarrow{\left(  \xi_{312},\xi_{132}\right)  }\\
& \overline{\Omega}_{\left(  1\right)  }^{\left(  p+q+r,3\right)  }\left(
M\right)  \times\overline{\Omega}_{\left(  1\right)  }^{\left(
p+q+r,3\right)  }\left(  M\right)
\end{align*}
is to be factored uniquely through the canonical injection
\begin{align*}
& \overline{\Omega}_{\left(  1\right)  }^{\left(  p+q+r,3\right)  }\left(
M\right)  \times_{\mathbf{T}^{\mathcal{W}_{D^{3}\{(1,3)\}}}\left[
\mathbf{T}^{\mathcal{W}_{D^{p+q+r}}}M\rightarrow M\right]  }\overline{\Omega
}_{\left(  1\right)  }^{\left(  p+q+r,3\right)  }\left(  M\right) \\
& \rightarrow\overline{\Omega}_{\left(  1\right)  }^{\left(  p+q+r,3\right)
}\left(  M\right)  \times\overline{\Omega}_{\left(  1\right)  }^{\left(
p+q+r,3\right)  }\left(  M\right)
\end{align*}
into
\begin{align}
& \overline{\Omega}_{\left(  1\right)  }^{\left(  p,1\right)  }\left(
M\right)  \times\overline{\Omega}_{\left(  1\right)  }^{\left(  q,1\right)
}\left(  M\right)  \times\overline{\Omega}_{\left(  1\right)  }^{\left(
r,1\right)  }\left(  M\right) \nonumber\\
& \rightarrow\overline{\Omega}_{\left(  1\right)  }^{\left(  p+q+r,3\right)
}\left(  M\right)  \times_{\mathbf{T}^{\mathcal{W}_{D^{3}\{(1,3)\}}}\left[
\mathbf{T}^{\mathcal{W}_{Dp+q+r}}M\rightarrow M\right]  }\overline{\Omega
}_{\left(  1\right)  }^{\left(  p+q+r,3\right)  }\left(  M\right)
\label{5.2.6.4}%
\end{align}

\item The morphism
\begin{align*}
& \overline{\Omega}_{\left(  1\right)  }^{\left(  p,1\right)  }\left(
M\right)  \times\overline{\Omega}_{\left(  1\right)  }^{\left(  q,1\right)
}\left(  M\right)  \times\overline{\Omega}_{\left(  1\right)  }^{\left(
r,1\right)  }\left(  M\right) \\
& \underrightarrow{\left(  \xi_{312},\xi_{132}\right)  }\\
& \overline{\Omega}_{\left(  1\right)  }^{\left(  p+q+r,3\right)  }\left(
M\right)  \times\overline{\Omega}_{\left(  1\right)  }^{\left(
p+q+r,3\right)  }\left(  M\right)
\end{align*}
is to be factored uniquely through the canonical injection
\begin{align*}
& \overline{\Omega}_{\left(  1\right)  }^{\left(  p+q+r,3\right)  }\left(
M\right)  \times_{\mathbf{T}^{\mathcal{W}_{D^{3}\{(1,2)\}}}\left[
\mathbf{T}^{\mathcal{W}_{D^{p+q+r}}}M\rightarrow M\right]  }\overline{\Omega
}_{\left(  1\right)  }^{\left(  p+q+r,3\right)  }\left(  M\right) \\
& \rightarrow\overline{\Omega}_{\left(  1\right)  }^{\left(  p+q+r,3\right)
}\left(  M\right)  \times\overline{\Omega}_{\left(  1\right)  }^{\left(
p+q+r,3\right)  }\left(  M\right)
\end{align*}
into
\begin{align}
& \overline{\Omega}_{\left(  1\right)  }^{\left(  p,1\right)  }\left(
M\right)  \times\overline{\Omega}_{\left(  1\right)  }^{\left(  q,1\right)
}\left(  M\right)  \times\overline{\Omega}_{\left(  1\right)  }^{\left(
r,1\right)  }\left(  M\right) \nonumber\\
& \rightarrow\overline{\Omega}_{\left(  1\right)  }^{\left(  p+q+r,3\right)
}\left(  M\right)  \times_{\mathbf{T}^{\mathcal{W}_{D^{3}\{(1,2)\}}}\left[
\mathbf{T}^{\mathcal{W}_{Dp+q+r}}M\rightarrow M\right]  }\overline{\Omega
}_{\left(  1\right)  }^{\left(  p+q+r,3\right)  }\left(  M\right)
\label{5.2.6.5}%
\end{align}

\item The morphism
\begin{align*}
& \overline{\Omega}_{\left(  1\right)  }^{\left(  p,1\right)  }\left(
M\right)  \times\overline{\Omega}_{\left(  1\right)  }^{\left(  q,1\right)
}\left(  M\right)  \times\overline{\Omega}_{\left(  1\right)  }^{\left(
r,1\right)  }\left(  M\right) \\
& \underrightarrow{\left(  \xi_{312},\xi_{132}\right)  }\\
& \overline{\Omega}_{\left(  1\right)  }^{\left(  p+q+r,3\right)  }\left(
M\right)  \times\overline{\Omega}_{\left(  1\right)  }^{\left(
p+q+r,3\right)  }\left(  M\right)
\end{align*}
is to be factored uniquely through the canonical injection
\begin{align*}
& \overline{\Omega}_{\left(  1\right)  }^{\left(  p+q+r,3\right)  }\left(
M\right)  \times_{\mathbf{T}^{\mathcal{W}_{D^{3}\{(1,2)\}}}\left[
\mathbf{T}^{\mathcal{W}_{D^{p+q+r}}}M\rightarrow M\right]  }\overline{\Omega
}_{\left(  1\right)  }^{\left(  p+q+r,3\right)  }\left(  M\right) \\
& \rightarrow\overline{\Omega}_{\left(  1\right)  }^{\left(  p+q+r,3\right)
}\left(  M\right)  \times\overline{\Omega}_{\left(  1\right)  }^{\left(
p+q+r,3\right)  }\left(  M\right)
\end{align*}
into
\begin{align}
& \overline{\Omega}_{\left(  1\right)  }^{\left(  p,1\right)  }\left(
M\right)  \times\overline{\Omega}_{\left(  1\right)  }^{\left(  q,1\right)
}\left(  M\right)  \times\overline{\Omega}_{\left(  1\right)  }^{\left(
r,1\right)  }\left(  M\right) \nonumber\\
& \rightarrow\overline{\Omega}_{\left(  1\right)  }^{\left(  p+q+r,3\right)
}\left(  M\right)  \times_{\mathbf{T}^{\mathcal{W}_{D^{3}\{(1,2)\}}}\left[
\mathbf{T}^{\mathcal{W}_{Dp+q+r}}M\rightarrow M\right]  }\overline{\Omega
}_{\left(  1\right)  }^{\left(  p+q+r,3\right)  }\left(  M\right)
\label{5.2.6.6}%
\end{align}

\end{enumerate}
\end{lemma}

\begin{notation}
We introduce the following six notations:

\begin{enumerate}
\item The composition of (\ref{5.2.6.1}) with
\begin{align*}
& \overline{\Omega}_{\left(  1\right)  }^{\left(  p+q+r,3\right)  }\left(
M\right)  \times_{\mathbf{T}^{\mathcal{W}_{D^{3}\{(2,3)\}}}\left[
\mathbf{T}^{\mathcal{W}_{Dp+q+r}}M\rightarrow M\right]  }\overline{\Omega
}_{\left(  1\right)  }^{\left(  p+q+r,3\right)  }\left(  M\right) \\
& \underrightarrow{\zeta^{\underset{1}{\overset{\cdot}{-}}}\left(  \left[
\mathbf{T}^{\mathcal{W}_{Dp+q+r}}M\rightarrow M\right]  \right)  }\\
& \overline{\Omega}_{\left(  1\right)  }^{\left(  p+q+r,2\right)  }\left(
M\right)
\end{align*}
is denoted
\[
\zeta^{\ast_{123}\underset{1}{\overset{\cdot}{-}}\ast_{132}}%
\]

\item The composition of (\ref{5.2.6.2}) with
\begin{align*}
& \overline{\Omega}_{\left(  1\right)  }^{\left(  p+q+r,3\right)  }\left(
M\right)  \times_{\mathbf{T}^{\mathcal{W}_{D^{3}\{(2,3)\}}}\left[
\mathbf{T}^{\mathcal{W}_{Dp+q+r}}M\rightarrow M\right]  }\overline{\Omega
}_{\left(  1\right)  }^{\left(  p+q+r,3\right)  }\left(  M\right) \\
& \underrightarrow{\zeta^{\underset{1}{\overset{\cdot}{-}}}\left(  \left[
\mathbf{T}^{\mathcal{W}_{Dp+q+r}}M\rightarrow M\right]  \right)  }\\
& \overline{\Omega}_{\left(  1\right)  }^{\left(  p+q+r,2\right)  }\left(
M\right)
\end{align*}
is denoted
\[
\zeta^{\ast_{231}\underset{1}{\overset{\cdot}{-}}\ast_{321}}%
\]

\item The composition of (\ref{5.2.6.3}) with
\begin{align*}
& \overline{\Omega}_{\left(  1\right)  }^{\left(  p+q+r,3\right)  }\left(
M\right)  \times_{\mathbf{T}^{\mathcal{W}_{D^{3}\{(1,3)\}}}\left[
\mathbf{T}^{\mathcal{W}_{Dp+q+r}}M\rightarrow M\right]  }\overline{\Omega
}_{\left(  1\right)  }^{\left(  p+q+r,3\right)  }\left(  M\right) \\
& \underrightarrow{\zeta^{\underset{2}{\overset{\cdot}{-}}}\left(  \left[
\mathbf{T}^{\mathcal{W}_{Dp+q+r}}M\rightarrow M\right]  \right)  }\\
& \overline{\Omega}_{\left(  1\right)  }^{\left(  p+q+r,2\right)  }\left(
M\right)
\end{align*}
is denoted
\[
\zeta^{\ast_{231}\underset{2}{\overset{\cdot}{-}}\ast_{213}}%
\]

\item The composition of (\ref{5.2.6.4}) with
\begin{align*}
& \overline{\Omega}_{\left(  1\right)  }^{\left(  p+q+r,3\right)  }\left(
M\right)  \times_{\mathbf{T}^{\mathcal{W}_{D^{3}\{(1,3)\}}}\left[
\mathbf{T}^{\mathcal{W}_{Dp+q+r}}M\rightarrow M\right]  }\overline{\Omega
}_{\left(  1\right)  }^{\left(  p+q+r,3\right)  }\left(  M\right) \\
& \underrightarrow{\zeta^{\underset{2}{\overset{\cdot}{-}}}\left(  \left[
\mathbf{T}^{\mathcal{W}_{Dp+q+r}}M\rightarrow M\right]  \right)  }\\
& \overline{\Omega}_{\left(  1\right)  }^{\left(  p+q+r,2\right)  }\left(
M\right)
\end{align*}
is denoted
\[
\zeta^{\ast_{312}\underset{2}{\overset{\cdot}{-}}\ast_{132}}%
\]

\item The composition of (\ref{5.2.6.5}) with
\begin{align*}
& \overline{\Omega}_{\left(  1\right)  }^{\left(  p+q+r,3\right)  }\left(
M\right)  \times_{\mathbf{T}^{\mathcal{W}_{D^{3}\{(1,3)\}}}\left[
\mathbf{T}^{\mathcal{W}_{Dp+q+r}}M\rightarrow M\right]  }\overline{\Omega
}_{\left(  1\right)  }^{\left(  p+q+r,3\right)  }\left(  M\right) \\
& \underrightarrow{\zeta^{\underset{3}{\overset{\cdot}{-}}}\left(  \left[
\mathbf{T}^{\mathcal{W}_{Dp+q+r}}M\rightarrow M\right]  \right)  }\\
& \overline{\Omega}_{\left(  1\right)  }^{\left(  p+q+r,2\right)  }\left(
M\right)
\end{align*}
is denoted
\[
\zeta^{\ast_{312}\underset{3}{\overset{\cdot}{-}}\ast_{321}}%
\]

\item The composition of (\ref{5.2.6.6}) with
\begin{align*}
& \overline{\Omega}_{\left(  1\right)  }^{\left(  p+q+r,3\right)  }\left(
M\right)  \times_{\mathbf{T}^{\mathcal{W}_{D^{3}\{(1,3)\}}}\left[
\mathbf{T}^{\mathcal{W}_{Dp+q+r}}M\rightarrow M\right]  }\overline{\Omega
}_{\left(  1\right)  }^{\left(  p+q+r,3\right)  }\left(  M\right) \\
& \underrightarrow{\zeta^{\underset{3}{\overset{\cdot}{-}}}\left(  \left[
\mathbf{T}^{\mathcal{W}_{Dp+q+r}}M\rightarrow M\right]  \right)  }\\
& \overline{\Omega}_{\left(  1\right)  }^{\left(  p+q+r,2\right)  }\left(
M\right)
\end{align*}
is denoted
\[
\zeta^{\ast_{123}\underset{3}{\overset{\cdot}{-}}\ast_{213}}%
\]

\end{enumerate}
\end{notation}

\begin{lemma}
\label{t5.2.7}We have the following three statements:

\begin{enumerate}
\item The morphism
\begin{align*}
& \overline{\Omega}_{\left(  1\right)  }^{\left(  p,1\right)  }\left(
M\right)  \times\overline{\Omega}_{\left(  1\right)  }^{\left(  q,1\right)
}\left(  M\right)  \times\overline{\Omega}_{\left(  1\right)  }^{\left(
r,1\right)  }\left(  M\right) \\
& \underrightarrow{\left(  \zeta^{\ast_{123}\underset{1}{\overset{\cdot}{-}%
}\ast_{132}},\zeta^{\ast_{231}\underset{1}{\overset{\cdot}{-}}\ast_{321}%
}\right)  }\\
& \overline{\Omega}_{\left(  1\right)  }^{\left(  p+q+r,2\right)  }\left(
M\right)  \times\overline{\Omega}_{\left(  1\right)  }^{\left(
p+q+r,2\right)  }\left(  M\right)
\end{align*}
is to be factored through the canonical injection
\begin{align*}
& \overline{\Omega}_{\left(  1\right)  }^{\left(  p+q+r,2\right)  }\left(
M\right)  \times_{\mathbf{T}^{\mathcal{W}_{D\left(  2\right)  }}\left[
\mathbf{T}^{\mathcal{W}_{Dp+q+r}}M\rightarrow M\right]  }\overline{\Omega
}_{\left(  1\right)  }^{\left(  p+q+r,2\right)  }\left(  M\right) \\
& \rightarrow\overline{\Omega}_{\left(  1\right)  }^{\left(  p+q+r,2\right)
}\left(  M\right)  \times\overline{\Omega}_{\left(  1\right)  }^{\left(
p+q+r,2\right)  }\left(  M\right)
\end{align*}
into
\begin{align}
& \overline{\Omega}_{\left(  1\right)  }^{\left(  p,1\right)  }\left(
M\right)  \times\overline{\Omega}_{\left(  1\right)  }^{\left(  q,1\right)
}\left(  M\right)  \times\overline{\Omega}_{\left(  1\right)  }^{\left(
r,1\right)  }\left(  M\right) \nonumber\\
& \rightarrow\overline{\Omega}_{\left(  1\right)  }^{\left(  p+q+r,2\right)
}\left(  M\right)  \times_{\mathbf{T}^{\mathcal{W}_{D\left(  2\right)  }%
}\left[  \mathbf{T}^{\mathcal{W}_{Dp+q+r}}M\rightarrow M\right]  }%
\overline{\Omega}_{\left(  1\right)  }^{\left(  p+q+r,2\right)  }\left(
M\right) \label{5.2.7.1}%
\end{align}

\item The morphism
\begin{align*}
& \overline{\Omega}_{\left(  1\right)  }^{\left(  p,1\right)  }\left(
M\right)  \times\overline{\Omega}_{\left(  1\right)  }^{\left(  q,1\right)
}\left(  M\right)  \times\overline{\Omega}_{\left(  1\right)  }^{\left(
r,1\right)  }\left(  M\right) \\
& \underrightarrow{\left(  \zeta^{\ast_{231}\underset{2}{\overset{\cdot}{-}%
}\ast_{213}},\zeta^{\ast_{312}\underset{2}{\overset{\cdot}{-}}\ast_{132}%
}\right)  }\\
& \overline{\Omega}_{\left(  1\right)  }^{\left(  p+q+r,2\right)  }\left(
M\right)  \times\overline{\Omega}_{\left(  1\right)  }^{\left(
p+q+r,2\right)  }\left(  M\right)
\end{align*}
is to be factored through the canonical injection
\begin{align*}
& \overline{\Omega}_{\left(  1\right)  }^{\left(  p+q+r,2\right)  }\left(
M\right)  \times_{\mathbf{T}^{\mathcal{W}_{D\left(  2\right)  }}\left[
\mathbf{T}^{\mathcal{W}_{Dp+q+r}}M\rightarrow M\right]  }\overline{\Omega
}_{\left(  1\right)  }^{\left(  p+q+r,2\right)  }\left(  M\right) \\
& \rightarrow\overline{\Omega}_{\left(  1\right)  }^{\left(  p+q+r,2\right)
}\left(  M\right)  \times\overline{\Omega}_{\left(  1\right)  }^{\left(
p+q+r,2\right)  }\left(  M\right)
\end{align*}
into
\begin{align}
& \overline{\Omega}_{\left(  1\right)  }^{\left(  p,1\right)  }\left(
M\right)  \times\overline{\Omega}_{\left(  1\right)  }^{\left(  q,1\right)
}\left(  M\right)  \times\overline{\Omega}_{\left(  1\right)  }^{\left(
r,1\right)  }\left(  M\right) \nonumber\\
& \rightarrow\overline{\Omega}_{\left(  1\right)  }^{\left(  p+q+r,2\right)
}\left(  M\right)  \times_{\mathbf{T}^{\mathcal{W}_{D\left(  2\right)  }%
}\left[  \mathbf{T}^{\mathcal{W}_{Dp+q+r}}M\rightarrow M\right]  }%
\overline{\Omega}_{\left(  1\right)  }^{\left(  p+q+r,2\right)  }\left(
M\right) \label{5.2.7.2}%
\end{align}

\item The morphism
\begin{align*}
& \overline{\Omega}_{\left(  1\right)  }^{\left(  p,1\right)  }\left(
M\right)  \times\overline{\Omega}_{\left(  1\right)  }^{\left(  q,1\right)
}\left(  M\right)  \times\overline{\Omega}_{\left(  1\right)  }^{\left(
r,1\right)  }\left(  M\right) \\
& \underrightarrow{\left(  \zeta^{\ast_{312}\underset{3}{\overset{\cdot}{-}%
}\ast_{321}},\zeta^{\ast_{123}\underset{3}{\overset{\cdot}{-}}\ast_{213}%
}\right)  }\\
& \overline{\Omega}_{\left(  1\right)  }^{\left(  p+q+r,2\right)  }\left(
M\right)  \times\overline{\Omega}_{\left(  1\right)  }^{\left(
p+q+r,2\right)  }\left(  M\right)
\end{align*}
is to be factored through the canonical injection
\begin{align*}
& \overline{\Omega}_{\left(  1\right)  }^{\left(  p+q+r,2\right)  }\left(
M\right)  \times_{\mathbf{T}^{\mathcal{W}_{D\left(  2\right)  }}\left[
\mathbf{T}^{\mathcal{W}_{Dp+q+r}}M\rightarrow M\right]  }\overline{\Omega
}_{\left(  1\right)  }^{\left(  p+q+r,2\right)  }\left(  M\right) \\
& \rightarrow\overline{\Omega}_{\left(  1\right)  }^{\left(  p+q+r,2\right)
}\left(  M\right)  \times\overline{\Omega}_{\left(  1\right)  }^{\left(
p+q+r,2\right)  }\left(  M\right)
\end{align*}
into
\begin{align}
& \overline{\Omega}_{\left(  1\right)  }^{\left(  p,1\right)  }\left(
M\right)  \times\overline{\Omega}_{\left(  1\right)  }^{\left(  q,1\right)
}\left(  M\right)  \times\overline{\Omega}_{\left(  1\right)  }^{\left(
r,1\right)  }\left(  M\right) \nonumber\\
& \rightarrow\overline{\Omega}_{\left(  1\right)  }^{\left(  p+q+r,2\right)
}\left(  M\right)  \times_{\mathbf{T}^{\mathcal{W}_{D\left(  2\right)  }%
}\left[  \mathbf{T}^{\mathcal{W}_{Dp+q+r}}M\rightarrow M\right]  }%
\overline{\Omega}_{\left(  1\right)  }^{\left(  p+q+r,2\right)  }\left(
M\right) \label{5.2.7.3}%
\end{align}

\end{enumerate}
\end{lemma}

\begin{notation}
We introduce the following three notations:

\begin{enumerate}
\item The composition of (\ref{5.2.7.1}) with
\begin{align*}
& \overline{\Omega}_{\left(  1\right)  }^{\left(  p+q+r,2\right)  }\left(
M\right)  \times_{\mathbf{T}^{\mathcal{W}_{D\left(  2\right)  }}\left[
\mathbf{T}^{\mathcal{W}_{Dp+q+r}}M\rightarrow M\right]  }\overline{\Omega
}_{\left(  1\right)  }^{\left(  p+q+r,2\right)  }\left(  M\right) \\
& \underrightarrow{\zeta^{\overset{\cdot}{-}}\left(  \left[  \mathbf{T}%
^{\mathcal{W}_{Dp+q+r}}M\rightarrow M\right]  \right)  }\\
& \overline{\Omega}_{\left(  1\right)  }^{\left(  p+q+r,1\right)  }\left(
M\right)
\end{align*}
is denoted
\[
\zeta^{(\ast_{123}\underset{1}{\overset{\cdot}{-}}\ast_{132})\overset{\cdot
}{-}(\ast_{231}\underset{1}{\overset{\cdot}{-}}\ast_{321})}%
\]

\item The composition of (\ref{5.2.7.2}) with
\begin{align*}
& \overline{\Omega}_{\left(  1\right)  }^{\left(  p+q+r,2\right)  }\left(
M\right)  \times_{\mathbf{T}^{\mathcal{W}_{D\left(  2\right)  }}\left[
\mathbf{T}^{\mathcal{W}_{Dp+q+r}}M\rightarrow M\right]  }\overline{\Omega
}_{\left(  1\right)  }^{\left(  p+q+r,2\right)  }\left(  M\right) \\
& \underrightarrow{\zeta^{\overset{\cdot}{-}}\left(  \left[  \mathbf{T}%
^{\mathcal{W}_{Dp+q+r}}M\rightarrow M\right]  \right)  }\\
& \overline{\Omega}_{\left(  1\right)  }^{\left(  p+q+r,1\right)  }\left(
M\right)
\end{align*}
is denoted
\[
\zeta^{(\ast_{231}\underset{2}{\overset{\cdot}{-}}\ast_{213})\overset{\cdot
}{-}(\ast_{312}\underset{2}{\overset{\cdot}{-}}\ast_{132})}%
\]

\item The composition of (\ref{5.2.7.3}) with
\begin{align*}
& \overline{\Omega}_{\left(  1\right)  }^{\left(  p+q+r,2\right)  }\left(
M\right)  \times_{\mathbf{T}^{\mathcal{W}_{D\left(  2\right)  }}\left[
\mathbf{T}^{\mathcal{W}_{Dp+q+r}}M\rightarrow M\right]  }\overline{\Omega
}_{\left(  1\right)  }^{\left(  p+q+r,2\right)  }\left(  M\right) \\
& \underrightarrow{\zeta^{\overset{\cdot}{-}}\left(  \left[  \mathbf{T}%
^{\mathcal{W}_{Dp+q+r}}M\rightarrow M\right]  \right)  }\\
& \overline{\Omega}_{\left(  1\right)  }^{\left(  p+q+r,1\right)  }\left(
M\right)
\end{align*}
is denoted
\[
\zeta^{(\ast_{312}\underset{3}{\overset{\cdot}{-}}\ast_{321})\overset{\cdot
}{-}(\ast_{123}\underset{3}{\overset{\cdot}{-}}\ast_{213})}%
\]

\end{enumerate}
\end{notation}

Now we are ready to present a proof of Theorem \ref{t5.2.4}.

\begin{proof}
(of Theorem \ref{t5.2.4}). By the morphisms (\ref{5.2.4.1})-(\ref{5.2.4.3})
are identical to the morphisms
\begin{align*}
& \zeta^{(\ast_{123}\underset{1}{\overset{\cdot}{-}}\ast_{132})\overset{\cdot
}{-}(\ast_{231}\underset{1}{\overset{\cdot}{-}}\ast_{321})}\\
& \zeta^{(\ast_{231}\underset{2}{\overset{\cdot}{-}}\ast_{213})\overset{\cdot
}{-}(\ast_{312}\underset{2}{\overset{\cdot}{-}}\ast_{132})}\\
& \zeta^{(\ast_{312}\underset{3}{\overset{\cdot}{-}}\ast_{321})\overset{\cdot
}{-}(\ast_{123}\underset{3}{\overset{\cdot}{-}}\ast_{213})}%
\end{align*}
respectively. Therefore Theorem \ref{t5.2.4} follows from Theorem \ref{t3.6}.
\end{proof}

\subsection{\label{s5.3}The Second Consideration}

In this subsection we are concerned with the Lie algebra structure of
$\overline{\Omega}_{\left(  12\right)  }^{\left(  p,1\right)  }\left(
M\right)  $'s, where $p$\ ranges over natural numbers. It is easy to see that

\begin{lemma}
\label{t5.3.1}The morphism
\begin{align*}
& \overline{\Omega}_{\left(  12\right)  }^{\left(  p,1\right)  }\left(
M\right)  \times\overline{\Omega}_{\left(  12\right)  }^{\left(  q,1\right)
}\left(  M\right) \\
& \underrightarrow{i_{\overline{\Omega}_{\left(  12\right)  }^{\left(
p,1\right)  }\left(  M\right)  }^{\overline{\Omega}_{\left(  1\right)
}^{\left(  p,1\right)  }\left(  M\right)  }\times i_{\overline{\Omega
}_{\left(  12\right)  }^{\left(  q,1\right)  }\left(  M\right)  }%
^{\overline{\Omega}_{\left(  1\right)  }^{\left(  q,1\right)  }\left(
M\right)  }}\\
& \overline{\Omega}_{\left(  1\right)  }^{\left(  p,1\right)  }\left(
M\right)  \times\overline{\Omega}_{\left(  1\right)  }^{\left(  q,1\right)
}\left(  M\right) \\
& \underrightarrow{\zeta_{p,q}^{L_{1}}\left(  M\right)  }\\
& \overline{\Omega}_{\left(  1\right)  }^{\left(  p+q,1\right)  }\left(
M\right)
\end{align*}
is to be factored uniquely through the canonical injection
\[
i_{\overline{\Omega}_{\left(  12\right)  }^{\left(  p+q,1\right)  }\left(
M\right)  }^{\overline{\Omega}_{\left(  1\right)  }^{\left(  p+q,1\right)
}\left(  M\right)  }:\overline{\Omega}_{\left(  12\right)  }^{\left(
p+q,1\right)  }\left(  M\right)  \rightarrow\overline{\Omega}_{\left(
1\right)  }^{\left(  p+q,1\right)  }\left(  M\right)
\]
into a morphism
\[
\overline{\Omega}_{\left(  12\right)  }^{\left(  p,1\right)  }\left(
M\right)  \times\overline{\Omega}_{\left(  12\right)  }^{\left(  q,1\right)
}\left(  M\right)  \rightarrow\overline{\Omega}_{\left(  12\right)  }^{\left(
p+q,1\right)  }\left(  M\right)
\]

\end{lemma}

\begin{notation}
The morphism in (\ref{5.4.1.2}) is denoted
\[
\zeta_{p,q}^{L_{12}}\left(  M\right)  :\overline{\Omega}_{\left(  12\right)
}^{\left(  p,1\right)  }\left(  M\right)  \times\overline{\Omega}_{\left(
12\right)  }^{\left(  q,1\right)  }\left(  M\right)  \rightarrow
\overline{\Omega}_{\left(  12\right)  }^{\left(  p+q,1\right)  }\left(
M\right)
\]

\end{notation}

\begin{proposition}
\label{t5.3.2}The morphism
\begin{align*}
& \overline{\Omega}_{\left(  12\right)  }^{\left(  p,1\right)  }\left(
M\right)  \times\overline{\Omega}_{\left(  12\right)  }^{\left(  q,1\right)
}\left(  M\right) \\
& \underrightarrow{\zeta_{p,q}^{L_{12}}\left(  M\right)  }\\
& \overline{\Omega}_{\left(  12\right)  }^{\left(  p+q,1\right)  }\left(
M\right)
\end{align*}
and the morphism
\begin{align*}
& \overline{\Omega}_{\left(  12\right)  }^{\left(  p,1\right)  }\left(
M\right)  \times\overline{\Omega}_{\left(  12\right)  }^{\left(  q,1\right)
}\left(  M\right) \\
& =\overline{\Omega}_{\left(  12\right)  }^{\left(  q,1\right)  }\left(
M\right)  \times\overline{\Omega}_{\left(  12\right)  }^{\left(  p,1\right)
}\left(  M\right) \\
& \underrightarrow{\zeta_{q,p}^{L_{12}}\left(  M\right)  }\\
& \overline{\Omega}_{\left(  12\right)  }^{\left(  p+q,1\right)  }\left(
M\right) \\
& \underrightarrow{\left(  \cdot^{\sigma_{p,q}}\right)  _{\overline{\Omega
}_{\left(  12\right)  }^{\left(  p+q,1\right)  }\left(  M\right)  }}\\
& \overline{\Omega}_{\left(  12\right)  }^{\left(  p+q,1\right)  }\left(
M\right)
\end{align*}
sum up only to vanish.
\end{proposition}

\begin{proof}
This follows directly from Proposition \ref{t5.2.3}.
\end{proof}

\begin{theorem}
\label{t5.3.3}The three morphisms
\begin{align*}
& \overline{\Omega}_{\left(  12\right)  }^{\left(  p,1\right)  }\left(
M\right)  \times\overline{\Omega}_{\left(  12\right)  }^{\left(  q,1\right)
}\left(  M\right)  \times\overline{\Omega}_{\left(  12\right)  }^{\left(
r,1\right)  }\left(  M\right) \\
& \underrightarrow{\mathrm{id}_{\overline{\Omega}_{\left(  12\right)
}^{\left(  p,1\right)  }\left(  M\right)  }\times\zeta_{q,r}^{L_{1}}\left(
M\right)  }\\
& \overline{\Omega}_{\left(  12\right)  }^{\left(  p,1\right)  }\left(
M\right)  \times\overline{\Omega}_{\left(  12\right)  }^{\left(  q+r,1\right)
}\left(  M\right) \\
& \underrightarrow{\zeta_{p,q+r}^{L_{12}}\left(  M\right)  }\\
& \overline{\Omega}_{\left(  12\right)  }^{\left(  p+q+r,1\right)  }\left(
M\right)  \text{,}%
\end{align*}
\begin{align*}
& \overline{\Omega}_{\left(  12\right)  }^{\left(  p,1\right)  }\left(
M\right)  \times\overline{\Omega}_{\left(  12\right)  }^{\left(  q,1\right)
}\left(  M\right)  \times\overline{\Omega}_{\left(  12\right)  }^{\left(
r,1\right)  }\left(  M\right) \\
& =\overline{\Omega}_{\left(  12\right)  }^{\left(  q,1\right)  }\left(
M\right)  \times\overline{\Omega}_{\left(  12\right)  }^{\left(  r,1\right)
}\left(  M\right)  \times\overline{\Omega}_{\left(  12\right)  }^{\left(
p,1\right)  }\left(  M\right) \\
& \underrightarrow{\mathrm{id}_{\overline{\Omega}_{\left(  12\right)
}^{\left(  q,1\right)  }\left(  M\right)  }\times\zeta_{r,p}^{L_{1}}\left(
M\right)  }\\
& \overline{\Omega}_{\left(  12\right)  }^{\left(  q,1\right)  }\left(
M\right)  \times\overline{\Omega}_{\left(  12\right)  }^{\left(  r+p,1\right)
}\left(  M\right) \\
& \underrightarrow{\zeta_{q,r+p}^{L_{12}}\left(  M\right)  }\\
& \overline{\Omega}_{\left(  12\right)  }^{\left(  p+q+r,1\right)  }\left(
M\right) \\
& \underrightarrow{\left(  \cdot^{\sigma_{p,q+r}}\right)  _{\overline{\Omega
}_{\left(  12\right)  }^{\left(  p+q+r,1\right)  }\left(  M\right)  }}\\
& \overline{\Omega}_{\left(  12\right)  }^{\left(  p+q+r,1\right)  }\left(
M\right)
\end{align*}
and
\begin{align*}
& \overline{\Omega}_{\left(  12\right)  }^{\left(  p,1\right)  }\left(
M\right)  \times\overline{\Omega}_{\left(  12\right)  }^{\left(  q,1\right)
}\left(  M\right)  \times\overline{\Omega}_{\left(  12\right)  }^{\left(
r,1\right)  }\left(  M\right) \\
& =\overline{\Omega}_{\left(  12\right)  }^{\left(  r,1\right)  }\left(
M\right)  \times\overline{\Omega}_{\left(  12\right)  }^{\left(  p,1\right)
}\left(  M\right)  \times\overline{\Omega}_{\left(  12\right)  }^{\left(
q,1\right)  }\left(  M\right) \\
& \underrightarrow{\mathrm{id}_{\overline{\Omega}_{\left(  12\right)
}^{\left(  r,1\right)  }\left(  M\right)  }\times\zeta_{p,q}^{L_{1}}\left(
M\right)  }\\
& \overline{\Omega}_{\left(  12\right)  }^{\left(  r,1\right)  }\left(
M\right)  \times\overline{\Omega}_{\left(  12\right)  }^{\left(  p+q,1\right)
}\left(  M\right) \\
& \underrightarrow{\zeta_{r,p+q}^{L_{12}}\left(  M\right)  }\\
& \overline{\Omega}_{\left(  12\right)  }^{\left(  p+q+r,1\right)  }\left(
M\right) \\
& \underrightarrow{\left(  \cdot^{\sigma_{r,p+q}}\right)  _{\overline{\Omega
}_{\left(  12\right)  }^{\left(  p+q+r,1\right)  }\left(  M\right)  }}\\
& \overline{\Omega}_{\left(  12\right)  }^{\left(  p+q+r,1\right)  }\left(
M\right)
\end{align*}
sum up only to vanish.
\end{theorem}

\begin{proof}
This follows directly from Theorem \ref{t5.2.4}.
\end{proof}

\subsection{\label{s5.4}The Third Consideration}

In this subsection we are concerned with the Lie algebra structure of
$\overline{\Omega}_{\left(  13\right)  }^{\left(  p,1\right)  }\left(
M\right)  $'s, where $p$\ ranges over natural numbers.

\begin{notation}
We introduce the following notations:

\begin{enumerate}
\item We denote by
\[
\mathcal{A}^{p}:\overline{\Omega}_{\left(  1\right)  }^{\left(  p,1\right)
}\left(  M\right)  \rightarrow\overline{\Omega}_{\left(  1\right)  }^{\left(
p,1\right)  }\left(  M\right)
\]
the morphism
\begin{align*}
& \overline{\Omega}_{\left(  1\right)  }^{\left(  p,1\right)  }\left(
M\right) \\
& \underrightarrow{\sum_{\sigma\in\mathbb{S}_{p}}\varepsilon_{\sigma}\left(
\cdot^{\sigma}\right)  _{\overline{\Omega}_{\left(  1\right)  }^{\left(
p,1\right)  }\left(  M\right)  }}\\
& \overline{\Omega}_{\left(  1\right)  }^{\left(  p,1\right)  }\left(
M\right)
\end{align*}
where $\mathbb{S}_{p}$\ is the group of permutations of the set $\left\{
1,2,...,n\right\}  $.

\item We denote by
\[
\mathcal{A}_{p,q}^{p+q}:\overline{\Omega}_{\left(  1\right)  }^{\left(
p+q,1\right)  }\left(  M\right)  \rightarrow\overline{\Omega}_{\left(
1\right)  }^{\left(  p+q,1\right)  }\left(  M\right)
\]
the morphism
\begin{align*}
& \overline{\Omega}_{\left(  1\right)  }^{\left(  p+q,1\right)  }\left(
M\right) \\
& \underrightarrow{\left(  1/p!q!\right)  \mathcal{A}^{p+q}}\\
& \overline{\Omega}_{\left(  1\right)  }^{\left(  p+q,1\right)  }\left(
M\right)
\end{align*}

\item We denote by
\[
\mathcal{A}_{p,q,r}^{p+q+r}:\overline{\Omega}_{\left(  1\right)  }^{\left(
p+q+r,1\right)  }\left(  M\right)  \rightarrow\overline{\Omega}_{\left(
1\right)  }^{\left(  p+q+r,1\right)  }\left(  M\right)
\]
the morphism
\begin{align*}
& \overline{\Omega}_{\left(  1\right)  }^{\left(  p+q+r,1\right)  }\left(
M\right) \\
& \underrightarrow{\left(  1/p!q!r!\right)  \mathcal{A}^{p+q+r}}\\
& \overline{\Omega}_{\left(  1\right)  }^{\left(  p+q+r,1\right)  }\left(
M\right)
\end{align*}

\end{enumerate}
\end{notation}

It is easy to see that

\begin{lemma}
\label{t5.4.1}The morphism
\begin{align*}
& \overline{\Omega}_{\left(  13\right)  }^{\left(  p,1\right)  }\left(
M\right)  \times\overline{\Omega}_{\left(  13\right)  }^{\left(  q,1\right)
}\left(  M\right) \\
& \underrightarrow{i_{\overline{\Omega}_{\left(  13\right)  }^{\left(
p,1\right)  }\left(  M\right)  }^{\overline{\Omega}_{\left(  1\right)
}^{\left(  p,1\right)  }\left(  M\right)  }\times i_{\overline{\Omega
}_{\left(  13\right)  }^{\left(  q,1\right)  }\left(  M\right)  }%
^{\overline{\Omega}_{\left(  1\right)  }^{\left(  q,1\right)  }\left(
M\right)  }}\\
& \overline{\Omega}_{\left(  1\right)  }^{\left(  p,1\right)  }\left(
M\right)  \times\overline{\Omega}_{\left(  1\right)  }^{\left(  q,1\right)
}\left(  M\right) \\
& \underrightarrow{\zeta_{p,q}^{L_{1}}\left(  M\right)  }\\
& \overline{\Omega}_{\left(  1\right)  }^{\left(  p+q,1\right)  }\left(
M\right) \\
& \underrightarrow{\mathcal{A}_{p,q}^{p+q}}\\
& \overline{\Omega}_{\left(  1\right)  }^{\left(  p+q,1\right)  }\left(
M\right)
\end{align*}
is to be factored uniquely through the canonical injection
\begin{equation}
i_{\overline{\Omega}_{\left(  13\right)  }^{\left(  p+q,1\right)  }\left(
M\right)  }^{\overline{\Omega}_{\left(  1\right)  }^{\left(  p+q,1\right)
}\left(  M\right)  }:\overline{\Omega}_{\left(  13\right)  }^{\left(
p+q,1\right)  }\left(  M\right)  \rightarrow\overline{\Omega}_{\left(
1\right)  }^{\left(  p+q,1\right)  }\left(  M\right) \label{5.4.1.1}%
\end{equation}
into a morphism
\begin{equation}
\overline{\Omega}_{\left(  13\right)  }^{\left(  p,1\right)  }\left(
M\right)  \times\overline{\Omega}_{\left(  13\right)  }^{\left(  q,1\right)
}\left(  M\right)  \rightarrow\overline{\Omega}_{\left(  13\right)  }^{\left(
p+q,1\right)  }\left(  M\right) \label{5.4.1.2}%
\end{equation}

\end{lemma}

\begin{notation}
The morphism in (\ref{5.4.1.2}) is denoted
\[
\zeta_{p,q}^{FN_{13}}\left(  M\right)  :\overline{\Omega}_{\left(  13\right)
}^{\left(  p,1\right)  }\left(  M\right)  \times\overline{\Omega}_{\left(
13\right)  }^{\left(  q,1\right)  }\left(  M\right)  \rightarrow
\overline{\Omega}_{\left(  13\right)  }^{\left(  p+q,1\right)  }\left(
M\right)
\]
where $FN$\ stands for Fr\"{o}licher and Nijenhuis.
\end{notation}

\begin{lemma}
\label{t5.4.2}Given $\sigma\in\mathbb{S}_{p}$, the morphism
\begin{align*}
& \overline{\Omega}_{\left(  1\right)  }^{\left(  p,1\right)  }\left(
M\right) \\
& \underrightarrow{\mathcal{A}^{p}}\\
& \overline{\Omega}_{\left(  1\right)  }^{\left(  p,1\right)  }\left(
M\right) \\
& \underrightarrow{\left(  \cdot^{\sigma}\right)  _{\overline{\Omega}_{\left(
1\right)  }^{\left(  p,1\right)  }\left(  M\right)  }}\\
& \overline{\Omega}_{\left(  1\right)  }^{\left(  p,1\right)  }\left(
M\right)
\end{align*}
is identical to the morphism
\begin{align*}
& \overline{\Omega}_{\left(  1\right)  }^{\left(  p,1\right)  }\left(
M\right) \\
& \underrightarrow{\left(  \cdot^{\sigma}\right)  _{\overline{\Omega}_{\left(
1\right)  }^{\left(  p,1\right)  }\left(  M\right)  }}\\
& \overline{\Omega}_{\left(  1\right)  }^{\left(  p,1\right)  }\left(
M\right) \\
& \underrightarrow{\mathcal{A}^{p}}\\
& \overline{\Omega}_{\left(  1\right)  }^{\left(  p,1\right)  }\left(
M\right)
\end{align*}
Both of them are identical to the morphism
\begin{align*}
& \overline{\Omega}_{\left(  1\right)  }^{\left(  p,1\right)  }\left(
M\right) \\
& \underrightarrow{\mathcal{A}^{p}}\\
& \overline{\Omega}_{\left(  1\right)  }^{\left(  p,1\right)  }\left(
M\right) \\
& \underrightarrow{\varepsilon_{\sigma}}\\
& \overline{\Omega}_{\left(  1\right)  }^{\left(  p,1\right)  }\left(
M\right)
\end{align*}

\end{lemma}

\begin{proof}
By mimicking the familiar token in establishing the antisymmetry of wedge
products of differential forms in orthodox differential geometry.
\end{proof}

\begin{proposition}
\label{t5.4.3}The morphisms
\begin{align}
& \overline{\Omega}_{\left(  13\right)  }^{\left(  p,1\right)  }\left(
M\right)  \times\overline{\Omega}_{\left(  13\right)  }^{\left(  q,1\right)
}\left(  M\right) \nonumber\\
& \underrightarrow{\zeta_{p,q}^{FN_{13}}\left(  M\right)  }\nonumber\\
& \overline{\Omega}_{\left(  13\right)  }^{\left(  p+q,1\right)  }\left(
M\right) \label{5.4.3.1}%
\end{align}
and
\begin{align}
& \overline{\Omega}_{\left(  13\right)  }^{\left(  p,1\right)  }\left(
M\right)  \times\overline{\Omega}_{\left(  13\right)  }^{\left(  q,1\right)
}\left(  M\right) \nonumber\\
& =\overline{\Omega}_{\left(  13\right)  }^{\left(  q,1\right)  }\left(
M\right)  \times\overline{\Omega}_{\left(  13\right)  }^{\left(  p,1\right)
}\left(  M\right) \nonumber\\
& \underrightarrow{\zeta_{q,p}^{FN_{13}}\left(  M\right)  }\nonumber\\
& \overline{\Omega}_{\left(  13\right)  }^{\left(  p+q,1\right)  }\left(
M\right) \nonumber\\
& \underrightarrow{\left(  -1\right)  ^{pq}}\nonumber\\
& \overline{\Omega}_{\left(  13\right)  }^{\left(  p+q,1\right)  }\left(
M\right) \label{5.4.3.2}%
\end{align}
sum up only to vanish.
\end{proposition}

\begin{proof}
The morphism (\ref{5.4.3.1}) followed by the canonical injection
(\ref{5.4.1.1}) is identical to the morphism
\begin{align}
& \overline{\Omega}_{\left(  13\right)  }^{\left(  p,1\right)  }\left(
M\right)  \times\overline{\Omega}_{\left(  13\right)  }^{\left(  q,1\right)
}\left(  M\right) \nonumber\\
& \underrightarrow{i_{\overline{\Omega}_{\left(  13\right)  }^{\left(
p,1\right)  }\left(  M\right)  }^{\overline{\Omega}_{\left(  1\right)
}^{\left(  p,1\right)  }\left(  M\right)  }\times i_{\overline{\Omega
}_{\left(  13\right)  }^{\left(  q,1\right)  }\left(  M\right)  }%
^{\overline{\Omega}_{\left(  1\right)  }^{\left(  q,1\right)  }\left(
M\right)  }}\nonumber\\
& \overline{\Omega}_{\left(  1\right)  }^{\left(  p,1\right)  }\left(
M\right)  \times\overline{\Omega}_{\left(  1\right)  }^{\left(  q,1\right)
}\left(  M\right) \nonumber\\
& \underrightarrow{\zeta_{p,q}^{L_{1}}\left(  M\right)  }\nonumber\\
& \overline{\Omega}_{\left(  1\right)  }^{\left(  p+q,1\right)  }\left(
M\right) \nonumber\\
& \underrightarrow{\mathcal{A}_{p,q}^{p+q}}\nonumber\\
& \overline{\Omega}_{\left(  1\right)  }^{\left(  p+q,1\right)  }\left(
M\right) \label{5.4.3.3}%
\end{align}
by definition, while the morphism (\ref{5.4.3.2}) followed by the canonical
injection (\ref{5.4.1.1}) is identical to the morphism
\begin{align}
& \overline{\Omega}_{\left(  13\right)  }^{\left(  p,1\right)  }\left(
M\right)  \times\overline{\Omega}_{\left(  13\right)  }^{\left(  q,1\right)
}\left(  M\right) \nonumber\\
& \underrightarrow{i_{\overline{\Omega}_{\left(  13\right)  }^{\left(
p,1\right)  }\left(  M\right)  }^{\overline{\Omega}_{\left(  1\right)
}^{\left(  p,1\right)  }\left(  M\right)  }\times i_{\overline{\Omega
}_{\left(  13\right)  }^{\left(  q,1\right)  }\left(  M\right)  }%
^{\overline{\Omega}_{\left(  1\right)  }^{\left(  q,1\right)  }\left(
M\right)  }}\nonumber\\
& \overline{\Omega}_{\left(  1\right)  }^{\left(  p,1\right)  }\left(
M\right)  \times\overline{\Omega}_{\left(  1\right)  }^{\left(  q,1\right)
}\left(  M\right) \nonumber\\
& \underrightarrow{\zeta_{p,q}^{L_{1}}\left(  M\right)  }\nonumber\\
& \overline{\Omega}_{\left(  1\right)  }^{\left(  p+q,1\right)  }\left(
M\right) \nonumber\\
& \underrightarrow{\left(  \cdot^{\sigma_{p,q}}\right)  _{\overline{\Omega
}_{\left(  1\right)  }^{\left(  p+q,1\right)  }\left(  M\right)  }}\nonumber\\
& \overline{\Omega}_{\left(  1\right)  }^{\left(  p+q,1\right)  }\left(
M\right) \nonumber\\
& \underrightarrow{\mathcal{A}_{p,q}^{p+q}}\nonumber\\
& \overline{\Omega}_{\left(  1\right)  }^{\left(  p+q,1\right)  }\left(
M\right) \label{5.4.3.4}%
\end{align}
by Lemma \ref{t5.4.2}. The sum of (\ref{5.4.3.3}) and (\ref{5.4.3.4}) vanishes
by dint of Proposition \ref{t5.2.3}, so that the sum of (\ref{5.4.3.1}) and
(\ref{5.4.3.2}) also vanishes.
\end{proof}

\begin{lemma}
\label{t5.4.4}The morphism
\begin{align*}
& \overline{\Omega}_{\left(  13\right)  }^{\left(  p,1\right)  }\left(
M\right)  \times\overline{\Omega}_{\left(  13\right)  }^{\left(  q,1\right)
}\left(  M\right)  \times\overline{\Omega}_{\left(  13\right)  }^{\left(
r,1\right)  }\left(  M\right) \\
& \underrightarrow{i_{\overline{\Omega}_{\left(  13\right)  }^{\left(
p,1\right)  }\left(  M\right)  }^{\overline{\Omega}_{\left(  1\right)
}^{\left(  p,1\right)  }\left(  M\right)  }\times i_{\overline{\Omega
}_{\left(  13\right)  }^{\left(  q,1\right)  }\left(  M\right)  }%
^{\overline{\Omega}_{\left(  1\right)  }^{\left(  q,1\right)  }\left(
M\right)  }\times i_{\overline{\Omega}_{\left(  13\right)  }^{\left(
r,1\right)  }\left(  M\right)  }^{\overline{\Omega}_{\left(  1\right)
}^{\left(  r,1\right)  }\left(  M\right)  }}\\
& \overline{\Omega}_{\left(  1\right)  }^{\left(  p,1\right)  }\left(
M\right)  \times\overline{\Omega}_{\left(  1\right)  }^{\left(  q,1\right)
}\left(  M\right)  \times\overline{\Omega}_{\left(  1\right)  }^{\left(
r,1\right)  }\left(  M\right) \\
& \underrightarrow{\mathrm{id}_{\overline{\Omega}_{\left(  1\right)
}^{\left(  p,1\right)  }\left(  M\right)  }\times\zeta_{q.r}^{L_{1}}\left(
M\right)  }\\
& \overline{\Omega}_{\left(  1\right)  }^{\left(  p,1\right)  }\left(
M\right)  \times\overline{\Omega}_{\left(  1\right)  }^{\left(  q+r,1\right)
}\left(  M\right) \\
& \underrightarrow{\mathrm{id}_{\overline{\Omega}_{\left(  1\right)
}^{\left(  p,1\right)  }\left(  M\right)  }\times\mathcal{A}_{q,r}^{q+r}}\\
& \overline{\Omega}_{\left(  1\right)  }^{\left(  p,1\right)  }\left(
M\right)  \times\overline{\Omega}_{\left(  1\right)  }^{\left(  q+r,1\right)
}\left(  M\right) \\
& \underrightarrow{\zeta_{p,q+r}^{L_{1}}\left(  M\right)  }\\
& \overline{\Omega}_{\left(  1\right)  }^{\left(  p+q+r,1\right)  }\left(
M\right) \\
& \underrightarrow{\mathcal{A}_{p,q+r}^{p+q+r}}\\
& \overline{\Omega}_{\left(  1\right)  }^{\left(  p+q+r,1\right)  }\left(
M\right)
\end{align*}
is identical to the morphism
\begin{align*}
& \overline{\Omega}_{\left(  13\right)  }^{\left(  p,1\right)  }\left(
M\right)  \times\overline{\Omega}_{\left(  13\right)  }^{\left(  q,1\right)
}\left(  M\right)  \times\overline{\Omega}_{\left(  13\right)  }^{\left(
r,1\right)  }\left(  M\right) \\
& \underrightarrow{i_{\overline{\Omega}_{\left(  13\right)  }^{\left(
p,1\right)  }\left(  M\right)  }^{\overline{\Omega}_{\left(  1\right)
}^{\left(  p,1\right)  }\left(  M\right)  }\times i_{\overline{\Omega
}_{\left(  13\right)  }^{\left(  q,1\right)  }\left(  M\right)  }%
^{\overline{\Omega}_{\left(  1\right)  }^{\left(  q,1\right)  }\left(
M\right)  }\times i_{\overline{\Omega}_{\left(  13\right)  }^{\left(
r,1\right)  }\left(  M\right)  }^{\overline{\Omega}_{\left(  1\right)
}^{\left(  r,1\right)  }\left(  M\right)  }}\\
& \overline{\Omega}_{\left(  1\right)  }^{\left(  p,1\right)  }\left(
M\right)  \times\overline{\Omega}_{\left(  1\right)  }^{\left(  q,1\right)
}\left(  M\right)  \times\overline{\Omega}_{\left(  1\right)  }^{\left(
r,1\right)  }\left(  M\right) \\
& \underrightarrow{\mathrm{id}_{\overline{\Omega}_{\left(  1\right)
}^{\left(  p,1\right)  }\left(  M\right)  }\times\zeta_{q.r}^{L_{1}}\left(
M\right)  }\\
& \overline{\Omega}_{\left(  1\right)  }^{\left(  p,1\right)  }\left(
M\right)  \times\overline{\Omega}_{\left(  1\right)  }^{\left(  q+r,1\right)
}\left(  M\right) \\
& \underrightarrow{\zeta_{p,q+r}^{L_{1}}\left(  M\right)  }\\
& \overline{\Omega}_{\left(  1\right)  }^{\left(  p+q+r,1\right)  }\left(
M\right) \\
& \underrightarrow{\mathcal{A}_{p,q,r}^{p+q+r}}\\
& \overline{\Omega}_{\left(  1\right)  }^{\left(  p+q+r,1\right)  }\left(
M\right)
\end{align*}

\end{lemma}

\begin{proof}
By mimicking the familiar token in establishing the associativity of wedge
products of differential forms in orthodox differential geometry.
\end{proof}

\begin{theorem}
\label{t5.4.5}The three morphisms
\begin{align}
& \overline{\Omega}_{\left(  13\right)  }^{\left(  p,1\right)  }\left(
M\right)  \times\overline{\Omega}_{\left(  13\right)  }^{\left(  q,1\right)
}\left(  M\right)  \times\overline{\Omega}_{\left(  13\right)  }^{\left(
r,1\right)  }\left(  M\right) \nonumber\\
& \underrightarrow{\mathrm{id}_{\overline{\Omega}_{\left(  13\right)
}^{\left(  p,1\right)  }\left(  M\right)  }\times\zeta_{q,r}^{FN_{13}}\left(
M\right)  }\nonumber\\
& \overline{\Omega}_{\left(  13\right)  }^{\left(  p,1\right)  }\left(
M\right)  \times\overline{\Omega}_{\left(  13\right)  }^{\left(  q+r,1\right)
}\left(  M\right) \nonumber\\
& \underrightarrow{\zeta_{p,q+r}^{FN_{13}}\left(  M\right)  }\nonumber\\
& \overline{\Omega}_{\left(  13\right)  }^{\left(  p+q+r,1\right)  }\left(
M\right) \label{5.4.5.1}%
\end{align}
\begin{align}
& \overline{\Omega}_{\left(  13\right)  }^{\left(  p,1\right)  }\left(
M\right)  \times\overline{\Omega}_{\left(  13\right)  }^{\left(  q,1\right)
}\left(  M\right)  \times\overline{\Omega}_{\left(  13\right)  }^{\left(
r,1\right)  }\left(  M\right) \nonumber\\
& =\overline{\Omega}_{\left(  13\right)  }^{\left(  q,1\right)  }\left(
M\right)  \times\overline{\Omega}_{\left(  13\right)  }^{\left(  r,1\right)
}\left(  M\right)  \times\overline{\Omega}_{\left(  13\right)  }^{\left(
p,1\right)  }\left(  M\right) \nonumber\\
& \underrightarrow{\mathrm{id}_{\overline{\Omega}_{\left(  13\right)
}^{\left(  q,1\right)  }\left(  M\right)  }\times\zeta_{r,p}^{FN_{13}}\left(
M\right)  }\nonumber\\
& \overline{\Omega}_{\left(  13\right)  }^{\left(  q,1\right)  }\left(
M\right)  \times\overline{\Omega}_{\left(  13\right)  }^{\left(  p+r,1\right)
}\left(  M\right) \nonumber\\
& \underrightarrow{\zeta_{q,p+r}^{FN_{13}}\left(  M\right)  }\nonumber\\
& \overline{\Omega}_{\left(  13\right)  }^{\left(  p+q+r,1\right)  }\left(
M\right) \nonumber\\
& \underrightarrow{\left(  -1\right)  ^{p\left(  q+r\right)  }}\nonumber\\
& \overline{\Omega}_{\left(  13\right)  }^{\left(  p+q+r,1\right)  }\left(
M\right) \label{5.4.5.2}%
\end{align}
and
\begin{align}
& \overline{\Omega}_{\left(  13\right)  }^{\left(  p,1\right)  }\left(
M\right)  \times\overline{\Omega}_{\left(  13\right)  }^{\left(  q,1\right)
}\left(  M\right)  \times\overline{\Omega}_{\left(  13\right)  }^{\left(
r,1\right)  }\left(  M\right) \nonumber\\
& =\overline{\Omega}_{\left(  13\right)  }^{\left(  r,1\right)  }\left(
M\right)  \times\overline{\Omega}_{\left(  13\right)  }^{\left(  p,1\right)
}\left(  M\right)  \times\overline{\Omega}_{\left(  13\right)  }^{\left(
q,1\right)  }\left(  M\right) \nonumber\\
& \underrightarrow{\mathrm{id}_{\overline{\Omega}_{\left(  13\right)
}^{\left(  r,1\right)  }\left(  M\right)  }\times\zeta_{p,q}^{FN_{13}}\left(
M\right)  }\nonumber\\
& \overline{\Omega}_{\left(  13\right)  }^{\left(  r,1\right)  }\left(
M\right)  \times\overline{\Omega}_{\left(  13\right)  }^{\left(  p+q,1\right)
}\left(  M\right) \nonumber\\
& \underrightarrow{\zeta_{r,p+q}^{FN_{13}}\left(  M\right)  }\nonumber\\
& \overline{\Omega}_{\left(  13\right)  }^{\left(  p+q+r,1\right)  }\left(
M\right) \nonumber\\
& \underrightarrow{\left(  -1\right)  ^{r\left(  p+q\right)  }}\nonumber\\
& \overline{\Omega}_{\left(  13\right)  }^{\left(  p+q+r,1\right)  }\left(
M\right) \label{5.4.5.3}%
\end{align}
sum up only to vanish.
\end{theorem}

\begin{proof}
The morphism (\ref{5.4.5.1}) followed by the canonical injection
(\ref{5.4.1.1}) is identical to the morphism
\begin{align}
& \overline{\Omega}_{\left(  13\right)  }^{\left(  p,1\right)  }\left(
M\right)  \times\overline{\Omega}_{\left(  13\right)  }^{\left(  q,1\right)
}\left(  M\right)  \times\overline{\Omega}_{\left(  13\right)  }^{\left(
r,1\right)  }\left(  M\right) \nonumber\\
& \underrightarrow{i_{\overline{\Omega}_{\left(  13\right)  }^{\left(
p,1\right)  }\left(  M\right)  }^{\overline{\Omega}_{\left(  1\right)
}^{\left(  p,1\right)  }\left(  M\right)  }\times i_{\overline{\Omega
}_{\left(  13\right)  }^{\left(  q,1\right)  }\left(  M\right)  }%
^{\overline{\Omega}_{\left(  1\right)  }^{\left(  q,1\right)  }\left(
M\right)  }\times i_{\overline{\Omega}_{\left(  13\right)  }^{\left(
r,1\right)  }\left(  M\right)  }^{\overline{\Omega}_{\left(  1\right)
}^{\left(  r,1\right)  }\left(  M\right)  }}\nonumber\\
& \overline{\Omega}_{\left(  1\right)  }^{\left(  p,1\right)  }\left(
M\right)  \times\overline{\Omega}_{\left(  1\right)  }^{\left(  q,1\right)
}\left(  M\right)  \times\overline{\Omega}_{\left(  1\right)  }^{\left(
r,1\right)  }\left(  M\right) \nonumber\\
& \underrightarrow{\mathrm{id}_{\overline{\Omega}_{\left(  1\right)
}^{\left(  p,1\right)  }\left(  M\right)  }\times\zeta_{q.r}^{L_{1}}\left(
M\right)  }\nonumber\\
& \overline{\Omega}_{\left(  1\right)  }^{\left(  p,1\right)  }\left(
M\right)  \times\overline{\Omega}_{\left(  1\right)  }^{\left(  q+r,1\right)
}\left(  M\right) \nonumber\\
& \underrightarrow{\zeta_{p,q+r}^{L_{1}}\left(  M\right)  }\nonumber\\
& \overline{\Omega}_{\left(  1\right)  }^{\left(  p+q+r,1\right)  }\left(
M\right) \nonumber\\
& \underrightarrow{\mathcal{A}_{p,q,r}^{p+q+r}}\nonumber\\
& \overline{\Omega}_{\left(  1\right)  }^{\left(  p+q+r,1\right)  }\left(
M\right) \label{5.4.5.4}%
\end{align}
by dint of Lemma \ref{t5.4.4}. The morphism (\ref{5.4.5.2}) followed by the
canonical injection (\ref{5.4.1.1}) is identical to the morphism
\begin{align}
& \overline{\Omega}_{\left(  13\right)  }^{\left(  p,1\right)  }\left(
M\right)  \times\overline{\Omega}_{\left(  13\right)  }^{\left(  q,1\right)
}\left(  M\right)  \times\overline{\Omega}_{\left(  13\right)  }^{\left(
r,1\right)  }\left(  M\right) \nonumber\\
& \underrightarrow{i_{\overline{\Omega}_{\left(  13\right)  }^{\left(
p,1\right)  }\left(  M\right)  }^{\overline{\Omega}_{\left(  1\right)
}^{\left(  p,1\right)  }\left(  M\right)  }\times i_{\overline{\Omega
}_{\left(  13\right)  }^{\left(  q,1\right)  }\left(  M\right)  }%
^{\overline{\Omega}_{\left(  1\right)  }^{\left(  q,1\right)  }\left(
M\right)  }\times i_{\overline{\Omega}_{\left(  13\right)  }^{\left(
r,1\right)  }\left(  M\right)  }^{\overline{\Omega}_{\left(  1\right)
}^{\left(  r,1\right)  }\left(  M\right)  }}\nonumber\\
& \overline{\Omega}_{\left(  1\right)  }^{\left(  p,1\right)  }\left(
M\right)  \times\overline{\Omega}_{\left(  1\right)  }^{\left(  q,1\right)
}\left(  M\right)  \times\overline{\Omega}_{\left(  1\right)  }^{\left(
r,1\right)  }\left(  M\right) \nonumber\\
& =\overline{\Omega}_{\left(  1\right)  }^{\left(  q,1\right)  }\left(
M\right)  \times\overline{\Omega}_{\left(  1\right)  }^{\left(  r,1\right)
}\left(  M\right)  \times\overline{\Omega}_{\left(  1\right)  }^{\left(
p,1\right)  }\left(  M\right) \nonumber\\
& \underrightarrow{\mathrm{id}_{\overline{\Omega}_{\left(  1\right)
}^{\left(  q,1\right)  }\left(  M\right)  }\times\zeta_{r.p}^{L_{1}}\left(
M\right)  }\nonumber\\
& \overline{\Omega}_{\left(  1\right)  }^{\left(  q,1\right)  }\left(
M\right)  \times\overline{\Omega}_{\left(  1\right)  }^{\left(  p+r,1\right)
}\left(  M\right) \nonumber\\
& \underrightarrow{\zeta_{q,p+r}^{L_{1}}\left(  M\right)  }\nonumber\\
& \overline{\Omega}_{\left(  1\right)  }^{\left(  p+q+r,1\right)  }\left(
M\right) \nonumber\\
& \underrightarrow{\left(  \cdot^{\sigma_{p,q+r}}\right)  _{\overline{\Omega
}_{\left(  1\right)  }^{\left(  p+q+r,1\right)  }\left(  M\right)  }%
}\nonumber\\
& \overline{\Omega}_{\left(  1\right)  }^{\left(  p+q+r,1\right)  }\left(
M\right) \nonumber\\
& \underrightarrow{\mathcal{A}_{p,q,r}^{p+q+r}}\nonumber\\
& \overline{\Omega}_{\left(  1\right)  }^{\left(  p+q+r,1\right)  }\left(
M\right) \label{5.4.5.5}%
\end{align}
by dint of Lemmas \ref{t5.4.2} and \ref{t5.4.4} with $\varepsilon
_{\sigma_{p,q+r}}=(-1)^{p\left(  q+r\right)  }$. The morphism (\ref{5.4.5.3})
followed by the canonical injection (\ref{5.4.1.1}) is identical to the
morphism
\begin{align}
& \overline{\Omega}_{\left(  13\right)  }^{\left(  p,1\right)  }\left(
M\right)  \times\overline{\Omega}_{\left(  13\right)  }^{\left(  q,1\right)
}\left(  M\right)  \times\overline{\Omega}_{\left(  13\right)  }^{\left(
r,1\right)  }\left(  M\right) \nonumber\\
& \underrightarrow{i_{\overline{\Omega}_{\left(  13\right)  }^{\left(
p,1\right)  }\left(  M\right)  }^{\overline{\Omega}_{\left(  1\right)
}^{\left(  p,1\right)  }\left(  M\right)  }\times i_{\overline{\Omega
}_{\left(  13\right)  }^{\left(  q,1\right)  }\left(  M\right)  }%
^{\overline{\Omega}_{\left(  1\right)  }^{\left(  q,1\right)  }\left(
M\right)  }\times i_{\overline{\Omega}_{\left(  13\right)  }^{\left(
r,1\right)  }\left(  M\right)  }^{\overline{\Omega}_{\left(  1\right)
}^{\left(  r,1\right)  }\left(  M\right)  }}\nonumber\\
& \overline{\Omega}_{\left(  1\right)  }^{\left(  p,1\right)  }\left(
M\right)  \times\overline{\Omega}_{\left(  1\right)  }^{\left(  q,1\right)
}\left(  M\right)  \times\overline{\Omega}_{\left(  1\right)  }^{\left(
r,1\right)  }\left(  M\right) \nonumber\\
& =\overline{\Omega}_{\left(  1\right)  }^{\left(  r,1\right)  }\left(
M\right)  \times\overline{\Omega}_{\left(  1\right)  }^{\left(  p,1\right)
}\left(  M\right)  \times\overline{\Omega}_{\left(  1\right)  }^{\left(
q,1\right)  }\left(  M\right) \nonumber\\
& \underrightarrow{\mathrm{id}_{\overline{\Omega}_{\left(  1\right)
}^{\left(  r,1\right)  }\left(  M\right)  }\times\zeta_{p,q}^{L_{1}}\left(
M\right)  }\nonumber\\
& \overline{\Omega}_{\left(  1\right)  }^{\left(  r,1\right)  }\left(
M\right)  \times\overline{\Omega}_{\left(  1\right)  }^{\left(  p+q,1\right)
}\left(  M\right) \nonumber\\
& \underrightarrow{\zeta_{q,p+r}^{L_{1}}\left(  M\right)  }\nonumber\\
& \overline{\Omega}_{\left(  1\right)  }^{\left(  p+q+r,1\right)  }\left(
M\right) \nonumber\\
& \underrightarrow{\left(  \cdot^{\sigma_{r,p+q}}\right)  _{\overline{\Omega
}_{\left(  1\right)  }^{\left(  p+q+r,1\right)  }\left(  M\right)  }%
}\nonumber\\
& \overline{\Omega}_{\left(  1\right)  }^{\left(  p+q+r,1\right)  }\left(
M\right) \nonumber\\
& \underrightarrow{\mathcal{A}_{p,q,r}^{p+q+r}}\nonumber\\
& \overline{\Omega}_{\left(  1\right)  }^{\left(  p+q+r,1\right)  }\left(
M\right) \label{5.4.5.6}%
\end{align}
by dint of Lemmas \ref{t5.4.2} and \ref{t5.4.4} with $\varepsilon
_{\sigma_{r,p+q}}=(-1)^{r\left(  p+q\right)  }$. The sum of the three
morphisms (\ref{5.4.5.4})-(\ref{5.4.5.6}) vanishes thanks to Theorem
\ref{t5.2.4}, so that the sum of the three morphisms (\ref{5.4.5.1}%
)-(\ref{5.4.5.3}) also vanishes.
\end{proof}

\subsection{\label{s5.5}The Fourth Consideration}

In this subsection we are concerned with the Lie algebra structure of
$\overline{\Omega}_{\left(  123\right)  }^{\left(  p,1\right)  }\left(
M\right)  $'s, where $p$\ ranges over natural numbers. It is easy to see that

\begin{lemma}
\label{t5.5.1}The morphism
\begin{align*}
& \overline{\Omega}_{\left(  123\right)  }^{\left(  p,1\right)  }\left(
M\right)  \times\overline{\Omega}_{\left(  123\right)  }^{\left(  q,1\right)
}\left(  M\right) \\
& \underrightarrow{i_{\overline{\Omega}_{\left(  123\right)  }^{\left(
p,1\right)  }\left(  M\right)  }^{\overline{\Omega}_{\left(  13\right)
}^{\left(  p,1\right)  }\left(  M\right)  }\times i_{\overline{\Omega
}_{\left(  123\right)  }^{\left(  q,1\right)  }\left(  M\right)  }%
^{\overline{\Omega}_{\left(  13\right)  }^{\left(  q,1\right)  }\left(
M\right)  }}\\
& \overline{\Omega}_{\left(  13\right)  }^{\left(  p,1\right)  }\left(
M\right)  \times\overline{\Omega}_{\left(  13\right)  }^{\left(  q,1\right)
}\left(  M\right) \\
& \underrightarrow{\zeta_{p,q}^{FN_{13}}\left(  M\right)  }\\
& \overline{\Omega}_{\left(  13\right)  }^{\left(  p+q,1\right)  }\left(
M\right)
\end{align*}
is to be factored uniquely through the canonical injection
\[
i_{\overline{\Omega}_{\left(  123\right)  }^{\left(  p+q,1\right)  }\left(
M\right)  }^{\overline{\Omega}_{\left(  13\right)  }^{\left(  p+q,1\right)
}\left(  M\right)  }:\overline{\Omega}_{\left(  123\right)  }^{\left(
p+q,1\right)  }\left(  M\right)  \rightarrow\overline{\Omega}_{\left(
13\right)  }^{\left(  p+q,1\right)  }\left(  M\right)
\]
into a morphism
\begin{equation}
\overline{\Omega}_{\left(  123\right)  }^{\left(  p,1\right)  }\left(
M\right)  \times\overline{\Omega}_{\left(  123\right)  }^{\left(  q,1\right)
}\left(  M\right)  \rightarrow\overline{\Omega}_{\left(  123\right)
}^{\left(  p+q,1\right)  }\left(  M\right) \label{5.5.1.1}%
\end{equation}

\end{lemma}

\begin{notation}
The morphism (\ref{5.5.1.1}) is denoted
\[
\zeta_{p,q}^{FN_{123}}\left(  M\right)  :\overline{\Omega}_{\left(
123\right)  }^{\left(  p,1\right)  }\left(  M\right)  \times\overline{\Omega
}_{\left(  123\right)  }^{\left(  q,1\right)  }\left(  M\right)
\rightarrow\overline{\Omega}_{\left(  123\right)  }^{\left(  p+q,1\right)
}\left(  M\right)
\]

\end{notation}

\begin{proposition}
\label{t5.5.2}The morphisms
\begin{align*}
& \overline{\Omega}_{\left(  123\right)  }^{\left(  p,1\right)  }\left(
M\right)  \times\overline{\Omega}_{\left(  123\right)  }^{\left(  q,1\right)
}\left(  M\right) \\
& \underrightarrow{\zeta_{p,q}^{FN_{123}}\left(  M\right)  }\\
& \overline{\Omega}_{\left(  123\right)  }^{\left(  p+q,1\right)  }\left(
M\right)
\end{align*}
and
\begin{align*}
& \overline{\Omega}_{\left(  123\right)  }^{\left(  p,1\right)  }\left(
M\right)  \times\overline{\Omega}_{\left(  123\right)  }^{\left(  q,1\right)
}\left(  M\right) \\
& =\overline{\Omega}_{\left(  123\right)  }^{\left(  q,1\right)  }\left(
M\right)  \times\overline{\Omega}_{\left(  123\right)  }^{\left(  p,1\right)
}\left(  M\right) \\
& \underrightarrow{\zeta_{q,p}^{FN_{123}}\left(  M\right)  }\\
& \overline{\Omega}_{\left(  123\right)  }^{\left(  p+q,1\right)  }\left(
M\right) \\
& \underrightarrow{\left(  -1\right)  ^{pq}}\\
& \overline{\Omega}_{\left(  123\right)  }^{\left(  p+q,1\right)  }\left(
M\right)
\end{align*}
sum up only to vanish.
\end{proposition}

\begin{proof}
This follows directly from Proposition \ref{t5.4.3}.
\end{proof}

\begin{theorem}
\label{t5.5.3}The three morphisms
\begin{align*}
& \overline{\Omega}_{\left(  123\right)  }^{\left(  p,1\right)  }\left(
M\right)  \times\overline{\Omega}_{\left(  123\right)  }^{\left(  q,1\right)
}\left(  M\right)  \times\overline{\Omega}_{\left(  123\right)  }^{\left(
r,1\right)  }\left(  M\right) \\
& \underrightarrow{\mathrm{id}_{\overline{\Omega}_{\left(  123\right)
}^{\left(  p,1\right)  }\left(  M\right)  }\times\zeta_{q,r}^{FN_{123}}\left(
M\right)  }\\
& \overline{\Omega}_{\left(  123\right)  }^{\left(  p,1\right)  }\left(
M\right)  \times\overline{\Omega}_{\left(  123\right)  }^{\left(
q+r,1\right)  }\left(  M\right) \\
& \underrightarrow{\zeta_{p,q+r}^{FN_{123}}\left(  M\right)  }\\
& \overline{\Omega}_{\left(  123\right)  }^{\left(  p+q+r,1\right)  }\left(
M\right)
\end{align*}
\begin{align*}
& \overline{\Omega}_{\left(  123\right)  }^{\left(  p,1\right)  }\left(
M\right)  \times\overline{\Omega}_{\left(  123\right)  }^{\left(  q,1\right)
}\left(  M\right)  \times\overline{\Omega}_{\left(  123\right)  }^{\left(
r,1\right)  }\left(  M\right) \\
& =\overline{\Omega}_{\left(  123\right)  }^{\left(  q,1\right)  }\left(
M\right)  \times\overline{\Omega}_{\left(  123\right)  }^{\left(  r,1\right)
}\left(  M\right)  \times\overline{\Omega}_{\left(  123\right)  }^{\left(
p,1\right)  }\left(  M\right) \\
& \underrightarrow{\mathrm{id}_{\overline{\Omega}_{\left(  123\right)
}^{\left(  q,1\right)  }\left(  M\right)  }\times\zeta_{r,p}^{FN_{123}}\left(
M\right)  }\\
& \overline{\Omega}_{\left(  123\right)  }^{\left(  q,1\right)  }\left(
M\right)  \times\overline{\Omega}_{\left(  123\right)  }^{\left(
p+r,1\right)  }\left(  M\right) \\
& \underrightarrow{\zeta_{q,p+r}^{FN_{123}}\left(  M\right)  }\\
& \overline{\Omega}_{\left(  123\right)  }^{\left(  p+q+r,1\right)  }\left(
M\right) \\
& \underrightarrow{\left(  -1\right)  ^{p\left(  q+r\right)  }}\\
& \overline{\Omega}_{\left(  123\right)  }^{\left(  p+q+r,1\right)  }\left(
M\right)
\end{align*}
and
\begin{align*}
& \overline{\Omega}_{\left(  123\right)  }^{\left(  p,1\right)  }\left(
M\right)  \times\overline{\Omega}_{\left(  123\right)  }^{\left(  q,1\right)
}\left(  M\right)  \times\overline{\Omega}_{\left(  123\right)  }^{\left(
r,1\right)  }\left(  M\right) \\
& =\overline{\Omega}_{\left(  123\right)  }^{\left(  r,1\right)  }\left(
M\right)  \times\overline{\Omega}_{\left(  123\right)  }^{\left(  p,1\right)
}\left(  M\right)  \times\overline{\Omega}_{\left(  123\right)  }^{\left(
q,1\right)  }\left(  M\right) \\
& \underrightarrow{\mathrm{id}_{\overline{\Omega}_{\left(  123\right)
}^{\left(  r,1\right)  }\left(  M\right)  }\times\zeta_{p,q}^{FN_{123}}\left(
M\right)  }\\
& \overline{\Omega}_{\left(  123\right)  }^{\left(  r,1\right)  }\left(
M\right)  \times\overline{\Omega}_{\left(  123\right)  }^{\left(
p+q,1\right)  }\left(  M\right) \\
& \underrightarrow{\zeta_{r,p+q}^{FN_{123}}\left(  M\right)  }\\
& \overline{\Omega}_{\left(  123\right)  }^{\left(  p+q+r,1\right)  }\left(
M\right) \\
& \underrightarrow{\left(  -1\right)  ^{r\left(  p+q\right)  }}\\
& \overline{\Omega}_{\left(  123\right)  }^{\left(  p+q+r,1\right)  }\left(
M\right)
\end{align*}
sum up only to vanish.
\end{theorem}

\begin{proof}
This follows directly from Theorem \ref{t5.4.5}.
\end{proof}

\end{document}